\documentclass[oneside,english]{amsart}
\usepackage[T1]{fontenc}
\usepackage[utf8]{inputenc}
\usepackage{amsmath, amssymb, amsthm, amsfonts}
\usepackage{mathrsfs}
\usepackage{xifthen}
\usepackage{graphicx}
\usepackage{listofitems}
\usepackage{xparse}
\usepackage{enumitem}
\usepackage{bbm}
\usepackage{physics}
\usepackage{mathtools}
\usepackage{hyperref}
\hypersetup{
    colorlinks=true,
    linkcolor=blue,
    filecolor=magenta,      
    urlcolor=blue,
    pdftitle={Overleaf Example},
    pdfpagemode=FullScreen,
}
\usepackage{biblatex}
\usepackage{tikz}
\usetikzlibrary{calc}
\usepackage{multicol}
\usepackage{array}
\usepackage{cases}
\usepackage{setspace}
\usepackage{babel}
\usepackage{csquotes}

\let\oldfrac\frac
\renewcommand{\frac}[2]{%
  \mathchoice
    {\oldfrac{#1}{#2}}
    {#1/#2}
    {\oldfrac{#1}{#2}}
    {\oldfrac{#1}{#2}}
}

\theoremstyle{definition}
\newtheorem{definition}{Definition}[section]

\theoremstyle{plain}
\newtheorem{theorem}{Theorem}[section]
\newtheorem{corollary}{Corollary}[section]
\newtheorem{proposition}{Proposition}[section]
\newtheorem{lemma}{Lemma}[section]
\newtheorem{remark}{Remark}[section]

\newcommand{\C}{\mathbb{C}}
\newcommand{\E}{\mathbb{E}}
\newcommand{\F}{\mathcal{F}}

\newcommand{\N}{\mathbb{N}}
\renewcommand{\P}{\mathbb{P}}

\newcommand{\R}{\mathbb{R}}

\newcommand{\p}[1]{\left(#1\right)}
\renewcommand{\b}[1]{\left[#1\right]}
\newcommand{\set}[1]{\left\{#1\right\}}

\NewDocumentCommand{\limit}{m O{\infty}}{\lim_{#1 \to #2}}

\newcommand{\besselI}[2]{I_{#1}\p{#2}}
\newcommand{\besselJ}[2]{J_{#1}\p{#2}}

\newcommand{\inner}[2]{\left\langle#1,#2\right\rangle}

\let\oldliminf\liminf
\renewcommand{\liminf}[1]{\oldliminf_{#1 \to \infty}}

\let\oldlimsup\limsup
\renewcommand{\limsup}[1]{\oldlimsup_{#1 \to \infty}}

\newcommand{\ind}[1]{\mathbbm{1}_{#1}}

\NewDocumentCommand{\deriv}{O{} m O{}}{\oldfrac{d^{#3}#1}{d#2^{#3}}}

\catcode`,\active

\catcode`\,12

\catcode`,\active

\catcode`\,12

\catcode`,\active

\catcode`\,12

\bibstyle{plain}
\numberwithin{equation}{section}
\numberwithin{figure}{section}

\title{Stochastic Kimura Equations}

\date{}

\makeatother

\begin{document}

\begin{abstract}

In this work we study the one-dimensional stochastic Kimura equation $\partial_{t}u\left(z,t\right)=z\partial_{z}^{2}u\left(z,t\right)+u\left(z,t\right)\dot{W}\left(z,t\right)$ for $z,t>0$ equipped with a Dirichlet boundary condition at $0$, with $\dot{W}$ being a Gaussian space-time noise. 
This equation can be seen as a degenerate analog of the parabolic Anderson model. 
We combine the Wiener chaos theory from Malliavin calculus, the Duhamel perturbation technique from PDEs, and the kernel analysis of (deterministic) degenerate diffusion equations to develop a solution theory for the stochastic Kimura equation. We establish results on existence, uniqueness, moments, and continuity for the solution $u\left(z,t\right)$. 
In particular, we investigate how the stochastic potential and the degeneracy in the diffusion operator jointly affect the properties of $u\left(z,t\right)$ near the boundary. 
We also derive explicit estimates on the comparison under the $L^{2}-$ norm between $u\left(z,t\right)$ and its deterministic counterpart for $\left(z,t\right)$ within a proper range. 

\end{abstract}

\thanks{Part of this work is supported by the NSERC Discovery Grant of the
second author (No. 241023).}
\author{Roland Riachi$^\dagger$}
\thanks{$\dagger$ roland.riachi@mail.mcgill.ca}
\author{Linan Chen$^{\ddagger}$}
\thanks{$\ddagger$linan.chen@mcgill.ca}
\maketitle
\allowdisplaybreaks
\tableofcontents
\newpage

\section{Introduction}  \label{sec:intro}
In this work we study a class of one-dimensional stochastic degenerate diffusion equations, which we call \emph{stochastic Kimura equations}, that are diffusion equations on $\R_+$ with the diffusion coefficient degenerating linearly at the boundary $0$ and subject to a Gaussian noise potential on $\R_+$.
This work is a marriage between the study of degenerate diffusion, which is a classical topic with a long-celebrated history, and the techniques from stochastic partial differential equations (SPDEs), which is a fast developing field with many recent breakthroughs. 
By combining analytic methods and stochastic analysis, we derive both qualitative (e.g., existence and uniqueness) and quantitative (e.g., moments and regularity) results on the solution to the stochastic Kimura equation.
We put a particular emphasis on the short-term behavior of the solution near the boundary, dissecting the joint effect of the degenerate diffusion operator and the space-time noise potential. 
To the best of our knowledge, these aspects of stochastic degenerate diffusion equations have not been previously studied.

\subsection{Kimura Diffusion and Parabolic Anderson Model}
\label{ssec:KimuraPAM}

In this subsection we give a coarse review of the two fields, degenerate diffusion equation and SPDEs, at the intersection of which our work lies. 
Both fields have rich and vast literature. 
Our reviews will be limited as we will only mention the results that are most relevant to the problem concerned in this work. 
In particular, the two pillars based on which our work is carried out are the \emph{Kimura diffusion}, a classical example of degenerate diffusion, and the \emph{parabolic Anderson model}, a well studied SPDE model. 
We now give a brief introduction to both models, but leave the technical details for the next section. 

\subsubsection*{1.1.1. Kimura Diffusion}
\label{sssec:Kimura}

The Kimura diffusion is an important mathematical tool in population genetics as a model of the propagation of a certain genotype among a population \cite{Kimura}. 
In the classical setting, it is described by the stochastic differential equation (SDE) $\mathrm d Z_{t}=\sqrt{2Z_{t}}\mathrm dB_{t}+b\left(Z_{t}\right)\mathrm dt$ with $Z_{0}=z>0$, $b$ being a drift coefficient and $0$ being an absorbing boundary. 
Correspondingly, the \emph{Kimura equation} is the partial differential equation (PDE) 
\begin{equation}
\partial_{t}u\left(z,t\right)=z\partial_{z}^{2}u\left(z,t\right)+b\left(z\right)\partial_{z}u\left(z,t\right) , \label{eqn:KimuraEqn}
\end{equation}
for $z,t>0$, equipped with the Dirichlet boundary condition at $0$.
Characterized by the degeneracy in diffusion, this pair of equations has been of great interest to analysts and probabilists. 
Properties of the solutions (e.g., existence, uniqueness, regularity, estimation,
etc.) have been extensively studied for the Kimura diffusion, as well as certain of its generalizations, in various contexts \cite{Ethier76,measure_value_proc_popul_gen,wellposed_mart_deg_diff_dynamic_population,EM11,AthreyaBarlowBassPerkins02}.
Both analytic and probabilistic tool-kits have been developed
to treat degenerate diffusion operators, replacing classical techniques
that are restricted to strict elliptic/parabolic operators. 

The specific Kimura diffusion that will be examined later under a stochastic potential is the model with $b\equiv0$, but the methods developed in this study are extendable to cases with general $b$.
Let $q_{0}\left(z,w,t\right)$ denote the fundamental solution to
(\ref{eqn:KimuraEqn}) with $b\equiv0$, explicitly determined in terms of Bessel function \cite{wfeq_Epstein_Mazzeo,chenWF,chenDEG}:
\begin{equation}
    q_{0}\left(z,w,t\right):= \sqrt{ \frac{z}{w} } \frac{e^{-\frac{z+w}{t}}}{t}I_1\p{\frac{2 \sqrt{z w}}{t}} \label{eqn:q0Formula}
\end{equation}
for every $z,w,t>0$. 
That is, if the Kimura equation is equipped with the initial data $u_{0}\left(z\right)$, then the solution is given by 
\begin{equation}
    \label{eqn:detKimuraEqnSolution}
    u_{0}\left(z,t\right) = \int_{0}^{\infty}u_{0}\left(w\right)q_{0}\left(z,w,t\right)dw,
\end{equation}
for every $z,t>0$. 
We will refer to \ref{eqn:detKimuraEqnSolution} as the \emph{deterministic solution}, which, as we will see later, serves as the starting point in the construction of the solution to the stochastic Kimura equation. 

In latter sections, a perturbation of the Kimura equation will also
be considered:
\[
    \partial_{t}u\left(z,t\right) = z\partial_{z}^{2}u\left(z,t\right)+u\left(z,t\right)V\left(z, t\right),
\]
where $V$ is a potential function whose fundamental solution will be denoted by $q_0^{V}\left(z,w,t\right)$. When $V$ is bounded and time-independent, a sharp regularity theory on $q_0^{V}\left(z,w,t\right)$ is developed in \cite{chenDEG} via the Duhamel perturbation method, which produces accurate estimates on the comparison between $q_0^{V}\left(z,w,t\right)$ and $q_{0}\left(z,w,t\right)$ near the boundary for small times. 
Details of these results will be presented in $\mathsection\ref{ssec:bgKim}$.

\subsubsection*{1.1.2. Parabolic Anderson Model}
\label{sssec:PAM}

SPDEs have become a novel and popular tool in modeling phenomena observed
in scientific experiments in random environments. 
Following the seminal papers \cite{dalang, nualart}, many classical deterministic PDEs have been re-visited when the domain is subject to certain random noise, and a fast-growing set of theories is being developed for solutions to such SPDEs \cite{dalang, nualart, balanHAM, hairer2023global, HAUSENBLAS20204174, Motyl2013, Berger2023}.
Most notably, the parabolic Anderson model (PAM) is the classical heat equation perturbed by a Gaussian random potential.
The one-dimensional model can be written as 
\begin{equation}
    \partial_{t}u\left(z,t\right) = \frac{1}{2}\partial_{z}^{2}u\left(z,t\right)+u\left(z,t\right)\dot{W}\left(z,t\right), \label{eqn:PAM}
\end{equation}
for $z \in \R$ and $t>0$, where $\dot{W} \left(z,t\right)$ denotes a Gaussian space-time noise on some filtered probability space $\left(\Omega,\mathcal{F},\mathbb{P}\right)$. 
Formally speaking, $\{ \dot{W}\left(z,t\right) \mid z\in\mathbb{R},t>0\} $ can be viewed as a centered Gaussian family with covariance 
\begin{equation}
    \mathbb{E}\left[\dot{W}\left(z_1,t_1\right)\dot{W}\left(z_2,t_2\right)\right] = f\left(z_1-z_2\right)\gamma\left(t_1-t_2\right),\label{eqn:formalNoise}
\end{equation}
for $z_1,z_2\in\mathbb{R}$ and $t_1,t_2>0$, where $f:\mathbb{R}\rightarrow\mathbb{R}^{+}$ and $\gamma:\R\rightarrow\mathbb{R}^{+}$ are two non-negative definite, symmetric, and locally integrable kernels. 
To solve (\ref{eqn:PAM}) with (deterministic) initial data $u_{0}\left(z\right)$, we seek a process $\left\{ u\left(z,t\right) \mid z \in \R, t \geq 0\right\} \subseteq L^{2}\left(\Omega\right)$, to which we refer as a \emph{(mild) solution} to (\ref{eqn:PAM}), such that for every $z\in\mathbb{R}$ and $t>0$,
\[
    u\left(z,t\right) = \int_{\mathbb{R}} g_{t}\left(z,w\right) u_{0}\left(w\right) dw + \int_{0}^{t} \int_{\mathbb{R}} g_{t-\tau}\left(z,w\right) u\left(w,\tau\right) W\left(dw,d\tau\right)\text{ a.s.},
\]
where $g_{t}$$\left(z,w\right)= e^{-\left|z-w\right|^{2}/2} / \sqrt{2\pi t}$ is the standard heat kernel, and the stochastic integral with ``respect to $W\left(dw,d\tau\right)$'' is defined as a generalized It\^o integral, following the theory of Malliavin calculus. 
In $\mathsection\ref{ssec:bgMalCal}$, we review all the necessary elements from Malliavin calculus, including the construction and the properties of such stochastic integrals. 

Common approaches toward the study of PAM combine Malliavin calculus with classical heat kernel analysis, yielding abundant results on the solution \cite{dalang, nualart, BalanFCN, balanCONT, balanRIC, PAMCarmonaMolchanov}.
Alternatively, another approach is the theory of rough paths developed by Hairer \cite{HairerRPT, PAMHairerLabbe}.
The PAM has also been studied in other settings, e.g., on $\mathbb{Z}^{d}$ \cite{gaertner2004}, on the hypercube \cite{AVENA20203369}, and on Heisenberg groups \cite{BAUDOIN2023109920}, as well as subject to L\'evy random noise \cite{Berger2023}.
For degenerate parabolic SPDEs, the Cauchy problem has been studied in the semilinear and quasilinear cases \cite{HOFMANOVA20134294,debussche:hal-00863829}.

\subsection{Stochastic Kimura Equation and Main Results}
\label{ssec:SKEandResults}

Inspired by the works reviewed in the previous subsection, we decide to bring the two models, the Kimura diffusion and the PAM, together and investigate the following stochastic Kimura equation:
\begin{equation}
\label{eqn:SKE1}
\begin{cases}   
    &\partial_{t}u\left(z,t\right) = z\partial_{z}^{2}u\left(z,t\right) + u\left(z,t\right)\dot{W}\left(z,t\right) \quad\quad z, t > 0,\\
    &\hspace{3.2mm}u(z, 0) = u_0(z) \hfill z > 0,\\
    &\hspace{3.8mm}u(0, t) = 0 \hfill t > 0.
\end{cases}
\end{equation}
where $\dot{W}$ is the same as in (\ref{eqn:formalNoise}). 
Similarly as in the PAM case, a (mild) solution to (\ref{eqn:SKE1}) is a process $\left\{ u\left(z,t\right) \mid z, t \geq 0 \right\} \subseteq L^{2}\left(\Omega\right)$
such that for every $z,t>0$,
\begin{align}
    u\left(z,t\right) = \int_0^\infty q_{0}\left(z,w,t\right) & u_{0}\left(w\right)dw \notag \\
    & + \int_{0}^{t}\int_0^\infty q_{0}\left(z,w,t-\tau\right) u\left(w,\tau\right )W\left(dw,d\tau\right) \text{ a.s.} \label{eqn:SKE1Soln}
\end{align}

The purpose of studying (\ref{eqn:SKE1}) has multiple folds. 
First, we aim to develop a solution theory for stochastic degenerate diffusion equations, which would yield applications in modeling diffusion processes in random mediums with absorbing boundaries, e.g., genotype propagation with random hibernation/revival.
The model of the stochastic Kimura equation is a good starting point since the deterministic equation is well studied with abundant results on the fundamental solution $q_{0}\left(z,w,t\right)$. 
Second, we want to extend the existing literature on the PAM to the settings of degenerate diffusions and investigate how the current methods in constructing the solution to the PAM, as well as the properties of the solution, may change in the presence of degeneracy in the diffusion. 
Furthermore, we are interested in studying the interaction between the degenerate diffusion operator and the stochastic potential in order to analyze their ``joint effect'' on the solution; to this end, we will pay particular attention to the behavior of the solution near the boundary $0$, and compare the solution to the deterministic solution in a proper sense.

In $\mathsection\ref{sec:eu}$ we prove that (\ref{eqn:SKE1}) admits a solution $\left\{ u\left(z,t\right) \mid z,t \geq 0 \right\}$ that satisfies (\ref{eqn:SKE1Soln}) and that this solution is unique up to modification (Theorems \ref{thm:euSTWN} and \ref{thm:euColored}). 
We adopt a strategy that is similar to the one invoked in \cite{balanHAM, nualart}, which is to combine the Wiener chaos expansion from the Malliavin calculus and the Duhamel method from PDEs. However, the techniques we apply are substantially different from those utilized in the PAM
case, because (1) although $q_{0}\left(z,w,t\right)$ is explicitly determined, unlike $g_{t}\left(z,w\right)$, $q_{0}\left(z,w,t\right)$ does not have a closed-form expression, nor certain desirable features such as symmetry; and (2) since our model is restricted on $\R_+$, some convenient technical elements based on Fourier analysis are not applicable in our case. 
We establish our result in two steps: first we solve (\ref{eqn:SKE1}) when $\dot{W}$ is the
space-time white noise, i.e., when both $f$ and $\gamma$ in (\ref{eqn:formalNoise}) are Dirac delta $\delta_0$, and next extend the construction of the solution to the colored noise setting.

Taking a close look at the behaviors of $u\left(z,t\right)$, we observe that the impact of the stochastic potential is most significant near the boundary, which makes the behavior of $u\left(z,t\right)$ for small $z$ fundamentally different from that of the solution to the deterministic Kimura equation $u_{0}\left(z,t\right)$.
To better understand the role of $\dot{W}$ near the boundary, in $\mathsection\ref{sec:ratio}$ we also introduce a degenerate coefficient to $\dot{W}$ and consider the following stochastic Kimura equation with a degenerate noise.
That is, for some $\beta > 0$
\begin{equation}
    \partial_{t}u\left(z,t\right) = z\partial_{z}^{2}u\left(z,t\right) + u\left(z,t\right) \hat z^\beta \dot{W}\left(z,t\right), \label{eqn:SKE2intro}
\end{equation}
where $\hat z := z \wedge 1$.
Similarly, (\ref{eqn:SKE2intro}) also admits a unique solution $\left\{ u\left(z,t\right) \mid z,t \geq 0\right\}$. 
It turns out that so long as $\beta>0$, the degeneracy in the stochastic potential will ``tame'' $u\left(z,t\right)$ enough so that it remains comparable to $u_{0}\left(z,t\right)$ under the $L^{2}-$ norm for some finite time $t$ and for $z$ all the way to the boundary. 

To be specific, we prove in Theorem \ref{thm:ratioBdGen} that, when $u_{0}\equiv1$ and $f$ satisfies some additional assumptions, there exists $T \in (0,1)$ such that
\begin{equation}
    \label{eqn:ratioSpace}
    \sup_{\substack{z > 0 \\ t \in [0,T]}} \mathbb{E}\left[\left(\frac{u\left(z,t\right)}{u_{0}\left(z,t\right)}\right)^{2}\right] < \infty.
\end{equation}
In addition, we also obtain that at every fixed $z>0$, 
\begin{equation}
    \label{eqn:ratioTime}
    \lim_{t\searrow0} \mathbb{E}\left[\left(\frac{u\left(z,t\right)}{u_{0}\left(z,t\right)} - 1\right)^{2}\right] = 0,
\end{equation}
i.e., $u\left(z,t\right)$ stays close to $u_{0}\left(z,t\right)$, again, in the $L^{2}-$ sense, for sufficiently small $t$, which is consistent with (3.10) in \cite{chenDEG}, a result established for the deterministic equation but with a bounded potential.

In $\mathsection\ref{sec:cont}$ we examine the regularity of $u\left(z,t\right)$, particularly its continuity in $z$ at the boundary. 
We discover that, in order to acquire boundary continuity, we need to impose further assumptions on the rate of degeneracy of the stochastic potential.
Indeed, when $\beta > 1/4$ we are able to prove a H\"older continuity result on $u\left(z,t\right)$ in $z$ all the way to the boundary in Theorem \ref{thm:contColored}. 
Along the way, in $\mathsection\ref{sec:moment}$ we also obtain results on $L^{p}-$ integrability for all $p\geq2$, as well as higher moment estimates, for $u\left(z,t\right)$ for all $z,t>0$. 

The aforementioned results are only part of our study of the stochastic Kimura equation as the first step in an attempt to advance our general knowledge on stochastic degenerate diffusion equations. 
There are other aspects of the work that are still ongoing and will not be addressed
here. 
For example, our methods can be extended to study the following two variations of (\ref{eqn:SKE1}) without substantial difficulty: (1) with an order $\alpha$ of degeneracy in the diffusion operator for $\alpha$ from some proper range, i.e., 
\[
    \partial_{t} u\left(z,t\right) = z^{\alpha}\partial_{z}^{2}u\left(z,t\right) + u\left(z,t\right) \dot{W}\left(z,t\right),
\]
and (2) with a drift coefficient $b$ satisfying some proper conditions, i.e., 
\[
    \partial_{t} u\left(z,t\right) = z\partial_{z}^{2} u\left(z,t\right) + b\left(z\right)\partial_{z} u\left(z,t\right) + u\left(z,t\right) \dot{W}\left(z,t\right).
\]
In addition, the framework developed in this paper also allows us to study long-term asymptotics of $u\left(z,t\right)$ and potentially establish results on intermittency. 

\subsection{Contents and Notations}
\label{ssec:outline}

In $\mathsection\ref{sec:bg}$ we review the necessary elements from the Malliavin calculus and technical results from the study of the deterministic Kimura equation.
In $\mathsection \ref{sec:results}$, we set up the model of the stochastic Kimura equation then proceed to study various aspects of the solution and prove results on existence, uniqueness, estimates, and continuity of the solution.
$\mathsection\ref{sec:conclusion}$ outlines some related problems and future directions.
Useful formulas on Bessel functions are deposited in $\mathsection\ref{sec:appendix}$.

We write $\R_+$ to mean the interval $(0, \infty)$.
For $a, b \in \R$, we write $a \vee b := \max\set{a, b}$ and $a \wedge b := \min\set{a, b}$.
For $z > 0$, we denote $\hat z := z \wedge 1$.
We denote by $\mathcal D(\R_+^2)$ the set of all smooth, compactly supported functions defined on $\R_+^2$.
We denote by $C_b(\R_+)$ the space of all bounded continuous functions defined on $\R_+$, equipped with the uniform norm 
\[
    \|f\|_u := \sup_{x > 0} |f(x)|.
\]
For $\alpha \in \C$, we denote by $I_\alpha$ and $J_\alpha$ the (modified) Bessel functions of the first kind of order $\alpha$.
For $q \geq p \geq 0$, $a_1, \ldots, a_p, b_1, \ldots, b_q \in \R$, and $x \in \C$, we denote by
\[
    {}_pF_q
    \left[
    \hspace{-1mm}
    \begin{array}{cc}
        a_1, \ldots, a_p  \\
        b_1, \ldots, b_q 
    \end{array} \hspace{-2mm}
    \mathrel{\Big|} x
    \right]
\]
the generalized Hypergeometric function.
The definitions of these special functions are reviewed in $\mathsection \ref{sec:appendix}$.

\section{Background} \label{sec:bg} 


\subsection{The Malliavin Calculus}
\label{ssec:bgMalCal}

In this section, we present some prerequisite material regarding the Malliavin calculus which is essential for discussing the model of the stochastic Kimura equation.
For details, we refer the reader to \cite{introMalCal, minicourse, advSA, nualartBook}.

Let $f$ and $\gamma$ be as in (\ref{eqn:formalNoise}), and consider the inner product
\begin{equation}
    \label{eqn:innerProd}
    \inner{\varphi}{\psi}_H := \int_{\R_+^4} f(x-y)\gamma(r-s)\varphi(x,r)\psi(y,s)\mathrm dx\mathrm dy\mathrm dr\mathrm ds,
\end{equation}
for $\varphi, \psi \in \mathcal D(\R_+^2)$.
We define the Hilbert space $(H, \inner{\cdot}{\cdot}_H)$ as the completion of $\mathcal D(\R_+^2)$ with respect to $\inner{\cdot}{\cdot}_{H}$.
It can be shown that there exists an isonormal Gaussian process $W = \set{W(h) \mid h \in H}$, on an appropriate probability space $(\Omega, \F, \P)$, such that
\begin{align*}
    \E[W(h)] \equiv 0 && \E[W(h_1)W(h_2)] = \inner{h_1}{h_2}_H.
\end{align*}

When viewed as a map between Hilbert spaces, $W : H \to W(H) =: \mathcal H \subseteq L^2(\Omega)$, is an isometry.
This framework is convenient because it allows us to study the noise process $W$ through well-behaved deterministic functions.

As it turns out, if $\F$ is the $\sigma$-algebra generated by $W$, then $L^2(\Omega)$ admits a decomposition into countably many closed images of $W$ under a special class of polynomials.
Let $n \geq 1$ and $H_n$ denote the $n$th Hermite polynomial.
The \emph{$n$th Wiener chaos space} $\mathcal H_n$ is defined as
\begin{equation}
    \label{eqn:wienerChaosSpaceDefn}
    \mathcal H_n := \overline{ \text{span}\set{H_n(W(h)) \mid h \in H, \|h\|_H = 1} },
\end{equation}
where the closure is taken in $L^2(\Omega)$.
Then,
\begin{equation}
    \label{eqn:l2Decomposition}
    L^2(\Omega) = \bigoplus_{n=0}^\infty \mathcal H_n,
\end{equation}
i.e., for any $X \in L^2(\Omega)$, for each $n \geq 0$ there exists a unique $X_n \in \mathcal H_n$ such that
\begin{equation}
    \label{eqn:wienerChaosDecomposition1}
    X = \sum_{n=0}^\infty X_n,
\end{equation}
where the series converges in $L^2(\Omega)$.
We remark that since $H_0 \equiv 1$, $\mathcal H_0 \cong \R$ is a space of constants and it can be shown that $X_0 = \E[X]$.
By \cite{nualart} (see the last line of page 62), for every $n \geq 1$ and $p \geq 2$, $\mathcal H_n \subseteq L^p(\Omega)$.
In fact, if $Y \in \mathcal H_n$, then 
\begin{equation}
    \label{eqn:lpInequality}
    \E[|Y|^p]^{\frac{1}{p}} \leq (p-1)^{\frac{n}{2}}\E[Y^2]^{\frac{1}{2}}.
\end{equation}

Below, we will demonstrate representations for these projections $X_n$ in terms of stochastic integrals in the sense of Skorohod, whose moments are related to the inner product $\inner{\cdot}{\cdot}_H$.

To this end, we first define a suitable notion of derivative.
Let $C_p^\infty (\R^N)$ denote the space of infinitely differentiable functions $F : \R^N\to \R$ which (along with their partial derivatives) have at most polynomial growth.
We say a random variable $X \in L^2(\Omega)$ is \emph{smooth} if there exists $F \in C_p^\infty (\R^N)$ and $h_1, \ldots, h_N \in H$ such that $X = F(W(h_1), \ldots, W(h_N))$.
Moreover, we denote by $\mathcal P \subseteq L^2(\Omega)$ the set of all such smooth random variables.

Consequently, for some $X = F(W(h_1), \ldots, W(h_N)) \in \mathcal P$, the \emph{Malliavin derivative} is an unbounded, closed linear operator $D : \mathcal P \to L^2(\Omega; H)$ given by
\begin{equation}
    \label{eqn:malDerivDefn}
    DX := \sum_{j=1}^N \partial_j F(W(h_1), \ldots, W(h_N)) h_j,
\end{equation}
where for $z, t > 0$,
\[
    D_{z, t}X = \sum_{j=1}^N \partial_j F(W(h_1), \ldots, W(h_N)) h_j(z, t),
\]
and without loss of generality, we may assume that $h_1, \ldots, h_N$ are orthonormal under $\inner{\cdot}{\cdot}_H$.

As with classical Sobolev spaces, we define the inner product
\begin{equation}
    \label{eqn:stochasticSobolevDefn}
    \inner{X}{Y}_{1,2} := \E[XY] + \E[\inner{DX}{DY}_H]
\end{equation}
for $X, Y \in \mathcal P$, and let $\mathbb D^{1,2}$ denote the closure of $\mathcal P$ with respect to $\inner{\cdot}{\cdot}_{1,2}$.
Note that $\mathcal P \subseteq \mathbb D^{1,2} \subseteq L^2(\Omega)$, and since $\mathcal P$ is dense in $L^2(\Omega)$, $\mathcal P$ is dense in $\mathbb D^{1,2}$ and hence $D$ can be extended to $\mathbb D^{1,2}$.

Viewing $D : \mathbb D^{1,2} \to L^2(\Omega; H)$, we are ready to introduce the notion of stochastic integral relevant to (\ref{eqn:SKE1}).
Let the \emph{divergence operator} $\delta$ be the adjoint operator of $D$.
Namely, the domain $\text{Dom}(\delta)$ of $\delta$ is the set of random variables $u \in L^2(\Omega; H)$ for which there exists $c_u > 0 $ such that,
\[
    |\E[\inner{DX}{u}_{H}]| \leq c_u \|X\|_{1,2},
\]
for all $X \in \mathbb D^{1,2}$; in this case, $\delta(u)$ is the unique element of $L^2(\Omega)$ such that
\begin{equation}
    \label{ref:adjointDefn}
    \E[\delta(u)X] = \E[\inner{DX}{u}_H].
\end{equation}
for all $X \in \mathbb D^{1,2}$.
As the adjoint of an unbounded and densely defined operator, $\delta$ is in turn also closed.
Furthermore, we use the notation
\begin{equation}
    \label{eqn:skorohodInt}
    \delta(u) := \int_{\R_+^2} u(z, t) W(\mathrm dz, \mathrm dt),
\end{equation}
where $W(\mathrm dz, \mathrm dt) = \dot W(z, t) \mathrm dz \mathrm dt$ (formally, $\dot W(z,t) = \partial_t\partial_z W(z,t)$).
This nomenclature is motivated by the fact that the divergence operator possesses many desirable properties of a typical integral and coincides with the It\^o stochastic integral in certain cases.
For this reason, we also refer to $\delta$ as the \emph{Skorohod integral}, and if $u \in \text{Dom}(\delta)$, we say that $u$ is \emph{Skorohod integrable}.

Continuing along these lines, we define the multiple Skorohod integral to arrive at representations for the random variables $X_n$ in (\ref{eqn:wienerChaosDecomposition1}).
For $n \geq 1$, let $H^{\otimes n}$ denote the $n$th tensor product of $H$, with the corresponding inner product and norm denoted by $\inner{\cdot}{\cdot}_n$ and $\|\cdot\|_n$, respectively.
Then for $f_n \in H^{\otimes n}$, iteratively applying the Skorohod integral yields
\begin{equation}
    \label{eqn:multipleSkorohodInt}
    \mathcal I_n(u) := \int_{\R_+^{2n}} f_n(z_1, t_1, \ldots, z_n, t_n) W(\mathrm dz_1, \mathrm dt_1) \cdots W(\mathrm dz_n, \mathrm dt_n).
\end{equation}

Given $f_n \in H^{\otimes n}$ and $g_m \in H^{\otimes m}$, we denote by $\tilde f_n$ the symmetrization of $f_n$, given by
\begin{equation}
    \label{eqn:symmetrizationDefn}
    \tilde f_n(z_1, t_1, \ldots, z_n, t_n) := \frac{1}{n!} \sum_{\sigma \in S_n} f(z_{\sigma(1)}, t_{\sigma(1)}, \ldots, z_{\sigma(n)}, t_{\sigma(n)}),
\end{equation}
where $S_n$ is the set of all permutations on $n$ elements.
Then $\mathcal I_n(f_n) = \mathcal I_n(\tilde f_n)$ and 
\begin{equation}
    \label{eqn:skorohodIntIsometry}
    \E[\mathcal I_n(f_n)\mathcal I_m(g_m)] = 
    \begin{cases}
        n! \langle\tilde f_n,\tilde g_m\rangle_n &\text{ if } n = m\\
        0 &\text{ otherwise }
    \end{cases}.
\end{equation}

Conversely, for every $n \geq 1$, $X_n \in \mathcal H_n$ and there exists a unique $f_n \in \R_+^{2n} \to \R$ symmetric such that $X_n = \mathcal I_n(f_n)$.
In light of this, (\ref{eqn:wienerChaosDecomposition1}) becomes
\begin{equation}
    \label{eqn:wienerChaosDecomposition2}
    X = \sum_{n=0}^\infty \mathcal I_n(f_n),
\end{equation}
where $\mathcal I_0(f_0) = \E[X]$.
Similarly, if $u \in L^2(\Omega; H)$, then for every $z, t> 0$, there exists a unique $f_n(\cdot; z, t) \in H^{\otimes n}$ symmetric in it first $n$ pairs of variables such that
\begin{equation}
    \label{eqn:wienerChaosDecomposition3}
    u(z, t) = \sum_{n=0}^\infty \mathcal I_n(f_n(\cdot; z, t)),
\end{equation}
and $\mathcal I_0(f_0(z, t)) = \E[u(z,t)]$.


We close this section by presenting the following equivalent criterion for Skorohod integrability.

\begin{proposition}[Proposition 1.3.7 of \cite{nualartBook}]
    \label{prop:swapSeriesAndDiv}
    Let $u \in L^2(\Omega; H)$ with the expansion (\ref{eqn:wienerChaosDecomposition3}).
    As a function $\R_+^{2(n+1)} \to \R$, $f_n \in H^{\otimes (n+1)}$. 
    Then $u \in \text{Dom}(\delta)$ if and only if the series
    \begin{equation}
        \label{eqn:swapSeriesAndDiv}
        \delta(u) = \sum_{n=0}^\infty \mathcal I_{n+1}(f_n)
    \end{equation}
    converges in $L^2(\Omega)$.
\end{proposition}
To check convergence of the series (\ref{eqn:swapSeriesAndDiv}) in $L^2(\Omega)$, we often make use of the property (\ref{eqn:skorohodIntIsometry}) of the multiple Skorohod integral.

\subsection{The Kimura Equation}
\label{ssec:bgKim}

In this section, we discuss the \emph{deterministic} Kimura equation, which is the other prerequisite for studying (\ref{eqn:SKE1}).
We will rely on the results on the deterministic Kimura equation to investigate the stochastic analog. 
Recall that we are interested in the Kimura equation (\ref{eqn:KimuraEqn}) with $b \equiv 0$.
In this section, following \cite{chenDEG}, we present an introduction to the slightly more general case of a constant drift $b \equiv \nu < 1$.
That is, we explore the equation
\begin{equation}
    \label{eqn:kimuraDrift}
    \partial_t u(z,t) = z \partial_z^2 u(z,t) + \nu \partial_z u(z,t),
\end{equation}
for $z, t > 0$, which has the fundamental solution
\begin{equation}
    \label{eqn:fsDriftDefn}
    q_\nu(z,w,t) := \frac{z^{\frac{1-\nu}{2}} w^{\frac{\nu-1}{2}}}{t} e^{-\frac{z+w}{t}} I_{1-\nu} \p{ \frac{2 \sqrt{zw}}{t} }.
\end{equation}
In particular, for $\nu = 0$, we recover (\ref{eqn:q0Formula}).
Furthermore, when $\nu = 0$, we have the following estimate.
\begin{proposition}
    \label{prop:gaussianBd}
    For every $z,w,t > 0$ we have
    \begin{equation}
        \label{eqn:gaussianBd}
        q_0(z,w,t) \leq \frac{z}{t^2}e^{-\frac{(\sqrt{z}-\sqrt{w})^2}{t}}
    \end{equation}
    and further if $zw \geq t^2$, then there exists a constant $C > 0$ such that
    \begin{equation}
        \label{eqn:gaussianBdRefined}
        q_0(z,x,t) \leq C \cdot \frac{z^{\frac{1}{4}}w^{-\frac{3}{4}}}{\sqrt t} e^{- \frac{(\sqrt{z} - \sqrt{w})^2}{t}}.
    \end{equation}
\end{proposition}

In addition, the authors of \cite{chenDEG} compare the behavior of $q_\nu(z,w,t)$ with that of the solution to (\ref{eqn:kimuraDrift}) subject to a bounded, deterministic potential $V$:
\begin{equation}
    \label{eqn:kimuraDriftPotential}
    \partial_t u^V(z,t) = z \partial_z^2 u^V(z,t) + \nu \partial_z u^V(z,t) + u^V(z,t) V(z).
\end{equation}
Namely, if $q_\nu^V(z,w,t)$ denotes the fundamental solution to (\ref{eqn:kimuraDriftPotential}), then the authors show that
\begin{equation}
    \label{eqn:fsRatio}
    \sup_{z,w > 0} \abs{ \frac{q_\nu^V(z,w,t)}{q_\nu(z,w,t)} - 1} \leq e^{t \|V\|_u} - 1,
\end{equation}
which allows them to establish a sharp regularity theory on $q_\nu^V(z,w,t)$ for small $t$.

Following the Duhamel perturbation method, it is possible to express the fundamental solution $q_\nu^V(z,w,t)$ of (\ref{eqn:kimuraDriftPotential}) in terms of $q_\nu(z,w,t)$.
In $\mathsection \ref{sec:eu}$, we apply a similar method to (\ref{eqn:SKE1}); therefore, as a primer we include the Duhamel perturbation method derivation in \cite{chenDEG}.

For $z, w, t > 0$, $q_\nu^V(z,w,t)$ satisfies the integral equation
\begin{equation}
    \label{eqn:kimuraDriftPotentialSoln}
    q_\nu^V(z,w,t) = q_\nu(z,w,t) + \int_0^t \int_0^\infty q_\nu(z,x,t-\tau)q_\nu^V(x,w,\tau)\mathrm dx\mathrm d\tau.
\end{equation}

Notice that $q_\nu^V$ appears recursively in the integrand.
Formally unfolding (\ref{eqn:kimuraDriftPotentialSoln}) once, we obtain
\begin{align*}
    &q_\nu^V(z,w,t)\\
    &= q_\nu(z,w,t) + \int_0^t \int_0^\infty q_\nu(z,x,t-\tau)q_\nu^V(x,w,\tau)V(x)\mathrm dx\mathrm d\tau\\
    &= q_\nu(z,w,t) + \int_0^t\int_0^\infty q_\nu(z,x,t-\tau)\biggr(q_\nu(x,w,\tau) \\
    &\hspace{25mm} \left.  + \int_0^\tau \int_0^\infty \hspace{-3mm} q_\nu(x,x',\tau-\tau') q_\nu(x',w,\tau')V(x') \mathrm dx'\mathrm d\tau'\right) V(x)\mathrm dx \mathrm d\tau \\
    &= q_\nu(z,w,t) + \int_0^t \int_0^\infty q_\nu(z,x,t-\tau)q_\nu(x,w,\tau) V(x)\mathrm dx\mathrm d\tau \\
    &+ \int_0^t\int_0^{\tau} \int_{\R_+^2} q_\nu(z,x,t-\tau)q_\nu(x,x',\tau-\tau')q_\nu^V(x',w,\tau') V(x')V(x)\mathrm dx'\mathrm dx \mathrm d\tau'\mathrm d\tau.
\end{align*}
Continuing this procedure, set $q_{\nu,0}(z,w,t) := q_\nu(z,w,t)$ and for $n \geq 1$, if we inductively define
\[
    q_{\nu,n}(z,w,t) := \int_0^t \int_0^\infty q_\nu(z,z_n,t-t_n)q_{\nu, n-1}(z_n,w,t_n)V(z_n)\mathrm dz_n\mathrm dt_n,
\]
then after unfolding $n$ times, we obtain
\begin{align*}
    q_\nu^V(z,w,t) = \sum_{k=0}^n q_{\nu, k}(z,w,t) + \int_{\Delta_{n+1} t}\int_{\R_+^{n+1}} q_\nu^V(z_0,w,t_0)\prod_{i=1}^{n+1} q_\nu(z_i, z_{i-1}, t_{i}-t_{i-1})V(z_i)\mathrm d\mathbf{z}\mathrm d\mathbf{t},
\end{align*}
where $\Delta_{n+1} t : = \set{(t_0, \ldots, t_n) \in (0,t)^{n+1} \mid 0 < t_0 < \cdots < t_n < t}$, $\mathbf z = (z_0, \ldots, z_n)$, $\mathbf t = (t_0, \ldots, t_n)$, and we put $z_{n+1} = z$ and $t_{n+1} = t$.
Iterating this process at infinitum, it can be shown that 
\begin{equation}
    \label{eqn:fsSeries}
    q_\nu^V(z,w,t) := \sum_{n=0}^\infty q_{\nu, n}(z,w,t),
\end{equation}
converges absolutely and satisfies (\ref{eqn:kimuraDriftPotentialSoln}).
Using the fact that $q_\nu$ satisfies the Kolmogorov-Chapman equations, by a simple inductive argument,
\[
    q_{\nu, n}(z,w,t) \leq \frac{t^n\|V\|_u^n}{n!}q_\nu(z,w,t),
\]
from which (\ref{eqn:fsRatio}) directly follows.

Intuitively, when $t$ is small, the potential $V$ has not yet generated enough ``perturbation'', so $q_\nu^V(z,w,t)$ is expected to be close to $q_\nu(z,w,t)$.
(\ref{eqn:fsRatio}) not only gives a precise interpretation of this fact, since it implies that the difference
\[
    \abs{\frac{q_\nu^V(z,w,t)}{q_\nu(z,w,t)} - 1} = O(t),
\]
for small $t$, but the estimate is also uniform in the spatial variables $(z,w)$ which is sharper and more accurate than the standard heat kernel estimate.

\section{Stochastic Kimura Equation} \label{sec:results}
In this section, we combine the methods from \cite{balanHAM, balanCONT, nualart} and the results from \cite{chenDEG, chenWF} to study various aspects of the stochastic Kimura equation
\[
    \begin{cases}   
        &\partial_{t}u\left(z,t\right) = z\partial_{z}^{2}u\left(z,t\right) + u\left(z,t\right)\dot{W}\left(z,t\right) \quad\quad z, t > 0,\\
        &\hspace{3.2mm}u(z, 0) = u_0(z) \hfill z > 0,\\
        &\hspace{3.8mm}u(0, t) = 0 \hfill t > 0,
    \end{cases}.
\]
such as existence and uniqueness of a solution, the moment bounds, and the regularity of its sample paths.
We also consider the variation wherein a degeneracy is introduced to the noise in the equation,
\begin{equation}
    \label{eqn:SKE2}
    \begin{cases}   
        &\partial_{t}u\left(z,t\right) = z\partial_{z}^{2}u\left(z,t\right) + \hat z^\beta u\left(z,t\right)\dot{W}\left(z,t\right) \quad\quad z, t > 0,\\
        &\hspace{3.2mm}u(z, 0) = u_0(z) \hfill z > 0,\\
        &\hspace{3.8mm}u(0, t) = 0 \hfill t > 0,
    \end{cases}
\end{equation}
and investigate how the choice of $\beta > 0$ affects the various properties studied.
Moreover, we compare the behaviors of the solutions of (\ref{eqn:SKE1}) and (\ref{eqn:SKE2}) to that of (\ref{eqn:kimuraDrift}) with $\nu = 0$ near the boundary 0 and for small times.
\subsection{Existence and Uniqueness of a Solution} 
\label{sec:eu}

In this section, we prove that (\ref{eqn:SKE1}) admits a uniqueness mild solution, up to a modification.
\begin{definition}
    \label{defn:mildSolution}
    
    A square-integrable random field $u = \set{u(z,t) \mid z,t \geq 0}$ is a \emph{mild solution} to the stochastic Kimura equation (\ref{eqn:SKE1}) with deterministic initial condition $u_0 \in C_b(\R_+)$ if
    \begin{enumerate}[label=(\textbf{\alph*})]
        \item $u$ has a jointly measurable modification (again denoted by $u$) and for all $z, t \geq 0$,
        \begin{equation}
            \label{eqn:uniformL2Bd}
            \E[u^2(z,t)] < \infty;
        \end{equation}
        \item For all $z, t > 0$ the process $\set{\ind{(0,t)}(\tau)q_0(z,w,t-\tau)u(w,\tau) \mid w, \tau > 0}$ is Skorohod integrable and 
        \begin{equation}
            \label{eqn:SKE1IntForm}
            u(z,t) = \int_0^\infty q_0(z,w,t)u_0(w)\mathrm dw + \int_0^t\int_0^\infty q_0(z,w,t-\tau)u(w,\tau)W(\mathrm dw, \mathrm d\tau).
        \end{equation}
    \end{enumerate}
\end{definition}
Suppose $u$ is a mild solution to (\ref{eqn:SKE1}), then for every $z,t > 0$, $u$ admits the following Wiener chaos expansion
\begin{equation}
    \label{eqn:solnWienerChaosExpansion}
    u(z,t) = \sum_{n=0}^\infty \mathcal I_n(f_n(\cdot ; z, t)),
\end{equation}
for some $f_n(\cdot; z, t) \in H^{\otimes n}$.
Consequently, to compute $\E[u^2(z,t)]$ we now obtain
\[
    \E[u^2(z,t)] = \sum_{n=0}^\infty \E[\mathcal I_n^2(f_n(\cdot ; z, t))] = \sum_{n=0}^\infty n! \|\tilde f_n(\cdot;z,t)\|_{n}^2,
\]
so we turn our attention to finding an expression for $f_n(\cdot; z, t)$ and computing its norm.

Recall that by virtue of being a mild solution to (\ref{eqn:SKE1}), $u$ satisfies the integral equation (\ref{eqn:SKE1IntForm}).
Following the method of Duhamel, unfolding (\ref{eqn:SKE1IntForm}) recursively in the Skorohod integral and denoting $u_0(z,t) = \E[u(z,t)]$ yields
\begin{align*}
    &u(z,t) = u_0(z,t) + \int_0^t\int_0^\infty q_0(z,z_1,t_1-\tau)u(z_1,t_1)W(\mathrm dz_1, \mathrm dt_1)\\
    &\hspace{10mm} = u_0(z,t) + \int_0^t\int_0^\infty q_0(z,z_1,t_1-\tau)u_0(z_1,t_1)W(\mathrm dz_1, \mathrm dt_1)\\
    &+ \int_0^t\int_0^{t_1}\int_0^\infty\int_0^\infty q_0(z,z_1,t-t_1)q_0(z_1,z_2,t_1-t_2)u(z_2,t_2)W(\mathrm dz_1, \mathrm dt_1)W(\mathrm dz_2, \mathrm dt_2),
\end{align*}
which can be repeated indefinitely.
In other words, setting $f_0'= u_0$ then recursively defining 
\begin{align}
    f_n'(z_1,t_1,\ldots,z_n,t_n;z,t) = \ind{(0,t)}(t_n)&q_0(z,z_n,t-t_n) \label{eqn:fnKernelsDefn} \\
    &\cdot f_{n-1}'(z_1,t_1,\ldots,z_{n-1},t_{n-1};z_n,t_n) \notag
\end{align}
for $n \geq 1$, we obtain (formally) that
\[
    u(z,t) = \sum_{n=0}^\infty \mathcal I_n(f_n'(\cdot;z,t)),
\]
and $f_n(\cdot; z,t) = f_n'(\cdot; z, t)$.
For these functions, we will slightly abuse the notation and write $n!\|\tilde f_n(\cdot; z,t)\|_n^2$ as $\|f_n(\cdot; z,t)\|_n^2$, where by direct computation
\begin{align*}
    \|f_n(\cdot;z,t)\|_n^2 = \int_{\Delta_nt}\int_{\R_+^n} \int_{\Delta_nt}\int_{\R_+^n} \prod_{i=0}^{n-1}& q_0(x_{i+1},x_i,r_{i+1}-r_i)q_0(y_{i+1},y_i,s_{i+1}-s_i)\\
    &\cdot f(x_i-y_i)\gamma(r_i-s_i) \mathrm d\mathbf x\mathrm d\mathbf y\mathrm d\mathbf r\mathrm d\mathbf s,
\end{align*}
where $x_n = y_n = z$ and $r_n = s_n = t$.

Of course, it remains to rigorously show that (\ref{eqn:solnWienerChaosExpansion}) with the derived expression in (\ref{eqn:fnKernelsDefn}) for $f_n$ is in fact the unique mild solution to (\ref{eqn:SKE1}).
We build towards the proof by first studying the case of space-time white noise via a treatment similar to \cite{balanHAM, nualart}.
Henceforth, without loss of generality we may assume constant initial condition.
Because $u_0 \in C_b(\R_+)$, all subsequent results may be recovered by first applying $u_0(z) \leq \|u_0\|_u$.

\subsubsection{Space-Time White Noise}
\label{ssec:euSTWN}

In the space-time white noise case, we are interested in (\ref{eqn:SKE1}) when the centered Gaussian noise $\{\dot W(z, t) \mid z, t > 0\}$ formally has covariance
\[
    \E[W(z_1,t_1)W(z_2,t_2)] = \delta_0(z_1-z_2)\delta_0(t_1-t_2),
\]
where $\delta_0$ denotes the Dirac delta.
In other words, $\inner{\cdot}{\cdot}_H$ coincides with the standard $L^2$ inner product.

Given the kernels $f_n$ defined above, we first aim to prove that for every $z, t > 0$, $f_n(\cdot; z, t) \in H^{\otimes n}$.
The proof for the base case $n = 0$ is straightforward, but we will need the following technical lemma on $u_0(z,t)$.
\begin{lemma}
    \label{lem:u0Int}
    Let $z,t > 0$, then the following equality holds
    \begin{equation}
        \label{eqn:u0}
        u_0(z,t) = \int_0^\infty q_0(z,w,t)\mathrm dw = 1 - e^{-\frac{z}{t}},
    \end{equation}
    and we have the bounds
    \begin{equation}
        \label{eqn:u0Bds}
        e^{-1}\frac{z}{t} \leq u_0(z,t) \leq \frac{z}{t} \wedge 1,
    \end{equation}
    where the lower bound holds for $z < t$.
\end{lemma}
\begin{proof}
    Applying $(\ref{eqn:watsonu0Int})$ with $\nu=0$ and $a=1$ and (\ref{eqn:BesselIDefn}), respectively, we have
    \begin{align*}
        \int_0^\infty q_0(z,w,t)\mathrm dw &= \int_0^\infty \sqrt{\frac{z}{w}} \frac{e^{-\frac{z+w}{t}}}{t} I_1\p{2\frac{\sqrt{zw}}{t}}\mathrm dw\\
        &= 2\frac{\sqrt ze^{-\frac{z}{t}}}{t}i^{-1}\int_0^\infty e^{-\frac{u^2}{t}}J_1\p{2\frac{\sqrt{z}i}{t}u}\mathrm du\\
        &= 2\frac{\sqrt ze^{-\frac{z}{t}}}{t}i^{-1} \cdot \frac{\sqrt z i}{2} 
        {}_1F_1
        \left[
        \hspace{-1mm}
        \begin{array}{cc}
            1  \\
            2
        \end{array} \hspace{-2mm}
        \mathrel{\Big|} \frac{z}{t}
        \right]\\
        &= \frac{ze^{-\frac{z}{t}}}{t} \frac{(e^{\frac{z}{t}}-1)t}{z}\\
        &= 1 - e^{-\frac{z}{t}}.
    \end{align*}
    The bounds follow by taking the first-order Taylor expansion.
\end{proof}

For $n \geq 1$, due to the recursive definition of $f_n(\cdot;z,t)$, in order to show that $f_n(\cdot ; z, t) \in H^{\otimes n}$, it suffices to verify that $\|f_n(\cdot ; z, t)\|_{n}^2 < \infty$ which we can in turn recursively compute:
\begin{equation}
    \label{eqn:recursiveNormSTWN}
    \|f_n(\cdot; z, t)\|_{n}^2 = \int_0^t \int_0^\infty q_0^2(z,w,t-\tau)\|f_{n-1}(\cdot; w, \tau)\|_{n-1}^2\mathrm dw\mathrm d\tau.
\end{equation}
Using (\ref{eqn:u0Bds}), for $n = 1$ we observe
\[
  \|f_1(\cdot; z, t)\|_{1}^2 = \int_0^t \int_0^\infty q_0^2(z,w,t-\tau)u_0^2(w,\tau)\mathrm dw\mathrm d\tau \leq \int_0^t \int_0^\infty q_0^2(z,w,\tau)\mathrm dw\mathrm d\tau,
\]
which can once again be direct computed and bounded.
\begin{lemma}
    \label{lem:fsSqInt}
    Define $U : \R_+ \to \R_+$ by
    \begin{equation}
        \label{eqn:fsSqIntDefn}
        U(x) := e^{-x}I_0(x) - \frac{1}{2}e^{-2x},
    \end{equation}
    then for $z, t > 0$,
    \begin{equation}
        \label{eqn:fsSqIntCalc}
        \int_0^t \int_0^\infty q_0^2(z,w,\tau)\mathrm dw \mathrm d\tau = U\p{\frac{z}{t}},
    \end{equation}
    and moreover $U(z/t) \leq 1/2$.
\end{lemma}
\begin{proof}
    Per the definition of $q_0$ and (\ref{eqn:BesselIDefn}), we have
    \begin{align*}
        \int_0^\infty \hspace{-2mm} q_0^2(z, w, \tau)\mathrm dw &= \int_0^\infty \frac{zw^{-1}}{\tau^2} e^{-\frac{2(z+w)}{\tau}}I_1^2\p{2\frac{\sqrt{zw}}{\tau}} \mathrm dw = \frac{ze^{-\frac{2z}{\tau}}}{\tau^2i^{2}} \int_0^\infty \hspace{-2mm} w^{-1}e^{-\frac{2w}{\tau}}J_1^2\p{\frac{2i\sqrt{z w}}{\tau}}\mathrm dw.
    \end{align*}
    Next, we apply Proposition \ref{prop:WeberIntegral} to obtain
    \[
        i^{-2}\int_0^\infty w^{-1}e^{-\frac{2w}{\tau}}J_1^2\p{\frac{2i\sqrt{z w}}{\tau}}\mathrm dw = e^{\frac{z}{\tau}}\p{I_0\p{\frac{z}{\tau}} - I_1\p{\frac{z}{\tau}}} - 1,
    \]
    and hence we get
    \begin{equation}
        \label{eqn:fsSqInt1}
        \int_0^t\int_0^\infty q_0^2(z,w,\tau)\mathrm dw = \int_0^t \frac{ze^{-\frac{2z}{\tau}}}{\tau^2}\p{ e^{\frac{z}{\tau}}\p{I_0\p{\frac{z}{\tau}} - I_1\p{\frac{z}{\tau}}} - 1}\mathrm d\tau.
    \end{equation}
    Via integration by parts, 
    \begin{align*}
        z\int_0^t \frac{e^{-\frac{z}{\tau}}}{\tau^2}\besselI{1}{\frac{z}{\tau}}\mathrm dr &= -\int_0^t e^{-\frac{z}{\tau}}\dv{}{\tau}\b{\besselI{0}{\frac{z}{\tau}}}\mathrm d\tau\\
        &= -e^{-\frac{z}{\tau}}\besselI{0}{\frac{z}{\tau}}\Big|_0^t + z\int_0^t \frac{e^{-\frac{z}{\tau}}}{\tau^2}\besselI{0}{\frac{z}{\tau}}\mathrm d\tau\\
        &= -e^{-\frac{z}{t}}\besselI{0}{\frac{z}{t}} + z\int_0^t \frac{e^{-\frac{z}{\tau}}}{\tau^2}\besselI{0}{\frac{z}{\tau}}\mathrm d\tau,
    \end{align*}
    and substituting this last expression into (\ref{eqn:fsSqInt1}), (\ref{eqn:fsSqIntCalc}) directly follows. 
    Finally, $\|U\|_u = 1/2$ follows from straightforward calculations.
\end{proof}

Once again, we have a uniform bound for the norm.
This reveals the main pillar of our strategy - using induction, and Lemmas \ref{lem:u0Int} and \ref{lem:fsSqInt} to conclude that
\begin{equation}
    \label{eqn:unUnifL2BdSTWN}
    \|f_n(\cdot; z, t)\|_{n}^2 \leq \frac{1}{2^n}.
\end{equation}

Indeed, the case $n = 0$ is confirmed by Lemma \ref{lem:u0Int}.
Next, assuming (\ref{eqn:unUnifL2BdSTWN}) holds for $n-1$ for some $n \geq 1$, then by (\ref{eqn:recursiveNormSTWN}),
\begin{align}
    \|f_n(\cdot; z, t)\|_{n}^2 &= \int_0^t \int_0^\infty q_0^2(z,w,t-\tau)\|f_{n-1}(\cdot; w, \tau)\|_{n}^2\mathrm dw\mathrm d\tau, \notag \\
    &\leq 2^{-n+1}\int_0^t \int_0^\infty q_0^2(z,w,t-\tau)\mathrm dw\mathrm d\tau = \frac{1}{2^n}. \label{eqn:unUnifL2BdSTWNProof}
\end{align}
Hence it directly follows that $f_n(\cdot ; z, t) \in H^{\otimes n}$ and we may set $u_n(z,t) := \mathcal I_n(f_n(\cdot; z, t))$ for notational simplicity.
We are now equipped to present the main result of this section.

\begin{theorem}
    \label{thm:euSTWN}
    When $W$ is the space-time white noise, $u = \set{u(z, t) \mid z, t \geq 0}$ as in (\ref{eqn:solnWienerChaosExpansion}) is well-defined as a convergent series in $L^2(\Omega)$.
    Moreover, 
    \begin{equation}
        \label{eqn:unifL2BdSTWN}
        \sup_{z, t \geq 0} \E[u^2(z,t)] \leq 2,
    \end{equation}
    and $u$ is the unique mild solution to (\ref{eqn:SKE1}).
\end{theorem}
\begin{proof}
    The convergence of the series in (\ref{eqn:solnWienerChaosExpansion}) to $u(z,t)$ in $L^2(\Omega)$ and (\ref{eqn:unifL2BdSTWN}) follow by the summability of the bounds (\ref{eqn:unUnifL2BdSTWN}) over $n \geq 0$.
    
    To see that $u$ has a jointly measurable modification, we first show that $(z,t) \mapsto u(z,t)$ is continuous in $L^2(\Omega)$.
    To this end, we have
    \begin{align*}
        &\E[|u_n(z_1,t_1) - u_n(z_2,t_2)|^2] = \\
        &\int_0^{t_1 \wedge t_2} \int_0^\infty \abs{q_0(z_1,w,t_1-\tau) - q_0(z_2,w,t_2-\tau)}^2\E[u_{n-1}^2(w,\tau)]\mathrm dw \mathrm d\tau\\
        &\leq 2^{-n+1}\int_0^{t_1 \wedge t_2} \int_0^\infty \abs{q_0(z_1,w,t_1-\tau) - q_0(z_2,w,t_2-\tau)}^2\mathrm dw \mathrm d\tau,
    \end{align*}
    where the last bound vanishes as $(z_1,t_1) \to (z_2,t_2)$ by the dominated convergence theorem and continuity of the fundamental solution $q_0$.
    By the continuity of $u_n$ in $L^2(\Omega)$ and (\ref{eqn:unifL2BdSTWN}), it directly follows that $u$ is continuous in $L^2(\Omega)$, and in turn by Theorem 30 in Chapter IV of \cite{DellacherieMaisonneuveMeyer1992}, $u$ possesses a jointly measurable modification.
    We denote this modification again by $u$.
    We have verified condition $(\textbf{a})$ in Definition \ref{defn:mildSolution}.
    
    Next, we check condition $(\textbf{b})$.
    For $z, w, t - \tau, \tau > 0$, let 
    \begin{equation}
        \label{eqn:solnIntegrand}
        u^{(z,t)}(w,\tau) := \ind{(0,t)}(\tau)q_0(z,w,t-\tau)u(w,\tau).
    \end{equation}
    By measurability of $u$, it is immediate that $u^{(z,t)}(w,\tau)$ is adapted to $W$.
    Moreover by (\ref{eqn:unifL2BdSTWN}), $\E[(u^{(z,t)}(w,\tau))^2] < \infty$ and so $u^{(z,t)}(w, \tau)$ admits the Wiener chaos expansion
    \[
        u^{(z,t)}(w, \tau) = \sum_{n=0}^\infty \mathcal I_n\p{g_n^{(z,t)}(\cdot; w, \tau)},
    \]
    where $g_n^{(z,t)}(w, \tau) = \ind{(0,t)}(\tau)q_0(z,w,t-\tau)f_n(\cdot; w, \tau) = f_{n+1}(\cdot, w, \tau; z, t)$.
    Clearly, $g_n^{(z,t)} \in H^{\otimes(n+1)}$ by the corresponding properties for $u$ and $f_{n+1}$, respectively, and 
    \[
        \delta\p{u^{(z,t)}} = \sum_{n=0}^\infty \mathcal I_{n+1}\p{g_n^{(z,t)}} = \sum_{n=0}^\infty \mathcal I_{n+1}\p{f_{n+1}(\cdot; z, t)} = \sum_{n=0}^\infty u_n(z,t)
    \]
    converges in $L^2(\Omega)$.
    
    Therefore $u^{(z,t)}$ is indeed Skorohod integrable and by Proposition \ref{prop:swapSeriesAndDiv} (1.3.7 of \cite{nualartBook}),
    \begin{align*}
        &\int_0^t\int_0^\infty q_0(z,w,t-\tau)u(w,\tau)W(\mathrm dw,\mathrm d\tau) \\
        &= \delta\p{u^{(z,t)}} = \sum_{n=0}^\infty \mathcal I_{n+1}\p{f_{n+1}(\cdot; z, t)} = u(z, t) - u_0(z, t).
    \end{align*}
    Thus, $u$ also satisfies condition $(\textbf{b})$ and is in fact a mild solution of (\ref{eqn:SKE1}).

    Next, we prove that $u$ is the unique mild solution to (\ref{eqn:SKE1}).
    Let $v = \set{v(z,t) : z,t \geq 0}$ be another solution to (\ref{eqn:SKE1}) with Wiener chaos expansion
    \[
        v(z,t) = \sum_{n=0}^\infty \mathcal I_n(k_n(\cdot; z, t))
    \]
    for some symmetric non-negative functions $k_n(\cdot; z, t) \in H^{\otimes n}$ and $z, t > 0$ fixed.
    We will show that $k_n(\cdot; z, t) = \tilde f_n(\cdot, z, t)$ for every $n \geq 1$.

    Let $v^{(z,t)}(w, \tau)$ be defined similarly to (\ref{eqn:solnIntegrand}).
    Then by definition of a solution, the process $v^{(z,t)} := \set{v^{(z,t)}(w,\tau) | w,\tau > 0}$ is Skorohod integrable and admits the Wiener chaos expansion
    \[
        v^{(z,t)}(w,\tau) = \sum_{n=0}^\infty \mathcal I_n \p{h_n^{(z,t)}(\cdot; w, \tau)},
    \]
    where $h_n^{(z,t)}(\cdot; w, \tau) = \ind{(0,t)}(\tau)q_0(z,w,t-\tau)k_n(\cdot; w, \tau)$.
    By Proposition \ref{prop:swapSeriesAndDiv},
    \[
        \sum_{n=0}^\infty \mathcal I_{n+1}\p{k_{n+1}(\cdot;z,t)} = v(z,t) - \E[v(z,t)] = \delta\p{v^{(z,t)}} =  \sum_{n=0}^\infty \mathcal I_{n+1}\p{\widetilde{h_n^{(z,t)}}},
    \]
    from which it follows that $k_{n+1}(\cdot, z, t) = \widetilde{h_n^{(z,t)}}$ by the uniqueness of the Wiener chaos expansion with symmetric kernels.

    We now express $k_n(\cdot, z, t)$ recursively.
    For $n = 0$, we obtain
    \[
        k_1(z_1, t_1, z, t) = \widetilde{ h_0^{(z,t)} } (z_1, t_1) = h_0^{(z,t)}(z_1,t_1) = \ind{(0,t)}(t_1)q_0(z,z_1,t-t_1)\p{1-e^{-\frac{z_1}{t_1}}}.
    \]
    For $n = 1$, we obtain
    \begin{align*}
        &k_2(z_1, t_1, z_2, t_2, z, t) = \widetilde{ h_1^{(z,t)} }(z_1,t_1,z_2,t_2) = \frac{1}{2}\b{ h_1^{(z,t)}(z_1,t_1,z_2,t_2) + h_1^{(z,t)}(z_2,t_2,z_1,t_1) } = \\
        &\frac{1}{2}\b{ \ind{(0,t)}(t_2)q_0(z,z_2,t-t_2)k_1(z_1,t_1,z_2,t_2) + \ind{(0,t)}(t_1)q_0(z,z_1,t-t_1)k_1(z_2,t_2,z_1,t_1) }.
    \end{align*}
    Using the expression for $k_1$, we conclude
    \begin{align*}
        k_2(z_1,t_1,z_2,t_2,z,t) = &\frac{1}{2} \left[ \ind{(0,t)}(t_2) q_0(z,z_2,t-t_2) \ind{(0,t_2)}(t_1) q_0(z_2,z_1,t_2-t_1)\p{1-e^{-\frac{z_1}{t_1}}} \right. \\
        &+ \left. \ind{(0,t)}(t_1) q_0(z,z_1,t-t_1) \ind{(0,t_1)}(t_2) q_0(z_1,z_2,t_1-t_2)\p{1-e^{-\frac{z_2}{t_2}}}\right],
    \end{align*}
    that is, $k_2(\cdot, z,t) = \tilde f_2(\cdot, z, t)$.
    Iterating this procedure, we infer that $k_n(\cdot, z, t) = \tilde f_n(\cdot, z, t)$ for every $n \geq 1$.
    Hence,
    \[
        v(z,t) = \sum_{n=0}^\infty \mathcal I_n(k_n(\cdot, z, t)) = \sum_{n=0}^\infty \mathcal I_n(\tilde f_n(\cdot, z, t)) = \sum_{n=0}^\infty \mathcal I_n( f_n(\cdot, z, t))= u(z,t).
    \]
    That is, the solution $u$ given by (\ref{eqn:solnWienerChaosExpansion}) is the unique mild solution to (\ref{eqn:SKE1}), up to a modification.
\end{proof}

\subsubsection{Colored Noise}
\label{ssec:euColored}

We continue to the case of colored noise.
That is, $W = \{\dot W(z,t) \mid z, t > 0\}$ is a centered Gaussian process whose covariance is formally given by
\[
    \E\b{\dot W(z_1,t_1)\dot W(z_2, t_2)} = f(z_1-z_2)\gamma(t_1-t_2),
\]
for $z_1, z_2, t_1, t_2 > 0$, where $f, \gamma : \R \to \R_+$ are non-negative definite, symmetric, and locally integrable kernels.

As in the previous section, we seek to show that the process $u$ defined in (\ref{eqn:solnWienerChaosExpansion}) is the unique mild solution to (\ref{eqn:SKE1}) when $W$ is a colored noise.
To this end, we show that the series in (\ref{eqn:solnWienerChaosExpansion}) converges in $L^2(\Omega)$ for all $z, t > 0$ and we recursively compute $\inner{f_{n}(\cdot;z_1,t_1)}{f_{n}(\cdot;z_2,t_2)}_n$: 
\begin{align}
    \label{eqn:recursiveNormGen} &\inner{f_{n}(\cdot;z_1,t_1)}{f_{n}(\cdot;z_2,t_2)}_n = \int_{[0,t_1] \times [0,t_2]} \hspace{-14mm} \gamma(r-s)\int_{\R_+^2} \hspace{-3mm} f(x-y) \\
    & \cdot q_0(z_1,x,t_1-r)q_0(z_2,y,t_2-s) \inner{f_{n-1}(\cdot;x,r)}{f_{n-1}(\cdot;y,s)}_{n-1} \mathrm dx\mathrm dy\mathrm dr\mathrm ds. \notag 
\end{align}
for $z_1, z_2, t_1, t_2 > 0$ and $n \geq 1$.

However, here the main challenge arises, unlike in (\ref{eqn:unUnifL2BdSTWNProof}), after we apply Lemma \ref{lem:u0Int} for $n = 1$, we arrive at
\[
    \inner{f_{1}(\cdot;z_1,t_1)}{f_{1}(\cdot;z_2,t_2)}_1 \leq \int_{[0,t_1] \times [0,t_2]} \hspace{-12mm} \gamma(r-s)\int_{\R_+^2} f(x-y)q_0(z_1,x,t_1-r)q_0(z_2,y,t_2-s)\mathrm dx\mathrm dy\mathrm dr\mathrm ds,
\]
which cannot be directly computed as in Lemma \ref{lem:fsSqInt}.
Instead, we seek to extract $\gamma$ and $f$ from the integrals by imposing the following additional conditions:
\begin{enumerate}[label=(\textbf{\roman*})]
    \item 
        $f$ is non-increasing on $\R_+$,
    \item
        for all $\varepsilon > 0$, $\int_{-\varepsilon}^\varepsilon f(x)\mathrm dx < \infty$ and
        \[
            \lim_{\varepsilon \searrow 0} \int_{-\varepsilon}^\varepsilon f(x)\mathrm dx = 0.
        \]
\end{enumerate}

In addition, for $\varepsilon, t > 0$ we introduce the following notation
\begin{align}
    A_\varepsilon := \set{(x,y) \in \R_+^2 \mid |x-y| < \varepsilon}, \label{def:Aeps}\\
    F_\varepsilon := \int_{-\varepsilon}^\varepsilon f(x) \mathrm dx = 2 \int_0^\varepsilon f(x)\mathrm dx, \label{def:Feps}\\
    \Gamma_t := \int_{-t}^t \gamma(\tau)\mathrm d\tau = 2 \int_0^t \gamma(\tau)\mathrm d\tau, \label{def:GAMeps}.
\end{align}
Note these quantities are well-defined by local-integrability and symmetry of $f$ and $\gamma$.
We proceed to make our strategy concrete.
\begin{lemma}
    \label{lem:u1BdColored}
    Let $z_1, z_2, t_1, t_2 > 0$ and let $U$ be as in (\ref{eqn:fsSqIntDefn}). 
    Then for any $\varepsilon > 0$ we have
    \begin{equation}
        \label{eqn:u1BdColored}
        \inner{f_{1}(\cdot;z_1,t_1)}{f_{1}(\cdot;z_2,t_2)}_1 \leq \Gamma_{t_1 \vee t_2}\p{\frac{F_\varepsilon}{2} + f(\varepsilon)\sqrt{t_1t_2}}.
    \end{equation}
\end{lemma}
\begin{proof}
    Let $\varepsilon > 0$ and define
    \begin{align}
        \label{eqn:u1BdColored1} L_{1,0} &:= \int_{[0,t_1]\times[0,t_2]}\hspace{-12mm}\gamma(r-s)\int_{A_\varepsilon} f(x-y)q_0(z_1,x,t_1-r)q_0(z_2,y,t_2-s) \mathrm dx\mathrm dy\mathrm dr\mathrm ds,\\
        \label{eqn:u1BdColored2} L_{2,0} &:= \int_{[0,t_1]\times[0,t_2]}\hspace{-12mm}\gamma(r-s)\int_{A_\varepsilon^c} f(x-y)q_0(z_1,x,t_1-r)q_0(z_2,y,t_2-s)\mathrm dx\mathrm dy\mathrm dr\mathrm ds.
    \end{align}
    Then by (\ref{eqn:u0Bds}), we'll treat the two terms (\ref{eqn:u1BdColored1}) and (\ref{eqn:u1BdColored2}) individually.
    By Proposition \ref{prop:modifiedYoung} applied twice (once for the spatial variables, once for the temporal variables), we get
    \begin{align*}
        L_{1,0} &\leq \int_{[0,t_1]\times[0,t_2]} \hspace{-12mm} \gamma(r-s) F_\varepsilon \p{\int_0^\infty q_0^2(z_1,x,t_1-r)\mathrm dx \int_0^\infty q_0^2(z_2,y,t_2-s) \mathrm dy}^{\frac{1}{2}}\mathrm dr\mathrm ds\\
        &\leq \Gamma_{t_1 \vee t_2}F_\varepsilon \p{\int_0^{t_1}\int_0^\infty q_0^2(z_1,x,t_1-r)\mathrm dx\mathrm dr \int_0^{t_2}\int_0^\infty q_0^2(z_2,y,t_2-s) \mathrm dy\mathrm ds}^{\frac{1}{2}}.
    \end{align*}
    Notice that these last integrals are precisely quantities studied in the white noise case.
    Therefore, Lemma \ref{lem:fsSqInt} immediately implies that
    \begin{equation}
        \label{eqn:u1BdNear}
        L_{1,0} \leq \Gamma_{t_1 \vee t_2}F_\varepsilon \sqrt{U\p{\frac{z_1}{t_1}}U\p{\frac{z_2}{t_2}}}.
    \end{equation}
    On the other hand, by our assumptions on $f$, we have
    \begin{align*}
        L_{2,0} &\leq f(\varepsilon)\int_{[0,t_1]\times[0,t_2]} \hspace{-12mm} \gamma(r-s)\int_{\R_+^2}q_0(z_1,x,t_1-r)q_0(z_2,y,t_2-s)\mathrm dx\mathrm dy\mathrm dr\mathrm ds\\
        &= f(\varepsilon)\int_{[0,t_1]\times[0,t_2]}\hspace{-12mm} \gamma(r-s) \int_0^\infty q_0(z_1,x,t_1-r)\mathrm dx \int_0^\infty q_0(z_2,y,t_2-s)\mathrm dy \mathrm dr\mathrm ds\\
        &= f(\varepsilon)\int_{[0,t_1]\times[0,t_2]}\hspace{-12mm} \gamma(r-s) u_0(z_1, t_1-r)u_0(z_2, t_2-s)\mathrm dr\mathrm ds.
    \end{align*}
    Again, the above integrals are also quantities studied in the white noise case.
    This time by Lemma \ref{lem:u0Int}, it follows that
    \[
         L_{2,0} \leq f(\varepsilon)\int_{[0,t_1]\times[0,t_2]}\hspace{-12mm}\gamma(r-s)\p{1-e^{-\frac{z_1}{t_1-r}}}\p{1-e^{-\frac{z_2}{t_2-s}}}\mathrm dr\mathrm ds \leq f(\varepsilon) \int_{[0,t_1]\times[0,t_2]}\hspace{-12mm}\gamma(r-s)\mathrm dr\mathrm ds.
    \]
    By applying Proposition \ref{prop:modifiedYoung} once again, 
    \begin{equation}
        \label{eqn:u1BdFar}
        L_{2,0} \leq \Gamma_{t_1\vee t_2}f(\varepsilon) \sqrt{t_1t_2}.
    \end{equation}
    Substituting the bounds (\ref{eqn:u1BdNear}) and (\ref{eqn:u1BdFar}), we achieve the stated result (\ref{eqn:u1BdColored}).
\end{proof}

As in the space-time white noise case, we obtain a bound uniform in $z_1$ and $z_2$.
Moreover, the bound in (\ref{eqn:u1BdColored}) is increasing with respect to the time variables.
Hence, it is tempting to replicate (\ref{eqn:unUnifL2BdSTWN}) and obtain a geometric bound
\begin{align*}
    &\|f_n(\cdot; z,t)\|_n^2 \\
    &= \int_{[0,t]^2} \hspace{-6mm} \gamma(r-s)\int_{\R_+^2} f(x-y)q_0(z,x,t-r)q_0(z,y,t-s)\inner{f_{n-1}(\cdot;x,r)}{f_{n-1}(\cdot;y,s)}_{n-1}\mathrm dx\mathrm dy\mathrm dr\mathrm ds,  \\
    &\leq \int_{[0,t]^2} \hspace{-6mm} \gamma(r-s)\int_{\R_+^2} f(x-y)q_0(z,x,t-r)q_0(z,y,t-s)\p{\Gamma_{r \vee s}\p{\frac{F_\varepsilon}{2} + f(\varepsilon)\sqrt{rs}}}^{n-1}\mathrm dx\mathrm dy\mathrm dr\mathrm ds, \\
    &\leq \Gamma_{t}^{n-1}\p{\frac{F_\varepsilon}{2} + f(\varepsilon)t}^{n-1}\int_{[0,t]^2} \hspace{-6mm} \gamma(r-s)\int_{\R_+^2} f(x-y)q_0(z,x,t-r)q_0(z,y,t-s)\mathrm dx\mathrm dy\mathrm dr\mathrm ds, 
\end{align*}
Thus, it follows from here that
\begin{equation}
    \label{eqn:unGeometricL2Bd} 
    \|f_n(\cdot; z,t)\|_n^2 \leq \Gamma_{t}^n\p{\frac{F_\varepsilon}{2} + f(\varepsilon)t}^{n}, 
\end{equation}
for the $L^2$-norm, but the following remark demonstrates that such an argument may potentially fail to hold for general $f$.
\begin{remark}
    \label{rem:geometricFails}
    Suppose $\gamma(z) = f(z) = |z|^{-1/2}$, then 
    \[
        \sum_{n=0}^\infty \|f_n(\cdot; z,t)\|_n^2 \leq \sum_{n=0}^\infty 4^nt^{\frac{n}{2}}\p{2\sqrt \varepsilon + \frac{t}{\sqrt \varepsilon}},
    \]
    and whose right hand side converges if and only if
    \[
        4^nt^{\frac{n}{2}}\p{2\sqrt \varepsilon + \frac{t}{\sqrt \varepsilon}} < 1.
    \]
    Yet for $t$ sufficiently large, no choice of $\varepsilon$ exists for which this condition holds.
\end{remark}

In light of Remark \ref{rem:geometricFails}, we are forced to derive tighter bounds.
As an example, let us examine the case $n = 2$.
As before, we arrive at
\begin{align*}
    &\inner{f_{2}(\cdot;z_1,t_1)}{f_{2}(\cdot;z_2,t_2)}_2 \\
    &\leq \Gamma_{t_1 \vee t_2}\int_{[0,t_1]\times[0,t_2]} \hspace{-14mm} \gamma(r-s) \p{\frac{F_\varepsilon}{2} + f(\varepsilon)\sqrt{rs}}\int_{\R_+^2} \hspace{-3.5mm} f(x-y)q_0(z,x,t-r)q_0(z,y,t-s)\mathrm dx\mathrm dy\mathrm dr\mathrm ds.
\end{align*}
Define
\begin{align}
    \label{eqn:u2BdColored1}  \raisetag{-7mm} \hspace{10mm} L_{1,0}' &:= \Gamma_{t_1 \vee t_2}\int_{[0,t_1]\times[0,t_2]} \hspace{-14mm} \gamma(r-s) \p{\frac{F_\varepsilon}{2} + f(\varepsilon)\sqrt{rs}} \int_{A_\varepsilon} \hspace{-3.5mm} f(x-y)q_0(z,x,t-r)q_0(z,y,t-s)\mathrm dx\mathrm dy\mathrm dr\mathrm ds,\\
    \label{eqn:u2BdColored2}  \raisetag{-6.5mm} \hspace{10mm} L_{2,0}' &:= \Gamma_{t_1 \vee t_2}\int_{[0,t_1]\times[0,t_2]} \hspace{-14mm} \gamma(r-s) \p{\frac{F_\varepsilon}{2} + f(\varepsilon)\sqrt{rs}}\int_{A_\varepsilon^c} \hspace{-3.5mm} f(x-y)q_0(z,x,t-r)q_0(z,y,t-s)\mathrm dx\mathrm dy\mathrm dr\mathrm ds.
\end{align}
Then by similar arguments to those in the proof of Lemma \ref{lem:u1BdColored},
we obtain
\begin{align*}
    L_{1,0}' &\leq \Gamma_{t_1 \vee t_2}\int_{[0,t_1]\times[0,t_2]} \hspace{-12mm} \gamma(r-s) \p{\frac{F_\varepsilon}{2} + f(\varepsilon)\sqrt{t_1t_2}} \int_{A_\varepsilon} f(x-y)q_0(z,x,t-r)q_0(z,y,t-s)\mathrm dx\mathrm dy\mathrm dr\mathrm ds\\
    &\leq \Gamma_{t_1 \vee t_2}^2\p{\frac{F_\varepsilon^2}{4} + \frac{F_\varepsilon}{2}f(\varepsilon)\sqrt{t_1t_2}},
\end{align*}
and
\[
    L_{2,0}' \leq \Gamma_{t_1 \vee t_2}^2\p{\frac{F_\varepsilon}{2}f(\varepsilon)\sqrt{t_1 t_2} + f(\varepsilon)^2\frac{t_1 t_2}{2}}. 
\]
Putting it all together, we have
\begin{align}
    \label{eqn:u2BdColored}  \raisetag{-7mm} \hspace{10mm}
    \inner{f_{2}(\cdot;z_1,t_1)}{f_{2}(\cdot;z_2,t_2)}_2 \leq L_{1,0}' + L_{2,0}' \leq \Gamma_{t_1 \vee t_2}^2 \p{\frac{F_\varepsilon^2}{4} + F_\varepsilon f(\varepsilon)\sqrt{t_1 t_2} + f(\varepsilon)^2\frac{t_1 t_2}{2} }.
\end{align}
Again, we obtain a bound uniform in $z_1$ and $z_2$ and increasing in $t_1$ and $t_2$, with powers of $\Gamma_{t_1 \vee t_2}, F(\varepsilon),$ and $f(\varepsilon)$ appearing throughout.
The key insight for generalizing to higher order Wiener chaos spaces is that this approach, which we have already applied twice, can be reiterated indefinitely to obtain a bound which maintains its general pattern.
As it turns out, it can be shown that this pattern possesses a tree-like structure (see Figure \ref{fig:my_label}).
\begin{figure}[ht]
    \centering
    \vspace{-12mm}
    \begin{tikzpicture}
        \tikzstyle{level 1}=[edge from parent={missing}, sibling distance=70mm]
        \tikzstyle{level 2}=[sibling distance=70mm]
        \tikzstyle{level 3}=[sibling distance=45mm]
        \tikzstyle{level 4}=[sibling distance=35mm]
            \hspace{-8mm}\node(root){}
            child{node(a){$n=0$} edge from parent[draw=none]
            child{node(b){$n=1$} edge from parent[draw=none]
            child{node(c){$n=2$} edge from parent[draw=none]
            child{node(n){} edge from parent[dashed]}}}}
            child{node(0){1} edge from parent[draw=none]
            child{node(1){$\displaystyle \frac{\Gamma_tF_\varepsilon}{2}$}
                    child{node(3){$\displaystyle \frac{\Gamma_t^2F_\varepsilon^2}{4}$}
                        child{node{$\displaystyle \frac{\Gamma_t^nF_\varepsilon^n}{2^n}$} edge from parent[dashed]}}
                    child{node{$\Gamma_t^2\displaystyle \frac{F_\varepsilon}{2}f(\varepsilon)t$}}}
            child{node(2){$\Gamma_t \displaystyle f(\varepsilon)t$}
                    child{node{$\Gamma_t^2\displaystyle \frac{F_\varepsilon}{2}f(\varepsilon)t$}}
                    child{node(4){$\Gamma_t^2 \displaystyle \frac{(f(\varepsilon)t)^2}{2}$}
                    child{node{$\Gamma_t^n\displaystyle \frac{(f(\varepsilon)t)^n}{n!}$} edge from parent[dashed]}}
            }};
            \node at($(3-1)!.5!(4-1)$){$\Gamma_t^n\displaystyle \binom{m-1}{k}\p{\frac{F_\varepsilon}{2}}^{n-k-1}\frac{(f(\varepsilon)t)^{k+1}}{(k+1)!}$};
            
            \node at($(1)!.5!(2)$){+};
            \node at($(3)!.5!(1-2)$){+};
            \node at($(1-2)!.5!(2-1)$){+};
            \node at($(2-1)!.5!(4)$){+};
            \node at($(3-1)!.16!(4-1)$){$+\cdots+$};
            \node at($(3-1)!.84!(4-1)$){$+\cdots+$};
    \end{tikzpicture}
    \caption{The left child of each vertex corresponds to the bound obtained by handling the integral over $A_\varepsilon$ while the right child corresponds to the integral over $A_\varepsilon^c$.
    Each $1 \leq m \leq n$ corresponds to the $m$th subtree from the left whose respective roots are the right children of the leftmost vertex at each depth.
    Moreover, for any given path in a subtree, each $0 \leq k \leq m - 1$ corresponds to the number of rights in the path.}
    \label{fig:my_label}
\end{figure}
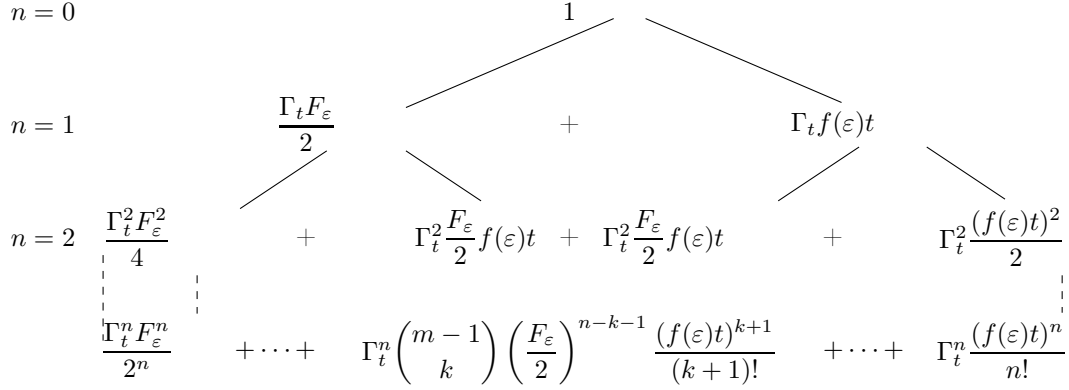

Therefore, in order to rigorously prove an analog to (\ref{eqn:unUnifL2BdSTWN}), we need the following technical lemma.
\begin{lemma}
    \label{lem:unTreeNodeBd}
    Let $z_1,z_2,t_1,t_2 > 0$ and let $k \geq 0$.
    Then for $U$ as in (\ref{eqn:fsSqIntDefn}) and any $\varepsilon > 0$ we have
    \begin{align}
        \int_{[0,t_1]\times[0,t_2]}\hspace{-12mm}\gamma(r-s)&\int_{\R_+^2} f(x-y)q_0(z_1,x,t_1-r)q_0(z_2,y,t_2-s)(r s)^{\frac{k}{2}}\mathrm dx\mathrm dy\mathrm dr\mathrm ds \notag \\
        \label{eqn:unTreeNodeBd} &\leq \Gamma_{t_1\vee t_2} \p{F_\varepsilon \sqrt{U\p{\frac{z_1}{t_1}}U\p{\frac{z_2}{t_2}}t_1^k t_2^k} + \frac{f(\varepsilon)(t_1 t_2)^{\frac{k+1}{2}}}{k+1}}.
    \end{align}
\end{lemma} 
\begin{proof}
    Following a similar strategy to Lemma \ref{lem:u1BdColored}, define $A_\varepsilon$ as before and
    \begin{align}
        \label{eqn:unTreeNodeBd1} \raisetag{-6.5mm} \hspace{10mm}
        L_{1,k} &:= \int_{[0,t_1]\times[0,t_2]}\hspace{-12mm}\gamma(r-s)\int_{A_\varepsilon} f(x-y)q_0(z_1,x,t_1-r)q_0(z_2,y,t_2-s)(rs)^{\frac{k}{2}}\mathrm dx\mathrm dy\mathrm dr\mathrm ds,\\
        \label{eqn:unTreeNodeBd2} \raisetag{-6.5mm} \hspace{10mm}
        L_{2,k} &:= \int_{[0,t_1]\times[0,t_2]}\hspace{-12mm}\gamma(r-s)\int_{A_\varepsilon^c} f(x-y)q_0(z_1,x,t_1-r)q_0(z_2,y,t_2-s)(rs)^{\frac{k}{2}}\mathrm dx\mathrm dy\mathrm dr\mathrm ds.
    \end{align}
    We easily treat (\ref{eqn:unTreeNodeBd1}) by reusing the bound for $L_{1,0}$ in (\ref{eqn:u1BdColored1}):
    \begin{align}
         \hspace{10mm} L_{1,k} &\leq (t_1t_2)^{\frac{k}{2}} \int_{[0,t_1]\times[0,t_2]}\hspace{-12mm}\gamma(r-s)\int_{A_\varepsilon} f(x-y)q_0(z_1,x,t_1-r)q_0(z_2,y,t_2-s) \mathrm dx\mathrm dy\mathrm dr\mathrm ds \notag \\
         \raisetag{-8.75mm} &\leq \Gamma_{t_1\vee t_2}F_\varepsilon \sqrt{U\p{\frac{z_1}{t_1}}U\p{\frac{z_2}{t_2}}t_1^k t_2^k}. \label{eqn:unNear}
    \end{align}
    To bound (\ref{eqn:unTreeNodeBd2}), we again apply a similar strategy used to obtain the bound for $L_{2,0}$ in (\ref{eqn:u1BdColored2}):
    \begin{align}
        L_{2,k} &\leq f(\varepsilon) \int_{[0,t_1]\times[0,t_2]}\hspace{-12mm}\gamma(r-s)(r s)^{\frac{k}{2}} u_0(z_1, t_1-r)u_0(z_2, t_2-s) \mathrm dr\mathrm ds \notag \\
        &\leq f(\varepsilon) \int_{[0,t_1]\times[0,t_2]}\hspace{-12mm}\gamma(r-s)(r s)^{\frac{k}{2}} \mathrm dr\mathrm ds \notag \\
        &\leq \Gamma_{t_1 \vee t_2}f(\varepsilon) \p{\int_0^{t_1} r^k \mathrm dr \int_0^{t_2} s^k\mathrm ds}^{\frac{1}{2}} \notag \\
        &= \frac{\Gamma_{t_1 \vee t_2}f(\varepsilon)(t_1t_2)^{\frac{k+1}{2}}}{k+1}, \label{eqn:unFar}
    \end{align}
    where the penultimate inequality follows per usual by Proposition \ref{prop:modifiedYoung}.
    Hence, the result of the lemma following by summing (\ref{eqn:unNear}) and (\ref{eqn:unFar}).
\end{proof}

We are now equipped to present the solution to (\ref{eqn:SKE1}).
\begin{theorem}
    \label{thm:euColored}
    Suppose $f$ satisfies conditions (\textbf{i}) and (\textbf{ii}).
    Then for every $z, t > 0$, $u(z,t)$ as in (\ref{eqn:solnWienerChaosExpansion}) is well-defined as a convergent series in $L^2(\Omega)$.
    Moreover for $\varepsilon > 0$ sufficiently small ($\varepsilon$ depends on $t)$,
    \begin{align}
        \E[u^2(z,t)] \leq u_0(z,t)^2 + & 
        U\p{\frac{z}{t}}\frac{2\Gamma_tF_\varepsilon}{2-\Gamma_tF_\varepsilon} \notag \\ 
        &\hspace{10mm}+ \frac{\sqrt 2 \Gamma_tf(\varepsilon)t}{\sqrt 2 - \sqrt{\Gamma_tF_\varepsilon}}\exp\set{\frac{2f(\varepsilon)t\sqrt{\Gamma_tF_\varepsilon}}{F_\varepsilon(\sqrt 2 - \sqrt{\Gamma_tF_\varepsilon})}}. \label{eqn:unifL2BdColored}
    \end{align}
    $u = \set{u(z,t) \mid z, t \geq 0}$ is the unique mild solution to (\ref{eqn:SKE1}).
\end{theorem}
\begin{proof}
    We focus on the convergence of the series in (\ref{eqn:solnWienerChaosExpansion}) to $u(z,t)$ in $L^2(\Omega)$.
    The rest of the theorem follows similarly as in the proof of Theorem \ref{thm:euSTWN} for the case of space-time white noise.

    Let $z_1, z_2, t_1, t_2, \varepsilon > 0$, we begin by inductively showing 
    \begin{align}
        &\inner{f_{n}(\cdot;z_1,t_1)}{f_{n}(\cdot;z_2,t_2)}_n \leq \p{\frac{\Gamma_{t_1\vee t_2}F_\varepsilon}{2}}^n \left[2\sqrt{U\p{\frac{z_1}{t_1}}U\p{\frac{z_2}{t_2}}}\right.\notag\\
        \raisetag{-8.75mm}&\hspace{55mm}+ \left.\sum_{m=1}^n \sum_{k=0}^{m-1}\binom{m-1}{k}\frac{1}{(k+1)!}\p{\frac{2f(\varepsilon)\sqrt{t_1t_2}}{F_\varepsilon}}^{k+1} \right]\label{eqn:unCovBdColored}
    \end{align}
    for $n \geq 1$.
    The base case $n = 1$ is proved in Lemma \ref{lem:u1BdColored}.
    Next, let $n \geq 2$ and suppose (\ref{eqn:unCovBdColored}) holds for $n - 1$.
    Then by (\ref{eqn:recursiveNormGen}), we obtain
    \begin{align*}
        &\inner{f_{n}(\cdot;z_1,t_1)}{f_{n}(\cdot;z_2,t_2)}_n \\
        &\leq \int_{[0,t_1]\times[0,t_2]}\hspace{-12mm}\gamma(r-s)\int_{\R_+^2}f(x-y)q_0(z_1,x,t_1-r)q_0(z_2,y,t_2-s)\\
        &\hspace{5mm}\p{\frac{\Gamma_{r\vee s}F_\varepsilon}{2}}^{n-1}\b{1 + \sum_{m=1}^{n-1} \sum_{k=0}^{m-1}\binom{m-1}{k}\frac{1}{(k+1)!}\p{\frac{2f(\varepsilon)\sqrt{rs}}{F_\varepsilon}}^{k+1}} \mathrm dx\mathrm dy\mathrm dr\mathrm ds\\
        &= \p{\frac{\Gamma_{t_1\vee t_2}F_\varepsilon}{2}}^{n-1} \int_{[0,t_1]\times[0,t_2]}\hspace{-12mm}\gamma(r-s)\int_{\R_+^2} f(x-y)q_0(z_1,x,t_1-r)q_0(z_2,y,t_2-s)\mathrm dx\mathrm dy\mathrm dr\mathrm ds\\
        &\hspace{5mm}+ \p{\frac{\Gamma_{t_1 \vee t_2}F_\varepsilon}{2}}^{n-1} \sum_{m=1}^{n-1}\sum_{k=0}^{m-1}\binom{m-1}{k}\frac{1}{(k+1)!} \p{\frac{2f(\varepsilon)}{F_\varepsilon}}^{k+1}\\
        &\hspace{10mm}\cdot\int_{[0,t_1]\times[0,t_2]}\hspace{-12mm}\gamma(r-s)\int_{\R_+^2} f(x-y)q_0(z_1,x,t_1-r)q_0(z_2,y,t_2-s)(r s)^{\frac{k+1}{2}}\mathrm dx\mathrm dy\mathrm dr\mathrm ds.
    \end{align*}
    By applying Lemma \ref{lem:unTreeNodeBd} to both integrals, it then follows that
    \begin{align*}
        &\E\b{u_n(z_1,t_1)u_n(z_2,t_2)} \leq \p{\frac{\Gamma_{t_1 \vee t_2}F_\varepsilon}{2}}^{n} \left[2 \sqrt{U\p{\frac{z_1}{t_1}}U\p{\frac{z_2}{t_2}}} + \frac{2f(\varepsilon)\sqrt{t_1t_2}}{F_\varepsilon} \right. \\
        &\left.+ \sum_{m=1}^{n-1}\sum_{k=0}^{m-1}\binom{m-1}{k}
        \p{\p{\frac{F_\varepsilon}{2}}^{-k-1}\frac{(f(\varepsilon)\sqrt{t_1t_2})^{k+1}}{(k+1)!} + \p{\frac{F_\varepsilon}{2}}^{-k-2}\frac{(f(\varepsilon)\sqrt{t_1 t_2})^{k+2}}{(k+2)!}}\right].
    \end{align*}
    Regrouping terms, we see that 
    \begin{align*}
        \sum_{m=1}^{n-1}\sum_{k=0}^{m-1}\binom{m-1}{k}
        &\p{\p{\frac{F_\varepsilon}{2}}^{-k-1}\frac{(f(\varepsilon)\sqrt{t_1t_2})^{k+1}}{(k+1)!} + \p{\frac{F_\varepsilon}{2}}^{-k-2}\frac{(f(\varepsilon)\sqrt{t_1 t_2})^{k+2}}{(k+2)!}}\\
        &= \sum_{m=2}^n\sum_{k=0}^{m-1}\binom{m-1}{k}\frac{1}{(k+1)!}\p{\frac{2f(\varepsilon)\sqrt{t_1 t_2}}{F_\varepsilon}}^{k+1},
    \end{align*}
    and
    \begin{align*}
        \frac{2f(\varepsilon)\sqrt{t_1t_2}}{F_\varepsilon} &+ \sum_{m=2}^n\sum_{k=0}^{m-1}\binom{m-1}{k}\frac{1}{(k+1)!}\p{\frac{2f(\varepsilon)\sqrt{t_1 t_2}}{F_\varepsilon}}^{k+1}\\
        &= \sum_{m=1}^n\sum_{k=0}^{m-1}\binom{m-1}{k}\frac{1}{(k+1)!}\p{\frac{2f(\varepsilon)\sqrt{t_1 t_2}}{F_\varepsilon}}^{k+1},
    \end{align*}
    from which (\ref{eqn:unCovBdColored}) follows.

    Setting $z_1 = z_2 = z, t_1 = t_2 = t > 0$, (\ref{eqn:unCovBdColored}) simplifies to
    \begin{align}
        \raisetag{-8.5mm} \hspace{10mm} \label{eqn:unBdColored}
        \|f_{n}(\cdot;z,t)\|_n^2 \leq \p{\frac{\Gamma_{t}F_\varepsilon}{2}}^n \b{ 2U\p{\frac{z}{t}} + \sum_{m=1}^n \sum_{k=0}^{m-1}\binom{m-1}{k}\frac{1}{(k+1)!}\p{\frac{2f(\varepsilon)t}{F_\varepsilon}}^{k+1} }.
    \end{align}
    
    We proceed to prove (\ref{eqn:unifL2BdColored}).
    For every $z, t > 0$, choose $\varepsilon > 0$ such that $\Gamma_t F_\varepsilon < 1$.
    By (\ref{eqn:unBdColored}), we have 
    \begin{align*}
        &\sum_{n=0}^\infty \|f_{n}(\cdot;z,t)\|_n^2 = u_0(z,t)^2 + \sum_{n=1}^\infty \|f_{n}(\cdot;z,t)\|_n^2\\
        &\leq u_0(z,t)^2 + U\p{\frac{z}{t}}\sum_{n=1}^\infty \frac{\p{\Gamma_tF_\varepsilon}^n}{2^{n-1}} + \sum_{n=1}^\infty \p{\frac{\Gamma_tF_\varepsilon}{2}}^n \sum_{m=1}^n \sum_{k=0}^{m-1}\binom{m-1}{k}\frac{1}{(k+1)!}\p{\frac{2f(\varepsilon)t}{F_\varepsilon}}^{k+1}.
    \end{align*}
    By the choice of $\varepsilon$, we have $\frac{\Gamma_tF_\varepsilon}{2} < 1$ and hence
    \[
        \sum_{n=1}^\infty \frac{\Gamma_t^nF_\varepsilon^n}{2^{n-1}} = \frac{\Gamma_tF_\varepsilon}{1 - \p{\displaystyle \frac{\Gamma_tF_\varepsilon}{2}}} = \frac{2\Gamma_tF_\varepsilon}{2-\Gamma_tF_\varepsilon} < \infty.
    \]
    It remains to show the last series converges.
    First, we have
    \begin{align*}
        \sum_{k=0}^{m-1}\binom{m-1}{k}\frac{1}{(k+1)!}\p{\frac{2f(\varepsilon)t}{F_\varepsilon}}^{k+1} &= \frac{2f(\varepsilon)t}{F_\varepsilon} \sum_{k=0}^{m-1}\binom{m-1}{k}\frac{1}{(k+1)!}\p{\frac{2f(\varepsilon)t}{F_\varepsilon}}^k\\
        &\leq \frac{2f(\varepsilon)t}{F_\varepsilon} \sum_{k=0}^{n-1}\binom{n-1}{k}\frac{1}{(k+1)!}\p{\frac{2f(\varepsilon)t}{F_\varepsilon}}^k,
    \end{align*}
    from which it follows that 
    \[
        \sum_{m=1}^n \sum_{k=0}^{m-1}\binom{m-1}{k}\frac{1}{(k+1)!}\p{\frac{2f(\varepsilon)t}{F_\varepsilon}}^{k+1} \leq \frac{2f(\varepsilon)t}{F_\varepsilon} n \sum_{k=0}^{n-1}\binom{n-1}{k}\frac{1}{(k+1)!}\p{\frac{2f(\varepsilon)t}{F_\varepsilon}}^k,
    \]
    and so
    \begin{align*}
        \sum_{n=1}^\infty \p{\frac{\Gamma_tF_\varepsilon}{2}}^n &\sum_{m=1}^n \sum_{k=0}^{m-1}\binom{m-1}{k}\frac{1}{(k+1)!}\p{\frac{2f(\varepsilon)t}{F_\varepsilon}}^{k+1}\\
        &\leq \Gamma_tf(\varepsilon)t \sum_{n=1}^\infty \p{\frac{\Gamma_tF_\varepsilon}{2}}^{n-1} n \sum_{k=0}^{n-1} \binom{n-1}{k} \frac{1}{k!} \p{\frac{2f(\varepsilon)t}{F_\varepsilon}}^k\\
        &= \Gamma_tf(\varepsilon)t \sum_{n=0}^\infty \p{\frac{\Gamma_tF_\varepsilon}{2}}^{n} (n+1) \sum_{k=0}^{n} \binom{n}{k} \frac{1}{k!} \p{\frac{2f(\varepsilon)t}{F_\varepsilon}}^k\\
        &\leq \Gamma_tf(\varepsilon)t \sum_{n=0}^\infty\p{\frac{\Gamma_tF_\varepsilon}{2}}^n(n+1)\sum_{k=0}^n 2^n \frac{1}{k!} \p{\frac{2f(\varepsilon)t}{F_\varepsilon}}^k\\
        &\leq \Gamma_tf(\varepsilon)t e^{\frac{2f(\varepsilon)t}{F_\varepsilon}}\sum_{n=0}^\infty (\Gamma_tF_\varepsilon)^n(n+1)\\
        &= \frac{\Gamma_tf(\varepsilon)t}{(\Gamma_tF_\varepsilon -1)^2}e^{\frac{2f(\varepsilon) t}{F_\varepsilon}},
    \end{align*}
    where convergence in the last equality follows by the fact that $\Gamma_t F_\varepsilon < 1$.
    Hence for all $z, t > 0$ we have (\ref{eqn:unifL2BdColored}) and so the series in (\ref{eqn:solnWienerChaosExpansion}) converges to $u(z,t)$ in $L^2(\Omega)$.
    The rest of the theorem follows similarly as in the proof of Theorem \ref{thm:euSTWN} for the case of space-time white noise.
\end{proof}

\subsection{Ratio Comparison in $L^2$}
\label{sec:ratio}

With existence and uniqueness established, next we study the behavior of the solution near the boundary of the domain.
As in (3.10) in Lemma 3.4 of \cite{chenDEG}, it is insightful to study how the solution to the Kimura equation with a potential compares to the unperturbed solution.
In our setting the solution $u$ is a random field, so instead of a point-wise comparison, we examine the ratio between $u(z,t)$ and the unperturbed, deterministic solution $u_0(z,t)$ in $L^2(\Omega)$.
Specifically, recalling that $u_n(z,t) := \mathcal I_n(f_n(\cdot ; z, t))$ with $f_n$ as defined in (\ref{eqn:fnKernelsDefn}), we study
\begin{equation}
    \label{eqn:ratioDefn}
    \E\b{\p{\frac{u(z,t)}{u_0(z,t)}}^2} = \sum_{n=0}^\infty \E\b{\p{\frac{u_n(z,t)}{u_0(z,t)}}^2}
\end{equation}
for $z$ near 0 and $t$ small.

In the case of (\ref{eqn:SKE1}) with space-time white noise, we see that the solution is quantifiably different at the boundary.
Indeed, using the fact that $I_1(w) \geq w/2$ for every $w > 0$, by the definition of $\E[u_1^2(z, t)]$, we get
\begin{align*}
  \E\b{u_1^2(z,t)} &= \int_0^t \int_0^\infty q_0^2(z,w,t-r)(1-e^{-\frac{w}{r}})^2\mathrm dw\mathrm dr, \\
  &\geq z^2 \int_0^t \frac{e^{-\frac{2z}{t-r}}}{(t-r)^4} \int_0^\infty e^{-\frac{2w}{t-r}}(1-e^{-\frac{w}{r}})^2\mathrm dw\mathrm dr,
\end{align*}
and so by direction computations it follows that
\begin{align*}
    \E\b{u_1^2(z,t)} &= z^2 \int_0^t \frac{e^{-\frac{2z}{t-r}}}{2t(t-r)(t+r)}\mathrm dr\\
    &\geq \frac{z^2}{4t^2} \int_0^t \frac{e^{-\frac{2z}{t-r}}}{(t-r)}\mathrm dr = \frac{z^2}{4t^2} \int_{\frac{2z}{t}}^\infty \frac{e^{-\tau}}{\tau}\mathrm d\tau.
\end{align*}
Since $1 - e^{-\frac{z}{t}} \approx z/t$ for fixed $t$ and sufficiently small $z$, we in turn have
\begin{align*}
    \lim_{z \searrow 0} \E\b{\p{\frac{u_1(z,t)}{u_0(z,t)}}^2} \geq  \lim_{z \searrow 0} \frac{1}{4}\int_{\frac{2z}{t}}^\infty \frac{e^{-\tau}}{\tau}\mathrm d\tau = \frac{1}{4} \int_0^\infty \frac{e^{-\tau}}{\tau}\mathrm d\tau = \infty,
\end{align*}
where the last integral diverges as it is the limit of the gamma function at zero.

To remedy this, we introduce a degenerate coefficient of order $\beta > 0$ to $\dot W$ and consider the following stochastic Kimura equation:
\[
    \begin{cases}   
        &\partial_{t}u\left(z,t\right) = z\partial_{z}^{2}u\left(z,t\right) + \hat z^\beta u\left(z,t\right)\dot{W}\left(z,t\right) \quad\quad z, t > 0,\\
        &\hspace{3.2mm}u(z, 0) = u_0(z) \hfill z > 0,\\
        &\hspace{3.8mm}u(0, t) = 0 \hfill t > 0,
    \end{cases}
\]
where recall that $\hat z := z \wedge 1$.
It can be shown similarly to $\mathsection \ref{sec:eu}$ that (\ref{eqn:SKE2}) admits a unique mild solution.
In addition, the second moments of the solution can be recursively computed by
\begin{align}
    \E[u_n^2(z,t)] = \int_{[0,t]^2}\hspace{-5mm} &\gamma(r-s) \int_{\R_+^2} \hspace{-3mm} f(x-y)\hat x^\beta \hat y^\beta \notag \\
    &\cdot q_0(z,x,t-r)q_0(z,y,t-s) \E[u_{n-1}(x,r)u_{n-1}(y,s)]\mathrm dx\mathrm dy\mathrm dr\mathrm ds. \label{eqn:recursiveDegenNormGen}
\end{align}

In the following subsections, we show that so long as $\beta > 0$, the stochastic and deterministic solutions are ``comparable'' in a proper sense.
As before, we begin our analysis in the case of space-time white noise which we bootstrap to the case of colored noise.

\subsubsection{Space-Time White Noise}
\label{ssec:ratioSTWN}

Our strategy for the case of space-time white noise is similar to the one used to show convergence of (\ref{eqn:solnWienerChaosExpansion}) in $L^2(\Omega)$.
Namely, we use (\ref{eqn:recursiveNormSTWN}) to recursively compute a bound with some repetitive pattern.
We first make use of the following technical lemma for $n = 1$.

\begin{lemma}
    \label{lem:u1RatioBdSpatialSTWN}
    For $\beta, z, t-\tau, \tau > 0$, define
    \begin{equation}
        \label{eqn:u1RatioDefnSpatialSTWN}
        K(z, \tau; 2\beta) := \int_0^\infty \hat w^{2\beta} q_0^2(z,w,t-\tau)\frac{u_0^2(w,\tau)}{u_0^2(z,t)}\mathrm dw.
    \end{equation}
    Then, there exists $C := C_\beta > 0$, such that for all $z, t, t-\tau, \tau > 0$,
    \begin{equation}
        \label{eqn:degenU1SpatialRatioBd}
        K(z, \tau; 2\beta) \leq \frac{4C}{(t-\tau)^{1-\alpha}}(\hat z \vee t)^{2\beta - \alpha},
    \end{equation}
    whenever $0 < \alpha \leq 2\beta \wedge 1/2$.
\end{lemma} 
\begin{proof}
    We further split $K(z, \tau; 2\beta)$ into two terms:
    \begin{align*}
        K_1 &:= \int_0^{\frac{(t-\tau)^2}{z}} \hat w^{2\beta}q_0^2(z,w,t-\tau)\frac{u_0^2(w,\tau)}{u_0^2(z,t)}\mathrm dw,\\
        K_2 &:= \int_{\frac{(t-\tau)^2}{z}}^\infty \hat w^{2\beta}q_0^2(z,w,t-\tau)\frac{u_0^2(w,\tau)}{u_0^2(z,t)}\mathrm dw.
    \end{align*}
    We will treat $K_1$ and $K_2$ in two separate cases.
    \\\\
    \emph{Case $\tau \leq t/2$.}
    In this case, we apply Proposition \ref{prop:gaussianBd} and (\ref{eqn:u0Bds}) to bound $K_1$ as
    \begin{align*}
        K_1 &\leq \frac{\displaystyle \frac{z^2}{(t-\tau)^4}\int_0^{\frac{(t-\tau)^2}{z}} \hat w^{2\beta}e^{-\frac{2(\sqrt z - \sqrt w)^2}{(t-\tau)}} \mathrm dw}{\displaystyle e^{-2}\frac{z^2}{t^2}}\\
        \p{t-\tau \geq \frac{t}{2}} &\leq \frac{C}{(t-\tau)^2} \int_0^{\frac{(t-\tau)^2}{z}}\hat w^{2\beta}e^{-\frac{2(\sqrt z - \sqrt w)^2}{(t-\tau)}} \mathrm dw\\
        &= \frac{C}{(t-\tau)^{\frac{3}{2}}} \p{\frac{1}{\sqrt{\pi (t-\tau)}} \int_0^{\frac{t-\tau}{\sqrt z}} \hat u^{4\beta} u e^{-\frac{2(\sqrt z - u)^2}{t-\tau}}\mathrm du }.
    \end{align*}
    Furthermore, if $z \leq t - \tau$ then by Proposition \ref{prop:gaussianMomentBd},
    \begin{align*}
        K_1 &\leq \frac{C}{(t-\tau)^{\frac{3}{2}}} \p{\frac{1}{\sqrt{\pi (t-\tau)}} \int_0^{\frac{t-\tau}{\sqrt z}} u^{4\beta+1} e^{-\frac{2(\sqrt z - u)^2}{t-\tau}}\mathrm du }\\
        &\leq \frac{2C}{(t-\tau)^{\frac{3}{2}}} (t-\tau)^{2\beta+\frac{1}{2}} = 2C(t-\tau)^{2\beta-1}.
    \end{align*}
    Meanwhile, if $z > t - \tau$, again by Proposition \ref{prop:gaussianMomentBd}, we get
    \begin{align*}
        K_1 &\leq C(t-\tau)^{-\frac{1}{2}}(t-\tau)^{-1}\p{\frac{1}{\sqrt{\pi (t-\tau)}} \int_0^{\frac{t-\tau}{\sqrt z}} \hat u^{4\beta} u e^{-\frac{2(\sqrt z - u)^2}{t-\tau}}\mathrm du }\\
        \p{(t-\tau)^{-1} \leq u^{-1}z^{-\frac{1}{2}}} &\leq Cz^{-\frac{1}{2}}(t-\tau)^{-\frac{1}{2}}\p{\frac{1}{\sqrt{\pi (t-\tau)}} \int_0^{\frac{t-\tau}{\sqrt z}} \hat  u^{4\beta}e^{-\frac{2(\sqrt z - u)^2}{t-\tau}}\mathrm du }\\
        &\leq \frac{2Cz^{-\frac{1}{2}}}{\sqrt{t-\tau}}\hat z^{2\beta}  \leq 2C\hat z^{2\beta} \frac{z^{-\alpha}}{(t-\tau)^{1-\alpha}} \leq 2C \frac{\hat z^{2\beta-\alpha}}{(t-\tau)^{1-\alpha}},
    \end{align*}
    for $0 \leq \alpha \leq 2\beta \wedge 1/2$.
    For $K_2$, we apply (\ref{eqn:gaussianBdRefined}) from Proposition \ref{prop:gaussianBd} and (\ref{eqn:u0Bds}) to get
    \begin{align*}
        K_2 &\leq \frac{\displaystyle \frac{\sqrt z}{t-\tau}\int_{\frac{(t-\tau)^2}{z}}^\infty \hat w^{2\beta} w^{-\frac{3}{2}} e^{-\frac{2(\sqrt z - \sqrt w)^2}{t-\tau}}\mathrm dw}{\displaystyle e^{-2}\frac{z^2}{t^2}}\\
        \p{t-\tau \geq \frac{t}{2}} &\leq C z^{-\frac{3}{2}}(t-\tau)\int_{\frac{(t-\tau)^2}{z}}^\infty \hat w^{2\beta} w^{-\frac{3}{2}} e^{-\frac{2(\sqrt z - \sqrt w)^2}{t-\tau}}\mathrm dw\\
        &= Cz^{-\frac{3}{2}}(t-\tau)^{\frac{3}{2}}\p{\frac{1}{\sqrt{\pi(t-\tau)}} \int_{\frac{t-\tau}{\sqrt z}}^\infty \hat u^{4\beta} u^{-2} e^{- \frac{2(\sqrt z - u)^2}{t-\tau}} \mathrm du}\\
        \p{u^{-1} \leq \frac{\sqrt z}{t - \tau}} &\leq C z^{-1}(t-\tau)^{\frac{1}{2}} \p{\frac{1}{\sqrt{\pi(t-\tau)}} \int_{\frac{t-\tau}{\sqrt z}}^\infty \hat u^{4\beta} e^{- \frac{2(\sqrt z - u)^2}{t-\tau}} \mathrm du}.
    \end{align*}
    Furthermore, if $z > t-\tau$, then again by Proposition \ref{prop:gaussianMomentBd}, for $0 \leq \alpha \leq 2\beta \wedge 1/2$, $K_2 \leq 2Cz^{-\frac{3}{2}}(t-\tau)^{\frac{1}{2}}\hat z^{2\beta} \leq 2C\hat z^{2\beta-\alpha} (t-\tau)^{\alpha-1}$.
    Meanwhile, if $z \leq t-\tau$ we get
    \begin{align*}
        K_2 &\leq Cz^{-\frac{3}{2}} \frac{(t-\tau)^3}{(t-\tau)^{\frac{3}{2}}} \p{\frac{1}{\sqrt{\pi(t-\tau)}} \int_{\frac{t-\tau}{\sqrt z}}^\infty u^{4\beta-2}e^{- \frac{2(\sqrt z - u)^2}{t-\tau}} \mathrm du}\\
        \p{(t-\tau)^{-3}\geq u^{-3}z^{-\frac{3}{2}}} &\leq C(t-\tau)^{-\frac{3}{2}}\p{\frac{1}{\sqrt{\pi(t-\tau)}} \int_{\frac{t-\tau}{\sqrt z}}^\infty u^{4\beta+1}e^{- \frac{2(\sqrt z - u)^2}{t-\tau}} \mathrm du}\\
        &\leq 2C (t-\tau)^{2\beta-1}.
    \end{align*}
    \emph{Case $\tau > t/2$}.
    For $K_1$, we apply (\ref{eqn:gaussianBd}) from Proposition \ref{prop:gaussianBd} and (\ref{eqn:u0Bds}) again to get
    \begin{align*}
        K_1 &\leq \frac{\displaystyle \frac{z^2}{(t-\tau)^4} \int_0^{\frac{(t-\tau)^2}{z}} \hat w^{2\beta}e^{-\frac{2(\sqrt z - \sqrt w)^2}{t-\tau}}u_0^2(w,\tau)\mathrm dw}{\displaystyle e^{-2}\frac{z^2}{t^2}}\\
        &\leq \frac{t^2}{(t-\tau)^4} \int_0^{\frac{(t-\tau)^2}{z}} \hat w^{2\beta} e^{-\frac{2(\sqrt z - \sqrt w)^2}{t-\tau}}\frac{w^2}{\tau^2}\mathrm dw\\
        \p{\tau > \frac{t}{2}} &\leq \frac{C}{(t-\tau)^4} \int_0^{\frac{(t-\tau)^2}{z}} \hat w^{2\beta} w^2 e^{-\frac{2(\sqrt z - \sqrt w)^2}{t-\tau}\mathrm dw }\\
        &= \frac{C}{(t-\tau)^{\frac{7}{2}}} \p{\frac{1}{\sqrt{\pi(t-\tau)}} \int_0^{\frac{t-\tau}{\sqrt z}} \hat u^{4\beta} u^5 e^{-\frac{2(\sqrt z - u)^2}{t-\tau}\mathrm du } }.
    \end{align*}
    Furthermore, if $z \leq t - \tau$ then $K_1 \leq 2C(t-\tau)^{2\beta-1}$.
    Meanwhile, if $z > t - \tau$ we get
    \begin{align*}
        K_1 &\leq  \frac{C}{(t-\tau)^{\frac{7}{2}}} \p{\frac{1}{\sqrt{\pi(t-\tau)}} \int_0^{\frac{t-\tau}{\sqrt z}} \hat u^{4\beta} u^5 e^{-\frac{2(\sqrt z - u)^2}{t-\tau}\mathrm du } }\\
        &= C(t-\tau)^{-\frac{1}{2}}(t-\tau)^{-3} \p{\frac{1}{\sqrt{\pi(t-\tau)}} \int_0^{\frac{t-\tau}{\sqrt z}} \hat u^{4\beta} u^5 e^{-\frac{2(\sqrt z - u)^2}{t-\tau}\mathrm du } }\\
        \p{(t-\tau)^{-3} \leq z^{-\frac{3}{2}}u^{-3}} &\leq Cz^{-\frac{3}{2}}(t-\tau)^{-\frac{1}{2}} \p{\frac{1}{\sqrt{\pi(t-\tau)}} \int_0^{\frac{t-\tau}{\sqrt z}} \hat u^{4\beta} u^2 e^{-\frac{2(\sqrt z - u)^2}{t-\tau}\mathrm du } }\\
        &\leq \frac{2C}{\sqrt{t-\tau}} \hat z^{2\beta} z^{-\frac{1}{2}} \leq 2C\hat z^{2\beta} \frac{z^{-\alpha}}{(t-\tau)^{1-\alpha}} \leq 2C \frac{\hat z^{2\beta-\alpha}}{(t-\tau)^{1-\alpha}}.
    \end{align*}
    Similarly as above, for $K_2$ we apply (\ref{eqn:gaussianBdRefined}) from Proposition \ref{prop:gaussianBd} and (\ref{eqn:u0Bds}) to get
    \begin{align*}
        K_2 &\leq \frac{\displaystyle \frac{\sqrt z}{t-\tau} \int_{\frac{(t-\tau)^2}{z}}^\infty \hat w^{2\beta} w^{-\frac{3}{2}} e^{-\frac{2(\sqrt z - \sqrt w)^2}{t-\tau}}u_0^2(w,\tau)\mathrm dw}{\displaystyle e^{-2}\frac{z^2}{t^2}}\\
        &\leq \frac{z^{-\frac{3}{2}}t^2}{t-\tau} \int_{\frac{(t-\tau)^2}{z}}^\infty \frac{\hat w^{2\beta} w^{\frac{1}{2}}}{\tau^2} e^{-\frac{2(\sqrt z - \sqrt w)^2}{t-\tau}} \mathrm dw\\
        \p{\tau > \frac{t}{2}} &\leq \frac{Cz^{-\frac{3}{2}}}{t-\tau} \int_{\frac{(t-\tau)^2}{z}}^\infty \hat w^{2\beta} w^{\frac{1}{2}} e^{-\frac{2(\sqrt z - \sqrt w)^2}{t-\tau}} \mathrm dw\\
        &= \frac{Cz^{-\frac{3}{2}}}{\sqrt{t-\tau}} \p{\frac{1}{\sqrt{\pi (t-\tau)}} \int_{\frac{t-\tau}{\sqrt z}}^\infty \hat u^{4\beta} u^2 e^{-\frac{2(\sqrt z - u)^2}{t-\tau}} \mathrm du }.
    \end{align*}
    Furthermore, if $z > t - \tau$ then $K_2 \leq 2C(t-\tau)^{-\frac{1}{2}}\hat z^{2\beta}z^{-\frac{1}{2}} \leq 2C\hat z^{2\beta-\alpha}(t-\tau)^{-1+\alpha}$.
    Meanwhile, if $z \leq t - \tau$ we get
    \begin{align*}
        K_2 &\leq C z^{-\frac{3}{2}}(t-\tau)^{-\frac{7}{2}}\cdot(t-\tau)^3\p{\frac{1}{\sqrt{\pi (t-\tau)}} \int_{\frac{t-\tau}{\sqrt z}}^\infty u^{4\beta+2} e^ {-\frac{2(\sqrt z - u)^2}{t-\tau}} \mathrm du }\\
        \p{(t-\tau)^3 \leq z^{\frac{3}{2}}u^3} &\leq C(t-\tau)^{-\frac{7}{2}}\p{\frac{1}{\sqrt{\pi (t-\tau)}} \int_{\frac{t-\tau}{\sqrt z}}^\infty u^{4\beta+5} e^{-\frac{2(\sqrt z - u)^2}{t-\tau}} \mathrm du }\\
        &\leq 2C(t-\tau)^{2\beta-1}.
    \end{align*}
    Combining all these cases, we have shown that
    \[
        K(z, \tau; 2\beta) = K_1 + K_2 \leq 4C (t-\tau)^{\alpha-1}(\hat z \vee t)^{2\beta-\alpha}.
    \]
\end{proof}

Note that if $\beta \leq \beta'$, then for every $w > 0$, $\hat w^{2\beta'} \leq \hat w^{2\beta}$ and hence
\begin{equation}
    \label{eqn:monotonicityKInBeta}
    K(z, \tau; 2\beta') \leq K(z, \tau; 2\beta).
\end{equation}

\begin{remark}
    \label{rem:offDiagonalRatioDefnSpatial}
    We observe that for $z, t > 0$ and $\beta > 0$, if we define
    \begin{equation}
        \label{eqn:offDiagonal}
        \tilde K(z, \tau; \beta) := \int_0^\infty \hat w^\beta q_0(z,w,t-\tau) \frac{u_0(w,\tau)}{u_0(w,\tau)}\mathrm dw,
    \end{equation}
    then following the proof of Lemma 3.5 almost step-by-step (with minor changes), we can get that
    \begin{equation}
        \label{eqn:offDiagonalRatioBdSpatial}
        \tilde K(z, \tau; \beta) \leq 4C(\hat z \vee (t-\tau))^\beta.
    \end{equation}
    This observation on $\tilde K(z, \tau; \beta)$ will be useful in $\mathsection \ref{ssec:ratioColored}$ in treating the ratio $u(z,t)/u_0(z,t)$ under colored noise.
\end{remark}

From this, the ratio follows straightforwardly by computing the integral in time.
\begin{corollary}
    \label{cor:degenU1RatioBd}
    Let $\beta, C,$ and $\alpha$ be the same as above and fixed.
    Set $C_{\alpha,\beta} := 4C/\alpha$.
    Then for all $z, t > 0$,
    \begin{equation}
        \label{eqn:degenU1RatioBd}
        \E\b{\p{\frac{u_1(z,t)}{u_0(z,t)}}^2} \leq C_{\alpha, \beta} (\hat z \vee t)^{2\beta-\alpha}.
    \end{equation}
\end{corollary} 
\begin{proof}
    Follows directly by integrating the result from Lemma \ref{lem:u1RatioBdSpatialSTWN} in $\tau$ over $[0,t]$.
\end{proof}

In light of Corollary \ref{cor:degenU1RatioBd}, for every $w, \tau > 0$,
\[
    \E[u_1^2(w,\tau)] \leq C_{\alpha,\beta} \tau^\alpha (\hat w \vee \tau)^{2\beta-\alpha}u_0^2(w, \tau),
\]
and hence it becomes possible for us to re-use the bound for the function $K(z, \tau; \cdot)$ in Lemma \ref{lem:u1RatioBdSpatialSTWN}.
As it turns out, this is the key ingredient for the main result for this section.

\begin{theorem}
    \label{thm:ratioBdSTWN}
    Let $\beta > 0$.
    When $W$ is a space-time white noise, there exists $T \in (0,1)$ such that for every $t \in [0,T]$, there exists $C_{t ,\beta} > 0$ such that
    \begin{equation}
        \label{eqn:spatialRatioBdSTWN}
        \sup_{z > 0} \E\b{\p{\frac{u(z,t)}{u_0(z,t)}}^2}< C_{t, \beta},
    \end{equation}
    where $T$ and $C_{t,\beta}$ can be explicitly determined.
    Furthermore, for every fixed $z > 0$,
    \begin{equation}
        \label{eqn:temporalRatioBdSTWN}
        \lim_{t \searrow 0} \E\b{\p{\frac{u(z,t)}{u_0(z,t)} - 1}^2} = 0.
    \end{equation}
\end{theorem}
\begin{proof}
    In light of (\ref{eqn:ratioDefn}) we proceed by induction on $n \geq 1$ to show that for $\beta, z, t > 0$ and $0 < \alpha \leq 2 \beta \wedge 1/2$,
    \begin{equation}
        \label{eqn:degenUnRatioBd}
        \E\b{\p{\frac{u_n(z,t)}{u_0(z,t)}}^2} \leq C_{\alpha,\beta}^n t^{n\alpha}(\hat z\vee t)^{(2\beta-\alpha)}.
    \end{equation}
    The base case is given by Corollary \ref{cor:degenU1RatioBd}.
    Next, assume that (\ref{eqn:degenUnRatioBd}) holds for $n - 1$ with $n \geq 2$.
    Then by using (\ref{eqn:degenUnRatioBd}) and (\ref{eqn:monotonicityKInBeta}), we have
    \begin{align*}
        \E\b{\p{\frac{u_n(z,t)}{u_0(z,t)}}^2} &= \frac{\int_0^t \int_0^\infty \hat w^{2\beta} q_0^2(z,w,t-\tau)\E[u_{n-1}^2(w,\tau)]\mathrm dw \mathrm d\tau}{u_0^2(z, t)} \\
        &= \frac{\int_0^t \p{\int_0^\tau + \int_\tau^\infty}\hat w^{2\beta} q_0^2(z,w,t-\tau)\E[u_{n-1}^2(w,\tau)]\mathrm dw \mathrm d\tau}{u_0^2(z, t)} \\
        &\leq C_{\alpha,\beta}^{n-1} \int_0^t \tau^{(n-1)\alpha + 2\beta - \alpha}\int_0^\tau \hat w^{2\beta} q_0^2(z,w,t-\tau)\frac{u_0^2(w,\tau)}{u_0^2(z,t)}\mathrm dw\mathrm d\tau \\
        &+ C_{\alpha,\beta}^{n-1} \int_0^t \tau^{(n-1)\alpha} \int_\tau^\infty \hat w^{4\beta-\alpha}q_0^2(z,w,t-\tau)\frac{u_0^2(w,\tau)}{u_0^2(z,t)}\mathrm dw\mathrm d\tau\\
        &\leq C_{\alpha,\beta}^{n-1}\int_0^t \tau^{(n-1)\alpha}K(z, \tau; 2\beta)\mathrm d\tau \\
        &\leq 4C C_{\alpha,\beta}^{n-1} (\hat z \vee t)^{2\beta-\alpha} \int_0^t \frac{\tau^{(n-1)\alpha}}{(t-\tau)^{1-\alpha}}\mathrm d\tau \\
        &\leq C_{\alpha,\beta}^n t^{n\alpha} (\hat z \vee t)^{2\beta-\alpha},
    \end{align*}
    as desired.
    Hence (\ref{eqn:degenUnRatioBd}) holds for $n \geq 1$.
    Therefore, whenever $C_{\alpha, \beta} t^\alpha < 1$, by substituting (\ref{eqn:degenUnRatioBd}) into (\ref{eqn:ratioDefn}), 
    \begin{equation}
        \label{eqn:degenURatioBd}
         \E\b{\p{\frac{u(z,t)}{u_0(z,t)}}^2} \leq \sum_{n=0}^\infty C_{\alpha, \beta}^n t^{n\alpha}(\hat z \vee t)^{2\beta-\alpha} = \frac{(\hat z \vee t)^{2\beta-\alpha}}{1 - C_{\alpha, \beta}t^\alpha}.
    \end{equation}

    We continue to prove (\ref{eqn:spatialRatioBdSTWN}) and (\ref{eqn:temporalRatioBdSTWN}).
    Let $\alpha = 2\beta \wedge 1/2$ and $T > 0$ such that $C_{\alpha, \beta}T^\alpha < 1$.
    From (\ref{eqn:degenURatioBd}) it follows that for every $t \in [0, T]$,
    \[
        \sup_{z > 0} \E\b{\p{\frac{u(z,t)}{u_0(z,t)}}^2} \leq \frac{(1 \vee t)^{2\beta-\alpha}}{1 - C_{\alpha, \beta}t^\alpha} =: C_{t, \beta},
    \]
    which is precisely (\ref{eqn:spatialRatioBdSTWN}).
    On the other hand, for any $z > 0$ fixed,
    \[
        \lim_{t \searrow 0} \E\b{\p{\frac{u(z,t)}{u_0(z,t)} - 1}^2} = \lim_{t \searrow 0} \sum_{n=1}^\infty \E\b{\p{\frac{u_n(z,t)}{u_0(z,t)}}^2} \leq \lim_{t \searrow 0}\frac{C_{\alpha, \beta}t^\alpha (\hat z \vee t)^{2\beta-\alpha}}{1 - C_{\alpha, \beta}t^\alpha} = 0, 
    \]
    from which (\ref{eqn:temporalRatioBdSTWN}) follows.
\end{proof}

\subsubsection{Colored Noise}
\label{ssec:ratioColored}

As before, we seek to bootstrap the results from the case of space-time white noise.
Fortunately, the same techniques we developed in $\mathsection \ref{sec:eu}$ are still applicable to us now.
Once again, we build to the main result, Theorem \ref{thm:ratioBdGen}, by first treating the case $n = 1$. 

\begin{lemma}
    \label{lem:coloredDegenU1RatioBd}
    Let $z_1, z_2, t_1, t_2 > 0$ and let $\beta, \varepsilon > 0$, then
    \begin{align}
        \raisetag{-6.5mm} \hspace{10mm}\label{eqn:coloredDegenU1RatioBd}
        \E\b{\frac{u_1(z_1,t_1)u_1(z_2,t_2)}{u_0(z_1,t_1)u_0(z_2,t_2)}} \leq \Gamma_{t_1 \vee t_2} (F_\varepsilon + f(\varepsilon))C_{\alpha, \beta} \sqrt{t_1^\alpha t_2^\alpha (\hat z_1 \vee t_1)^{2\beta-\alpha}(\hat z_2 \vee t_2)^{2\beta-\alpha}},
    \end{align}
    where $0 < \alpha \leq 2\beta \wedge 1/2$ and $C_{\alpha, \beta}$ may be greater than in Corollary \ref{cor:degenU1RatioBd}.
\end{lemma} 
\begin{proof}
    We define
    \begin{align*}
        V_1 &:= \int_{[0,t_1]\times[0,t_2]} \hspace{-13mm} \gamma(r-s)\int_{A_\varepsilon} \hat x^\beta \hat y^\beta f(x-y)q_0(z_1,x,t_1-r)q_0(z_2,y,t_2-s) \frac{u_0(x,r)u_0(y,s)}{u_0(z_1,t_1)u_0(z_2,t_2)}\mathrm dx\mathrm dy\mathrm dr\mathrm ds,\\
        V_2 &:= \int_{[0,t_1]\times[0,t_2]} \hspace{-13mm} \gamma(r-s)\int_{A_\varepsilon^c} \hat x^\beta \hat y^\beta f(x-y)q_0(z_1,x,t_1-r)q_0(z_2,y,t_2-s) \frac{u_0(x,r)u_0(y,s)}{u_0(z_1,t_1)u_0(z_2,t_2)}\mathrm dx\mathrm dy\mathrm dr\mathrm ds,
    \end{align*}
    and treat each term separately.
    First, applying Proposition \ref{prop:modifiedYoung} and Lemma \ref{lem:u1RatioBdSpatialSTWN}, we see
    \begin{align}
        V_1 &\leq \Gamma_{t_1 \vee t_2} F_\varepsilon \sqrt{\int_0^{t_1} K(z_1, r, 2\beta)\mathrm dr \int_0^{t_2} K(z_2, s, 2\beta)\mathrm dr} \notag \\
        &\leq \Gamma_{t_1 \vee t_2 }F_\varepsilon C_{\alpha, \beta} \sqrt{t_1^\alpha t_2^\alpha (\hat z_1 \vee t_1)^{2\beta-\alpha}(\hat z_2 \vee t_2)^{2\beta-\alpha}}, \label{eqn:coloredDegenU1RatioBd1}
     \end{align}
    for $K$ as in (\ref{eqn:u1RatioDefnSpatialSTWN}).
    For the second term, by Proposition \ref{prop:modifiedYoung} and Remark \ref{rem:offDiagonalRatioDefnSpatial}, we see
    \begin{align}
        V_2 &\leq \Gamma_{t_1 \vee t_2} f(\varepsilon) \sqrt{\int_0^{t_1} \tilde K(z_1, r, \beta)^2\mathrm dr \int_0^{t_2} \tilde K(z_2, s, \beta)^2\mathrm ds} \notag \\
        &\leq \Gamma_{t_1 \vee t_2} f(\varepsilon) (4C) \sqrt{t_1 t_2} (\hat z_1 \vee t_1)^\beta (\hat z_2 \vee t_2)^\beta \notag \\
        &\leq \Gamma_{t_1 \vee t_2} f(\varepsilon) C_{\alpha, \beta} \sqrt{t_1^\alpha t_2^\alpha (\hat z_1 \vee t_1)^{2\beta-\alpha}(\hat z_2 \vee t_2)^{2\beta-\alpha}}, \label{eqn:coloredDegenU1RatioBd2}
    \end{align}
    where (\ref{eqn:coloredDegenU1RatioBd2}) follows because $\hat z_1, \hat z_2, t_1, t_2 < 1$.
    Combining (\ref{eqn:coloredDegenU1RatioBd1}) and (\ref{eqn:coloredDegenU1RatioBd2}) in turn leads to (\ref{eqn:coloredDegenU1RatioBd}).
\end{proof}

Considering the form of (\ref{eqn:degenUnRatioBd}), it is not hard to see that for $n \geq 1$, we expect 
\begin{align}
    \raisetag{-6.5mm} \hspace{10mm} \label{eqn:coloredDegenUnRatioBd}
    \frac{\E[u_n(z_1,t_1)u_n(z_2,t_2)]}{u_0(z_1,t_1)u_0(z_2,t_2)} \leq \Gamma_{t_1 \vee t_2}^n (F_\varepsilon + f(\varepsilon))^nC_{\alpha,\beta}^n  (t_1 t_2)^{\frac{n\alpha}{2}}((\hat z_1 \vee t_1)(\hat z_2 \vee t_2))^{\beta - \frac{\alpha}{2}}.
\end{align}
    
Before presenting the proof, we first recall Remark \ref{rem:geometricFails}, which motivated the need for a bound better than geometric in case one could not guarantee the existence of $\varepsilon$ such that the concerned series in $\mathsection \ref{ssec:euColored}$ is convergent for all time $t$.
However, the mere issue of convergence will not occur here by virtue of only studying the solution $u$ for small values of $t$. 
Given any value of $\varepsilon > 0$, we may assume that either $t_1$ or $t_2$ is sufficiently small as to ensure summability of (\ref{eqn:coloredDegenUnRatioBd}) over $n \geq 1$.

\begin{theorem}
    \label{thm:ratioBdGen}
    Let $\beta > 0$ and $\varepsilon > 0$.
    Suppose $f$ satisfies conditions (\textbf{i}) and (\textbf{ii}).
    Then there exists $T \in (0,1)$ such that for every $t \in [0,T]$, there exists $C_{t, \beta, \epsilon} > 0$ such that
    \begin{equation}
        \label{eqn:spatialRatioBdColored}
        \sup_{z > 0} \E\b{\p{\frac{u(z,t)}{u_0(z,t)}}^2} < C_{t, \beta, \varepsilon},
    \end{equation}
    where $T$ and $C_{t,\beta,\epsilon}$ can be explicitly determined.
    Furthermore, for every fixed $z > 0$,
    \begin{equation}
        \label{eqn:temporalRatioBdColored}
        \lim_{t \searrow 0} \E\b{\p{\frac{u(z,t)}{u_0(z,t)} - 1}^2} = 0.
    \end{equation}
\end{theorem}
\begin{proof}
    We begin by proving (\ref{eqn:coloredDegenUnRatioBd}) by induction on $n \geq 1$.
    The base case is given by Lemma \ref{lem:coloredDegenU1RatioBd}.
    Now assume that (\ref{eqn:coloredDegenUnRatioBd}) holds for $n - 1$ with $n \geq 2$.
    Define
    \begin{align*}
        \tilde V_1 &:= \int_{[0,t_1]\times[0,t_2]} \hspace{-13mm} \gamma(r-s)\int_{A_\varepsilon} \hspace{-3mm} \hat x^\beta \hat y^\beta f(x-y) q_0(z_1,x,t_1-r)q_0(z_2,y,t_2-s) \\
        &\hspace{6mm} \cdot (\hat x \vee r)^{\beta-\frac{\alpha}{2}}(\hat y \vee s)^{\beta-\frac{\alpha}{2}}\frac{u_0(x,r)u_0(y,s)}{u_0(z_1,t_1)u_0(z_2,t_2)}\mathrm dx\mathrm dy\mathrm dr\mathrm ds\\
        \tilde V_2 &:= \int_{[0,t_1]\times[0,t_2]} \hspace{-13mm} \gamma(r-s)\int_{A_\varepsilon^c} \hspace{-3mm} \hat x^\beta \hat y^\beta f(x-y) q_0(z_1,x,t_1-r)q_0(z_2,y,t_2-s) \\
        &\hspace{6mm}\cdot  (\hat x \vee r)^{\beta-\frac{\alpha}{2}}(\hat y \vee s)^{\beta-\frac{\alpha}{2}} \frac{u_0(x,r)u_0(y,s)}{u_0(z_1,t_1)u_0(z_2,t_2)}\mathrm dx\mathrm dy\mathrm dr\mathrm ds.
    \end{align*}
    We start by examining $\tilde V_1$.
    Applying Proposition \ref{prop:modifiedYoung} and (\ref{eqn:monotonicityKInBeta}), then following the proof for (\ref{eqn:coloredDegenU1RatioBd1}), we see
    \begin{align*}
        \tilde V_1 &\leq \Gamma_{t_1 \vee t_2} F_\varepsilon \sqrt{\int_0^{t_1} K(z_1, r, 2\beta)\mathrm dr \int_0^{t_2} K(z_2, s, 2\beta)\mathrm ds} \\
        &\leq \Gamma_{t_1 \vee t_2} F_\varepsilon C_{\alpha, \beta} \sqrt{t_1^\alpha t_2^\alpha(\hat z_1 \vee t_1)^{2\beta-\alpha}(\hat z_2 \vee t_2)^{2\beta-\alpha}}.
    \end{align*}
    Likewise, again by Proposition \ref{prop:modifiedYoung} and (\ref{eqn:monotonicityKInBeta}), then following the proof for (\ref{eqn:coloredDegenU1RatioBd2}),
    \[
        \tilde V_2 \leq \Gamma_{t_1 \vee t_2} f(\varepsilon) C_{\alpha, \beta} \sqrt{t_1^\alpha t_2^\alpha(\hat z_1 \vee t_1)^{2\beta-\alpha}(\hat z_2 \vee t_2)^{2\beta-\alpha}}.
    \]
    Putting it all together with the inductive assumption, we obtain
    \begin{align*}
        \frac{\E[u_n(z_1,t_1)u_n(z_2,t_2)]}{u_0(z_1,t_1)u_0(z_2,t_2)} &\leq \Gamma_{t_1 \vee t_2}^{n-1}\p{F_\varepsilon + f(\varepsilon)}^{n-1}C_{\alpha,\beta}^{n-1}(t_1t_2)^{\frac{(n-1)\alpha}{2}}\p{\tilde V_1 + \tilde V_2}\\
        &\leq \Gamma_{t_1 \vee t_2}^n (F_\varepsilon + f(\varepsilon))^nC_{\alpha,\beta}^n \p{t_1t_2}^{\frac{n\alpha}{2}}\p{(\hat z_1 \vee t_1)(\hat z_2 \vee t_2)}^{\beta-\frac{\alpha}{2}},
    \end{align*}
    as claimed.

    Next, given $\varepsilon$ and $\alpha$, choose $t > 0$ such that $\Gamma_t(F_\varepsilon + f(\varepsilon))C_{\alpha,\beta}t^\alpha < 1$.
    Then by (\ref{eqn:ratioDefn}) and (\ref{eqn:coloredDegenUnRatioBd}),
    \begin{align}
        \E\b{\p{\frac{u(z,t)}{u_0(z,t)}}^2} &\leq \sum_{n=0}^\infty \Gamma_t^n(F_\varepsilon + f(\varepsilon))^nC_{\alpha,\beta}^nt^{n\alpha} (\hat z \vee t)^{2\beta-\alpha}, \notag \\
        \label{eqn:coloredDegenURatioBd} &= \frac{(\hat z \vee t)^{2\beta-\alpha}}{1-\Gamma_t(F_\varepsilon + f(\varepsilon))C_{\alpha,\beta}t^\alpha}.
    \end{align}

    We continue to prove (\ref{eqn:spatialRatioBdColored}) and (\ref{eqn:temporalRatioBdColored}).
    First, let $\alpha = 2\beta \wedge 1/2$ and $T > 0$ be such that $\Gamma_T(F_\varepsilon + f(\varepsilon))C_{\alpha,\beta}T^\alpha < 1$. For all $t \in [0,T]$, from (\ref{eqn:coloredDegenURatioBd}) it follows that
    \[
        \sup_{z > 0} \E\b{\p{\frac{u(z,t)}{u_0(z,t)}}^2} \leq \frac{(1 \vee t)^{2\beta-\alpha}}{1-\Gamma_t(F_\varepsilon + f(\varepsilon))C_{\alpha,\beta}t^\alpha} =: C_{t, \beta, \varepsilon},
    \]
    which is precisely (\ref{eqn:spatialRatioBdColored}).
    On the other hand, for any $z > 0$ fixed,
    \[
        \lim_{t \searrow 0} \E\b{\p{\frac{u(z,t)}{u_0(z,t)} - 1}^2} \leq \lim_{t \searrow 0} \frac{\Gamma_t (F_\varepsilon + f(\varepsilon))C_{\alpha,\beta}t^\alpha(1 \vee t)^{2\beta-\alpha}}{1-\Gamma_t(F_\varepsilon + f(\varepsilon))C_{\alpha,\beta}t^\alpha} = 0,
    \]
    from which (\ref{eqn:temporalRatioBdColored}) follows.
\end{proof}
\subsection{$L^p-$ Bounds}
\label{sec:moment}

In $\mathsection\ref{sec:eu}$, we presented the bounds (\ref{eqn:unUnifL2BdSTWN}) and (\ref{eqn:unifL2BdSTWN}), and (\ref{eqn:unBdColored}) and (\ref{eqn:unifL2BdColored}) for the $L^2-$ norms of $u_n(z, t)$ and $u(z, t)$ in the space-time white noise and colored noise cases, respectively.
Naturally, one might ask about bounds for general $L^p-$ norms with $p \geq 2$.
However, unlike for the space $L^2(\Omega)$, there is no isometry or expansion property that allows us to directly treat
higher moments.
Nevertheless, for a solution $u$ to either (\ref{eqn:SKE1}) or (\ref{eqn:SKE2}) with Wiener chaos expansion given by (\ref{eqn:solnWienerChaosExpansion}), (\ref{eqn:lpInequality}) allows us to return to the treatment of the $L^2$-norm at the cost of an additional geometric term depending on $p$.

Unfortunately, while the methods developed in $\mathsection\ref{sec:eu}$ give $u_n(z, t) \in L^p(\Omega)$ for all $z, t > 0$ and $p \geq 2$, they do not produce sufficiently sharp bounds to guarantee convergence of the series in (\ref{eqn:solnWienerChaosExpansion}) to $u(z, t)$ in $L^p(\Omega)$ for $p \geq 2$.
For example, in the case of space-time white noise, directly applying the bound (\ref{eqn:unUnifL2BdSTWN}) with (\ref{eqn:lpInequality}) and (\ref{eqn:solnWienerChaosExpansion}) gives
\begin{equation}
    \label{eqn:unrefinedBdsFail}
    \sum_{n=0}^\infty \E\b{|u_n(z,t)|^p}^{\frac{1}{p}} \leq \sum_{n=0}^\infty \p{\frac{p-1}{2}}^{\frac{n}{2}},
\end{equation}
for $z, t > 0$.
Notice that the right-hand side of (\ref{eqn:unrefinedBdsFail}) fails to converge for $p \geq 3$.
As it turns out, the condition $\beta \geq 1/4$ is sufficient to obtain results on $L^p$-integrability for all $p \geq 2$.
To this end, we need to refine the bounds developed in $\mathsection\ref{sec:eu}$.

\subsubsection{Space-Time White Noise}
\label{ssec:momentSTWN}

We begin by re-examining our bound for $\E[u_1^2(z, t)]$ in the case of space-time white noise.
That is, recall that by Lemma \ref{lem:fsSqInt} we have
\[
    \E[u_1^2(z, t)] \leq U\p{\frac{z}{t}} = e^{\frac{-z}{t}}I_0\p{\frac{z}{t}} - \frac{1}{2}e^{-\frac{2z}{t}} \leq \frac{1}{2}.
\]
While we uniformly bounded the function $U$ for convenience, notice that we in fact have $U(x) = O(x^{-1/2})$.
Thus, it is possible to obtain a tighter estimate of $U$ for large argument values by opting to not bound it by a constant.
We make this notion precise in our first refined result.

\begin{lemma}
    \label{lem:u1BdRefinedSTWN}
    There exists a constant $C > 0$ such that for all $z, t-\tau, \tau > 0$, 
    \begin{equation}
        \label{eqn:u1BdSpatialRefinedSTWN}
        \int_0^\infty q_0^2\left(z,w,t-\tau\right)dw \leq Cz^{-\frac{1}{2}}\left(t-\tau\right)^{-\frac{1}{2}}.
    \end{equation}
    Moreover, for $z, t > 0$
    \begin{equation}
        \label{eqn:u1BdRefinedSTWN}
        \int_0^t \int_0^\infty q_0^2(z,w,t-\tau)\mathrm dw\mathrm d\tau \leq \frac{C\sqrt \pi}{\Gamma\p{\frac{3}{2}}} z^{-\frac{1}{2}}t^{\frac{1}{2}},
    \end{equation}
    where $\Gamma$ denotes the Gamma function.
\end{lemma}
\begin{proof}
    As in the proof of Lemma \ref{lem:fsSqInt}, we have
    \begin{equation}
        \label{eqn:u1BdRefinedSTWN1}
        \int_{0}^{\infty}q_{0}^{2}\left(z,w,t-\tau\right)dw = \frac{2z}{\left(t-\tau\right)^{2}}e^{-\frac{2z}{t-\tau}}\int_{0}^{\infty}u^{-1}e^{-\frac{\left(t-\tau\right)u^{2}}{2z}}I_{1}^{2}\left(u\right)du.
    \end{equation}
    In fact, the integral in the right-hand side of (\ref{eqn:u1BdRefinedSTWN1}) is a representation of the Confluent Hypergeometric function $_2F_2$ with parameters $a_1 = 3/2, a_2 = 1$ and $b_1 = 2, b_2 = 3$ (see $\mathsection \ref{sec:appendix}$).
    The tail of said function is well-studied; using Proposition \ref{prop:2F2AsymptoticExpansion}, there exists a constant $C > 0$ depending only on $a_1, a_2, b_1, b_2$ such that
    \[
        \int_{0}^{\infty}q_{0}^{2}\left(z,w,t-\tau\right)dw = \frac{2z}{\left(t-\tau\right)^{2}}e^{-\frac{2z}{t-\tau}}\frac{z}{t-\tau} 
        {}_2F_2
        \left[
        \hspace{-1mm}
        \begin{array}{cc}
            \displaystyle \frac{3}{2}, 1  \\
            2, 3
        \end{array} \hspace{-2mm}
        \mathrel{\Bigg|} \frac{2z}{t-\tau}
        \right]
        \leq Cz^{-\frac{1}{2}}\left(t-\tau\right)^{-\frac{1}{2}}.
    \]
    Hence we have (\ref{eqn:u1BdSpatialRefinedSTWN}).
    Simply integrating with respect to $\tau$ then gives (\ref{eqn:u1BdRefinedSTWN}).
\end{proof}

Now consider the stochastic Kimura equation (\ref{eqn:SKE2}) studied in $\mathsection \ref{sec:ratio}$.
With a little more effort, we can refine the estimate for $\E[u_n^2(z,t)]$ by using Lemma \ref{lem:u1BdRefinedSTWN} with the recursive expression (\ref{eqn:recursiveDegenNormGen}) in the case of space-white noise.
As a result, we can continue to show convergence of the series in (\ref{eqn:solnWienerChaosExpansion}) to $u(z,t)$ in $L^p(\Omega)$ for all $p \geq 2$.

\begin{theorem}
    \label{thm:uPthMomentBdSTWN}
    For every $n \geq 1$, $p \geq 2$, $\beta \geq 1/4$, and $z, t > 0$, when $W$ is a space-time white noise, we have
    \begin{equation}
        \label{eqn:unPthMomentBdSTWN}
        \E[|u_n(z,t)|^p]^{1/p} \leq \frac{\p{C(p-1)\sqrt \pi \sqrt t}^{\frac{n}{2}}}{\sqrt{\Gamma\p{\displaystyle \frac{n}{2}+1}} z^{\frac{1}{4}}}.
    \end{equation}
    Moreover, for every $z, t > 0$ and $p \geq 2$, we have $u(z,t) \in L^p(\Omega)$ and the series in (\ref{eqn:solnWienerChaosExpansion}) converges to $u(z, t)$ in $L^p(\Omega)$.
\end{theorem}
\begin{proof}
    We first revisit the case $p = 2$.
    For $n \geq 1$ and $\beta \geq 1/4$, we show that for $z, t > 0$,
    \begin{equation}
        \label{eqn:unBdRefinedSTWN}
        \E[u_n^2(z,t)] \leq \frac{C^n \pi^{\frac{n}{2}}}{\Gamma\p{\displaystyle \frac{n}{2}+1}} z^{-\frac{1}{2}}t^{\frac{n}{2}}.
    \end{equation}
    We proceed by induction on $n$.
    The base $n = 1$ follows from Lemma \ref{lem:u1BdRefinedSTWN}.
    Now assume that (\ref{eqn:unBdRefinedSTWN}) holds for $n - 1$ with $n \geq 2$.
    Using (\ref{eqn:recursiveDegenNormGen}) with (formally) $\gamma = f = \delta_0$ together with (\ref{eqn:u1BdSpatialRefinedSTWN}),
    \begin{align*}
        \mathbb{E}\left[u_{n+1}^{2}\left(z,t\right)\right] & =\int_{0}^{t}\int_{0}^{\infty} \hat w^{2\beta} q_{0}^{2}\left(z,w,t-\tau\right)\mathbb{E}\left[u_{n}^{2}\left(w,\tau\right)\right] dwd\tau \\
        & \leq \frac{C^n\pi^{\frac{n}{2}}}{\Gamma\p{\displaystyle \frac{n}{2}+1}} \int_{0}^{t}\int_{0}^{\infty} \hat w^{2\beta} q_{0}^{2}\left(z,w,t-\tau\right)w^{-\frac{1}{2}} dw\tau^{\frac{n}{2}}d\tau\\
        (\beta > 1/4) & \leq \frac{C^n\pi^{\frac{n}{2}}}{\Gamma\p{\displaystyle \frac{n}{2}+1}}\int_{0}^{t}\int_{0}^{\infty}q_{0}^{2}\left(z,w,t-\tau\right)dw\tau^{\frac{n}{2}}d\tau\\
        & \leq \frac{C^{n+1}\pi^{\frac{n}{2}}}{\Gamma\p{\displaystyle \frac{n}{2}+1}}z^{-\frac{1}{2}}\int_{0}^{t}\left(t-\tau\right)^{-\frac{1}{2}}\tau^{\frac{n}{2}}d\tau\\
        &= \frac{C^{n+1}\pi^{\frac{n}{2}}}{\Gamma\p{\displaystyle \frac{n}{2}+1}}z^{-\frac{1}{2}} \frac{\sqrt\pi \Gamma\p{\displaystyle \frac{n}{2}+1}}{\Gamma\p{\displaystyle \frac{n+1}{2}+1}} t^{\frac{n+1}{2}}\\
        &= \frac{C^{n+1}\pi^{\frac{n+1}{2}}}{\Gamma\p{\displaystyle \frac{n+1}{2}+1}}z^{-\frac{1}{2}} t^{\frac{n+1}{2}}.
    \end{align*}
    Therefore (\ref{eqn:unBdRefinedSTWN}) holds for all $n \geq 1$.
    Furthermore, it follows from a simple convergence test that the series
    \begin{equation}
        \sum_{n=0}^\infty \E[u_n^2(z, t)] \leq u_0^2(z,t) + \frac{1}{\sqrt z}\sum_{n=0}^\infty \frac{(C \sqrt \pi \sqrt t)^n}{\Gamma\p{\displaystyle \frac{n}{2}+1}} = u_0^2(z, t) + z^{-\frac{1}{2}}\phi(C \sqrt{\pi t})
    \end{equation}
    converges for any $z, t > 0$, where
    \begin{equation}
        \label{eqn:refinedSeriesValue}
        \phi(x) = e^{x^2}\p{1 + \frac{2}{\sqrt \pi}\int_0^x e^{-s^2}\mathrm ds}.
    \end{equation}

    Continuing to general $p \geq 2$, (\ref{eqn:unPthMomentBdSTWN}) follows immediately by (\ref{eqn:lpInequality}) and (\ref{eqn:unBdRefinedSTWN}).
    Again, by a convergence test, for every $z, t > 0$,
    \[
        \sum_{n=0}^{\infty} \p{\E[|u_n(z,t)^p|]}^{\frac{1}{p}} < \infty,
    \]
    which implies that $u(z,t) \in L^p(\Omega)$ and that $u(z,t) = \sum_{n=0}^\infty u_n(z,t)$ in $L^p(\Omega)$.
\end{proof}

\subsubsection{Colored Noise}
\label{ssec:momentColored}

With $\mathsection \ref{ssec:momentSTWN}$ and $\mathsection \ref{ssec:euColored}$, we almost have all the necessary tools to prove a version of Theorem \ref{thm:uPthMomentBdSTWN} corresponding to colored noise kernels $\gamma$ and $f$.
Notice that by (\ref{eqn:lpInequality}) and (\ref{eqn:unBdColored}),
\begin{align}
    \raisetag{-9.5mm} \hspace{10mm} \label{eqn:uUnrefinedSeriesColored} &\sum_{n=0}^\infty \E[|u_n(z,t)|^p]^{\frac{1}{p}} \leq u_0(z,t) \\
    & + \sum_{n=1}^\infty \p{\frac{\sqrt{p-1}\Gamma_tF_\varepsilon}{2}}^{\frac{n}{2}} \p{1 + \sum_{m=1}^n \sum_{k=0}^{m-1}\binom{m-1}{k}\frac{1}{(k+1)!}\p{\frac{2f(\varepsilon)t}{F_\varepsilon}}^{k+1}}^{\frac{1}{2}}. \notag 
\end{align}
The only problem with convergence stems from the geometric series, while the rest converges for some $\varepsilon > 0$ sufficiently small, depending on $z, t > 0$ and $p \geq 2$.

Recall that this bound was obtained by recursively applying in (\ref{eqn:recursiveDegenNormGen}) the uniform bound in Lemma \ref{lem:fsSqInt}, of which we now have the refined version (\ref{eqn:u1BdRefinedSTWN}).
Hence, we can now incorporate (\ref{eqn:u1BdRefinedSTWN}) to prove convergence of (\ref{eqn:uUnrefinedSeriesColored}).
In addition, we can show that it is possible to obtain convergence for all $z,t > 0$, independent of the choice of $\varepsilon > 0$.

\begin{theorem}
    \label{thm:uPthMomentBdColored}
    For every $n \geq 1$, $p \geq 2$, $\beta \geq 1/4$, and $z, t > 0$, if $f$ satisfies conditions (\textbf{i}) and (\textbf{ii}), then we have
    \begin{align}
        \mathbb{E}\left[\left|u_{n}\left(z,t\right)\right|^{p}\right]^{\frac{1}{p}} &\leq\hat{z}^{-\frac{1}{4}}\frac{\left(\left(p-1\right)\Gamma_{t}F_{\epsilon}C\sqrt{\pi t}\right)^{\frac{n}{2}}}{\sqrt{\Gamma\left(\oldfrac{n}{2}+1\right)}} \notag \\
        &\hspace{15mm} \cdot\left[1+\sum_{m=1}^{n}\sum_{k=0}^{m-1}\binom{m-1}{k}\left(\frac{f\left(\epsilon\right)\sqrt{t}}{F_{\epsilon}C\sqrt{\pi}}\right)^{k+1}\frac{\Gamma\left(\oldfrac{n}{2}+1\right)}{\Gamma\left(\oldfrac{n+k+1}{2}+1\right)}\right]^{\frac{1}{2}}.  \label{eqn:uPthMomentBdColored}
    \end{align}
    Finally, for every $z, t > 0$ we have $u(z,t) \in L^p(\Omega)$ and the series in (\ref{eqn:solnWienerChaosExpansion}) converges to $u(z,t)$ in $L^p(\Omega)$.
\end{theorem}
\begin{proof}
    Let $z_{1},z_{2},t_{1},t_{2},\epsilon>0$. Note that 
    \[
        \max\p{\left(z_{1}z_{2}\right)^{-\frac{1}{4}}, 1} \leq \left(\hat{z_{1}}\hat{z_{2}}\right)^{-\frac{1}{4}}.
    \]
    We claim that for every $n\geq1$, 
    \begin{align*}
        \mathbb{E} &\left[u_{n}\left(z_{1},t_{1}\right) u_{n}\left(z_{2},t_{2}\right)\right] \leq \left(\hat{z_{1}}\hat{z_{2}}\right)^{-\frac{1}{4}}\left(\Gamma_{t_{1}\vee t_{2}}F_{\epsilon}C\sqrt{\pi}\right)^{n}\\
        & \hspace{25mm} \cdot \left[\oldfrac{\left(t_{1}t_{2}\right)^{\frac{n}{4}}}{\Gamma\left(\oldfrac{n}{2}+1\right)}+\sum_{m=1}^{n}\sum_{k=0}^{m-1}\binom{m-1}{k}\left(\frac{f\left(\epsilon\right)}{F_{\epsilon}C\sqrt{\pi}}\right)^{k+1}\frac{\left(t_{1}t_{2}\right)^{\oldfrac{n+k+1}{4}}}{\Gamma\left(\oldfrac{n+k+1}{2}+1\right)}\right].
    \end{align*}
    First, we examine the base case when $n = 1$:
    \begin{align*}
        &\mathbb{E} \left[u_{1}\left(z_{1},t_{1}\right) u_{1}\left(z_{2},t_{2}\right)\right] \leq \int_{\left[0,t_{1}\right]\times\left[0,t_{2}\right]} \hspace{-14mm}\gamma\left(r-s\right)\left(\int_{A_{\epsilon}}+\int_{A_{\epsilon}^{\complement}}\right) f\left(x-y\right)\hat x^\beta y^\beta \\
        &\hspace{10mm} \cdot q_0(z_1,x,t_1-r)q_0(z_2,y,t_2-s)\E\b{u_{n-1}(x,r)u_{n-1}(y,s)} dxdydrds\\
        & \leq\Gamma_{t_{1}\vee t_{2}}F_{\epsilon}\sqrt{\int_{0}^{t_{1}}\int_{0}^{\infty}\hat x^{2\beta}q_{0}^{2}\left(z_{1},x,t_{1}-r\right)dxdr\int_{0}^{t_{2}}\int_{0}^{\infty}\hat y^{2\beta}q_{0}^{2}\left(z_{2},y,t_{2}-s\right)dyds}\\
        &+ f\left(\epsilon\right) \Gamma_{t_{1}\vee t_{2}}\sqrt{\int_{0}^{t_{1}} \left|\int_{0}^{\infty}\hat x^\beta q_{0}\left(z_{1},x,t_{1}-r\right)dx\right|^{2}\hspace{-2mm}dr \int_{0}^{t_{2}}\left|\int_{0}^{\infty} \hat y^\beta q_{0}\left(z_{2},y,t_{2}-s\right)dy\right|^{2}ds}\\
        & \leq\Gamma_{t_{1}\vee t_{2}}F_{\epsilon}\frac{C\sqrt{\pi}}{\Gamma\left(3/2\right)}\left(z_{1}z_{2}\right)^{-\frac{1}{4}}\left(t_{1}t_{2}\right)^{\frac{1}{4}}+\Gamma_{t_{1}\vee t_{2}}f\left(\epsilon\right)\sqrt{t_{1}t_{2}}\\
        & \leq\Gamma_{t_{1}\vee t_{2}}F_{\epsilon}C\sqrt{\pi}\left(\hat{z_{1}}\hat{z_{2}}\right)^{-\frac{1}{4}}\left[\frac{\left(t_{1}t_{2}\right)^{\frac{1}{4}}}{\Gamma\left(3/2\right)}+\frac{f\left(\epsilon\right)}{F_{\epsilon}C\sqrt{\pi}}\frac{\left(t_{1}t_{2}\right)^{\frac{1}{2}}}{\Gamma\left(2\right)}\right],
    \end{align*}
    which matches the desired expression. Now assume that the expression
    holds for $n-1$ for some $n\geq2$. 
    \begin{align*}
        \mathbb{E}\left[u_{n}\left(z_{1},t_{1}\right)u_{n}\left(z_{2},t_{2}\right)\right] =\int_{\left[0,t_{1}\right]\times\left[0,t_{2}\right]}\hspace{-14mm} \gamma\left(r-s\right)\int_{\mathbb{R}_{+}^{2}} \hspace{-3.5mm}f\left(x-y\right)&\hat{x}^{\beta}\hat{y}^{\beta} q_{0}\left(z_{1},x,t_{1}-r\right)q_{0}\left(z_{2},y,t_{2}-s\right)\\
         & \cdot \mathbb{E}\left[u_{n-1}\left(x,r\right)u_{n-1}\left(y,s\right)\right]dxdydrds.
    \end{align*}
    Similarly, we split the integral above in $\left(x,y\right)$ into
    $A_{\epsilon}$ and $A_{\epsilon}^{\complement}$, and treat each
    term separately. 
    First, denoting
    \begin{equation}
        U_n(z,t) := \int_0^t r^{n}\int_0^\infty q_0^2(z, x, t-r)\mathrm dx\mathrm dr.
    \end{equation}
    Then, by the induction assumption and (\ref{eqn:u1BdSpatialRefinedSTWN}),
    \begin{align*}
     & \int_{\left[0,t_{1}\right]\times\left[0,t_{2}\right]}\hspace{-14mm}\gamma\left(r-s\right)\int_{A_{\epsilon}}\hspace{-3.5mm}f\left(x-y\right)\hat{x}^{\beta}\hat{y}^{\beta} q_{0}\left(z_{1},x,t_{1}-r\right)q_{0}\left(z_{2},y,t_{2}-s\right) \\
     &\hspace{33mm} \cdot \mathbb{E}\left[u_{n-1}\left(x,r\right)u_{n-1}\left(y,s\right)\right]dxdydrds\\
    \leq & \left(\Gamma_{t_{1}\vee t_{2}}F_{\epsilon}C\sqrt{\pi}\right)^{n-1}\int_{\left[0,t_{1}\right]\times\left[0,t_{2}\right]}\hspace{-12mm}\gamma\left(r-s\right)\int_{A_{\epsilon}}f\left(x-y\right)\hat{x}^{\beta-\frac{1}{4}}\hat{y}^{\beta-\frac{1}{4}} q_{0}\left(z_{1},x,t_{1}-r\right)q_{0}\left(z_{2},y,t_{2}-s\right)\\
     &\qquad \cdot \left[\frac{\left(rs\right)^{\frac{n-1}{4}}}{\Gamma\left(\oldfrac{n-1}{2}+1\right)}+\sum_{m=1}^{n-1}\sum_{k=0}^{m-1}\binom{m-1}{k}\left(\frac{f\left(\epsilon\right)}{F_{\epsilon}C\sqrt{\pi}}\right)^{k+1}\frac{\left(rs\right)^{\frac{n+k}{4}}}{\Gamma\left(\oldfrac{n+k}{2}+1\right)}\right]dxdydrds\\
    \leq & \left(C\sqrt{\pi}\right)^{n-1}\left(\Gamma_{t_{1}\vee t_{2}}F_{\epsilon}\right)^{n}\left\{ \frac{\left(U_{\frac{n-1}{2}}(z_1,t_1)\right)^{\frac{1}{2}}\left(U_{\frac{n-1}{2}}(z_2,t_2)\right)^{\frac{1}{2}}}{\Gamma\left(\oldfrac{n-1}{2}+1\right)}\right.\\
     &\left. + \sum_{m=1}^{n-1}\sum_{k=0}^{m-1}\binom{m-1}{k}\left(\frac{f\left(\epsilon\right)}{F_{\epsilon}C\sqrt{\pi}}\right)^{k+1}\frac{\left(U_{\frac{n+k}{2}}(z_1,t_1)\right)^{\frac{1}{2}}\left(U_{\frac{n+k}{2}}(z_2,t_2)\right)^{\frac{1}{2}}}{\Gamma\left(\oldfrac{n+k}{2}+1\right)}\right\} \\
    \leq & \left(\Gamma_{t_{1}\vee t_{2}}F_{\epsilon}C\sqrt{\pi}\right)^{n}\left(\hat{z_{1}}\hat{z_{2}}\right)^{-\frac{1}{4}}\left[\frac{\left(t_{1}t_{2}\right)^{\frac{n}{4}}}{\Gamma\left(\oldfrac{n}{2}+1\right)} + \sum_{m=1}^{n-1}\sum_{k=0}^{m-1}\binom{m-1}{k}\left(\frac{f\left(\epsilon\right)}{F_{\epsilon}C\sqrt{\pi}}\right)^{k+1}\frac{\left(t_{1}t_{2}\right)^{\frac{n+k+1}{4}}}{\Gamma\left(\oldfrac{n+k+1}{2}+1\right)}\right].
    \end{align*}
    Next, again by the induction assumption, 
    \begin{align*}
     & \int_{\left[0,t_{1}\right]\times\left[0,t_{2}\right]}\hspace{-14mm}\gamma\left(r-s\right)\int_{A_{\epsilon}^{\complement}}\hspace{-3.5mm}f\left(x-y\right)\hat{x}^{\beta}\hat{y}^{\beta} q_{0}\left(z_{1},x,t_{1}-r\right)q_{0}\left(z_{2},y,t_{2}-s\right) \\
     &\hspace{33mm} \cdot \mathbb{E}\left[u_{n-1}\left(x,r\right)u_{n-1}\left(y,s\right)\right]dxdydrds\\
    \leq & \left(\Gamma_{t_{1}\vee t_{2}}F_{\epsilon}C\sqrt{\pi}\right)^{n-1}\int_{\left[0,t_{1}\right]\times\left[0,t_{2}\right]}\hspace{-14mm}\gamma\left(r-s\right)\int_{A_{\epsilon}^{\complement}}\hspace{-3.5mm}f\left(x-y\right)\hat{x}^{\beta-\frac{1}{4}}\hat{y}^{\beta-\frac{1}{4}} q_{0}\left(z_{1},x,t_{1}-r\right)q_{0}\left(z_{2},y,t_{2}-s\right) \\
     & \qquad \cdot \left[\frac{\left(rs\right)^{\frac{n-1}{4}}}{\Gamma\left(\oldfrac{n-1}{2}+1\right)}+\sum_{m=1}^{n-1}\sum_{k=0}^{m-1}\binom{m-1}{k}\left(\frac{f\left(\epsilon\right)}{F_{\epsilon}C\sqrt{\pi}}\right)^{k+1}\frac{\left(rs\right)^{\frac{k+n}{4}}}{\Gamma\left(\oldfrac{n+k}{2}+1\right)}\right]dxdydrds\\
    \leq & \left(\Gamma_{t_{1}\vee t_{2}}F_{\epsilon}C\sqrt{\pi}\right)^{n}\left\{ \frac{f\left(\epsilon\right)}{F_{\epsilon}C\sqrt{\pi}}\frac{\left(\int_{0}^{t_{1}}r^{\frac{n-1}{2}}dr\right)^{\frac{1}{2}}\left(\int_{0}^{t_{2}}s^{\frac{n-1}{2}}ds\right)^{\frac{1}{2}}}{\Gamma\left(\oldfrac{n-1}{2}+1\right)}+\right.\\
     & \qquad\left.\sum_{m=1}^{n-1}\sum_{k=0}^{m-1}\binom{m-1}{k}\left(\frac{f\left(\epsilon\right)}{F_{\epsilon}C\sqrt{\pi}}\right)^{k+2}\frac{\left(\int_{0}^{t_{1}}r^{\frac{n+k}{2}}dr\right)^{\frac{1}{2}}\left(\int_{0}^{t_{2}}s^{\frac{n+k}{2}}ds\right)^{\frac{1}{2}}}{\Gamma\left(\oldfrac{n+k}{2}+1\right)}\right\} \\
    \leq & \left(\Gamma_{t_{1}\vee t_{2}}F_{\epsilon}C\sqrt{\pi}\right)^{n}\left(\hat{z_{1}}\hat{z_{2}}\right)^{-\frac{1}{4}}\left[\frac{f\left(\epsilon\right)}{F_{\epsilon}C\sqrt{\pi}}\frac{\left(t_{1}t_{2}\right)^{\frac{n+1}{4}}}{\Gamma\left(\oldfrac{n+1}{2}+1\right)} \right. \\
    &\hspace{39mm} + \left. \sum_{m=1}^{n-1}\sum_{k=0}^{m-1}\binom{m-1}{k}\left(\frac{f\left(\epsilon\right)}{F_{\epsilon}C\sqrt{\pi}}\right)^{k+2}\frac{\left(t_{1}t_{2}\right)^{\frac{n+k+2}{4}}}{\Gamma\left(\oldfrac{n+k+2}{2}+1\right)}\right].
    \end{align*}
    Adding the two expressions above and regrouping terms in the summation,
    we confirm that $\mathbb{E}\left[u_{n}\left(z_{1},t_{1}\right)u_{n}\left(z_{2},t_{2}\right)\right]$
    takes the desired form. 
    
    Moreover, we get that for every $z,t>0$ and $\epsilon>0$, 
    \begin{align*}
    \mathbb{E}\left[u_{n}^{2}\left(z,t\right)\right] & \leq\hat{z}^{-\frac{1}{2}}\left(\Gamma_{t}F_{\epsilon}C\sqrt{\pi}\right)^{n}\cdot\\
     & \qquad\qquad\left[\frac{t^{\frac{n}{2}}}{\Gamma\left(\oldfrac{n}{2}+1\right)}+\sum_{m=1}^{n}\sum_{k=0}^{m-1}\binom{m-1}{k}\left(\frac{f\left(\epsilon\right)}{F_{\epsilon}C\sqrt{\pi}}\right)^{k+1}\frac{t^{\frac{n+k+1}{2}}}{\Gamma\left(\oldfrac{n+k+1}{2}+1\right)}\right]\\
     & \leq\hat{z}^{-\frac{1}{2}}\frac{\left(\Gamma_{t}F_{\epsilon}C\sqrt{\pi t}\right)^{n}}{\Gamma\left(\oldfrac{n}{2}+1\right)}\left[1+\sum_{m=1}^{n}\sum_{k=0}^{m-1}\binom{m-1}{k}\left(\frac{f\left(\epsilon\right)\sqrt{t}}{F_{\epsilon}C\sqrt{\pi}}\right)^{k+1}\frac{\Gamma\left(\oldfrac{n}{2}+1\right)}{\Gamma\left(\oldfrac{n+k+1}{2}+1\right)}\right],
    \end{align*}
    and hence for all $p \geq 2$, 
    \begin{align*}
        \mathbb{E}\left[\left|u_{n}\left(z,t\right)\right|^{p}\right]^{\frac{1}{p}} \leq \hat{z}^{-\frac{1}{4}}&\frac{\left(\left(p-1\right)\Gamma_{t}F_{\epsilon}C\sqrt{\pi t}\right)^{\frac{n}{2}}}{\sqrt{\Gamma\left(\oldfrac{n}{2}+1\right)}} \\
        &\cdot \left[1+\sum_{m=1}^{n}\sum_{k=0}^{m-1}\binom{m-1}{k}\left(\frac{f\left(\epsilon\right)\sqrt{t}}{F_{\epsilon}C\sqrt{\pi}}\right)^{k+1}\frac{\Gamma\left(\oldfrac{n}{2}+1\right)}{\Gamma\left(\oldfrac{n+k+1}{2}+1\right)}\right]^{\frac{1}{2}}.
    \end{align*}
    Since
    \[
        \frac{\Gamma\left(\oldfrac{n}{2}+1\right)}{\Gamma\left(\oldfrac{n+k+1}{2}+1\right)} \leq \frac{2}{\Gamma\left(\oldfrac{k+1}{2}\right)},
    \]
    it is straightforward to verify that for all $z,t>0$ and $p\geq2$,
    \[
        \sum_{n=1}^{\infty}\mathbb{E}\left[\left|u_{n}\left(z,t\right)\right|^{p}\right]^{\frac{1}{p}}<\infty.
    \]
\end{proof}

\subsection{H\"older Continuity up to Boundary}
\label{sec:cont}

As mentioned in $\mathsection\ref{ssec:outline}$, now that we have bounds for $L^p-$ norms for $p \geq 2$ and refined bounds for $L^2$ norms, we continue to establish H\"older continuity of the solution to (\ref{eqn:SKE2}).
For $z_1, z_2, t_1, t_2 > 0$, similarly to the derivation of (\ref{eqn:fnKernelsDefn}), the process $u(z_1, t_1) - u(z_2, t_2) \in L^2(\Omega)$ and admits the following Wiener chaos expansion
\begin{equation}
    \label{eqn:contWienerChaosExpansion}
    \E\b{\abs{u(z_1, t_1) - u(z_2, t_2)}^2} = \sum_{n=0}^\infty \mathcal I_n\p{f_n(\cdot; z_1, z_2, t_1, t_2)},
\end{equation}
where we abuse the notation to denote $f_0(z_1, z_2, t_1, t_2) := u_0(z_1,t_1) - u_0(z_2,t_2)$ and
\begin{equation}
    \label{eqn:contKernelsDefn}
    f_n(\cdot; z_1, z_2, t_1, t_2) := f_n(\cdot; z_1, t_1) - f_n(\cdot; z_2, t_2)
\end{equation}
for $f_n(\cdot; z_1, t_1)$ and $f_n(\cdot; z_2, t_2)$ as in (\ref{eqn:fnKernelsDefn}).
Hence, setting 
\begin{equation}
    \label{eqn:fsDiffDefn}
    d_0(z_1, z_2, w, t) := q_0(z_1, w, t) - q_0(z_1, w, t), 
\end{equation}
for $z_1, z_2, w, t > 0$, then for $n \geq 1$ and $t_1 = t_2 = t > 0$,
\begin{align}
    \E\b{\abs{u_n(z_1, t) - u_n(z_2, t)}^2} = \int_{[0,t]^2} \hspace{-6mm} \gamma(r-s)\int_{\R_+^2} & \hspace{-3mm} f(x-y)\hat x^\beta \hat y^\beta d_0(z_1,z_2,x,t-r)d_0(z_1,z_2,y,t-s) \notag \\
    &\cdot \E[u_{n-1}(x,r)u_{n-1}(y,s)]\mathrm dx \mathrm dy\mathrm dr\mathrm ds. \label{eqn:recursiveDegenContNormGen}
\end{align}

Thus we can reuse the refined bounds developed in $\mathsection\ref{sec:moment}$.
We shall show that for $\beta > 1/4$, the map $z \mapsto u(z, t)$ is H\"older continuous in $L^p(\Omega)$ for any fixed $t > 0$.

\subsubsection{Space-Time White Noise}
\label{ssec:contSTWN}

In order to show the H\"older continuity, we seek to bound $\E\b{\abs{u_n(z_1, t) - u_n(z_2, t)}^2}$ by a power of $|z_1 - z_2|$.
We begin by first examining the case of space-time white noise.
In this case, (\ref{eqn:recursiveDegenContNormGen}) becomes
\begin{align}
    \label{eqn:recursiveDegenContNormSTWN} \raisetag{-7mm}\hspace{10mm}
    \E\b{\abs{u_n(z_1, t) - u_n(z_2, t)}^2} = \int_0^t \int_0^\infty \hat w^{2\beta} d_0^2(z_1,z_2,w,t-\tau)\E[u_{n-1}^2(w,\tau)] \mathrm dw \mathrm d\tau,
\end{align}

In light of (\ref{eqn:unBdRefinedSTWN}), for $n \geq 2$,
\[
    \E\b{\abs{u_n(z_1, t) - u_n(z_2, t)}^2} \leq \frac{(C\sqrt{\pi t})^{n-1}}{\displaystyle \Gamma\p{\frac{n+1}{2}}} \int_0^t \int_0^\infty \hat w^{2\beta}w^{-\frac{1}{2}} d_0^2(z_1,z_2,w,t-\tau)\mathrm dw \mathrm d\tau.
\]
So, no recursive argument is needed.
Instead, we need only focus on bounding
\begin{equation}
    \label{eqn:fsDiffSqIntDefn}
        Q_t(z_1, z_2) := \int_0^t \int_0^\infty \hat w^{2\beta}w^{-\frac{1}{2}}d_0^2(z_1,z_2,w,t-\tau)\mathrm dw \mathrm d\tau.
\end{equation}

To this end, we first present some technical results on the H\"older continuity of the map $z \mapsto q_0(z,w,t-\tau)$ for $w, t-\tau, \tau > 0$.
\begin{lemma}
    \label{lem:fsHolderCont}
    There exists a constant $M > 0$ such that for every $z_1, z_2, w, t-\tau, \tau > 0$, 
    \begin{equation}
        \label{eqn:fsHolderCont}
        \abs{d_0(z_1,z_2,w,t-\tau)} \leq M \frac{|z_1-z_2|^{\frac{1}{2}}}{(t-\tau)^{\frac{3}{2}}}.
    \end{equation}
\end{lemma}
\begin{proof}
    By direct computation, we have that
    \[
        \partial_z q_0(z, w, t-\tau) = \frac{1}{t-\tau}\p{q_1(z, w, t-\tau) - q_0(z,w,t-\tau)},
    \]
    where $q_1$ is given by (\ref{eqn:fsDriftDefn}), with $\nu = 1$.
    Hence,
    \begin{align*}
        \abs{d_0(z_1,z_2,w,t-\tau)} &= \abs{\int_{z_1}^{z_2} \partial_z q_0(z, w, t-\tau) \mathrm dz}\\
        &= \abs{\frac{1}{t-\tau}\int_{z_1}^{z_2} q_1(z, w, t) - q_0(z,w,t) \mathrm dz}\\
        &\leq \frac{|z_1-z_2|^{\frac{1}{2}}}{t-\tau} \b{\int_0^\infty \p{q_1^2(z,w,t-\tau) + q_0^2(z,w,t-\tau)}\mathrm dz}^{\frac{1}{2}}.
    \end{align*}
    Again, by direct computation, we have
    \begin{align*}
        \int_0^\infty q_1^2(z,w,t-\tau)\mathrm dz &= \frac{e^{-\frac{2w}{t-\tau}}}{2(t-\tau)} 
        {}_2F_2
        \left[
        \hspace{-1mm}
        \begin{array}{cc}
            \displaystyle \frac{1}{2}, 1  \\
            1, 1
        \end{array} \hspace{-2mm}
        \mathrel{\Bigg|} \frac{2w}{t-\tau}
        \right], \\
        \int_0^\infty q_0^2(z,w,t-\tau)\mathrm dz &= \frac{e^{-\frac{2w}{t-\tau}}}{4(t-\tau)} {}_2F_2
        \left[
        \hspace{-1mm}
        \begin{array}{cc}
            \displaystyle \frac{3}{2}, 3  \\
            2, 3
        \end{array} \hspace{-2mm}
        \mathrel{\Bigg|} \frac{2w}{t-\tau}
        \right].
    \end{align*}
    Combining these last two computations with Proposition \ref{prop:2F2AsymptoticExpansion} yields
    \[
        \abs{d_0(z_1,z_2,w,t-\tau)} \leq \frac{|z_1-z_2|^{\frac{1}{2}}}{t-\tau} \cdot \frac{M}{\sqrt{t-\tau}} = M \frac{|z_1-z_2|^{\frac{1}{2}}}{(t-\tau)^{\frac{3}{2}}},
    \]
    as claimed.
\end{proof}

Therefore, we observe that it is possible to extrapolate the H\"older continuity of the stochastic solution from the H\"older continuity of the fundamental solution to the deterministic, unperturbed problem.
It thus remains to treat the aftermath of substituting (\ref{eqn:fsHolderCont}) into (\ref{eqn:fsDiffSqIntDefn}) in the hopes of obtaining a bound independent of $z_1$ and $z_2$.

\begin{lemma}
    \label{lem:fsDiffsqInt}
    Fix $\beta > 1/4$ and $t > 0$. 
    For $z_1, z_2 > 0$ and any $0 \leq \lambda < (4\beta-1)\wedge1/2$, there exists a constant $M_{\beta,\lambda} > 0$ such that
    \begin{equation}
        \label{eqn:fsDiffSqIntBd}
        Q_t(z_1, z_2) \leq M_{\beta, \lambda} t^\eta |z_1-z_2|^{\frac{\lambda}{2}},
    \end{equation}
    where
    \begin{equation}
        \label{eqn:fsDiffSqIntEta}
        \eta = \b{\p{2\beta - \frac{1}{2}} \wedge \frac{1}{4}} - \frac{\lambda}{2}
    \end{equation}
\end{lemma}
\begin{proof}
    By Lemma \ref{lem:fsHolderCont}, for any $\lambda \in (0, 2)$, we have
    \begin{align*}
        Q_t(z_1, z_2) &\leq M_\lambda |z_1-z_2|^{\frac{\lambda}{2}} \int_0^t \frac{1}{(t-\tau)^{\frac{3\lambda}{2}}} \int_0^\infty \hat w^{2\beta} w^{-\frac{1}{2}}\abs{d_0(z_1,z_2,w,t-\tau)}^{2-\lambda}\mathrm dw\mathrm d\tau, \\
        &\leq M_\lambda |z_1-z_2|^{\frac{\lambda}{2}} \\
        & \hspace{3mm} \cdot \int_0^t \hspace{-1.5mm}\frac{1}{(t-\tau)^{\frac{3\lambda}{2}}} \int_0^\infty \hspace{-1.5mm}\hat w^{2\beta-\frac{1}{2}}\b{q_0^{2-\lambda}(z_1,w,t-\tau) + q_0^{2-\lambda}(z_2, w, t-\tau)}\mathrm dw\mathrm d\tau.
    \end{align*}
    It remains to treat, for any $z > 0$,
    \begin{equation}
        \label{eqn:fsPowerInt}
        Q := \int_0^t \frac{1}{(t-\tau)^{\frac{3\lambda}{2}}}\int_0^\infty q_0^{2-\lambda}(z,w,t-\tau)\hat w^{2\beta-1/2}\mathrm dw\mathrm d\tau.
    \end{equation}
    Note that $\hat w^{2\beta-\frac{1}{2}} \leq \hat w^{(2\beta-\frac{1}{2}) \wedge \frac{1}{4}} \leq w^{(2\beta-\frac{1}{2}) \wedge \frac{1}{4}}$.
    We apply the monotonicity of the moments to bound (\ref{eqn:fsPowerInt}) from above as
    \begin{align*}
        Q &\leq \int_0^t \frac{1}{(t-\tau)^{\frac{3\lambda}{2}}} \int_0^\infty w^{(2\beta-\frac{1}{2}) \wedge \frac{1}{4}} q_0^{2-\lambda}(z,w,t-\tau)\mathrm dw\mathrm d\tau\\
        &= \int_0^t \frac{1}{(t-\tau)^{\frac{3\lambda}{2}}} \int_0^\infty w^{(2\beta-\frac{1}{2}) \wedge \frac{1}{4}} q_0^{1-\lambda}(z,w,t-\tau)(q_0(z,w,t-\tau)\mathrm dw)\mathrm d\tau\\
        &\leq \int_0^t \frac{1}{(t-\tau)^{\frac{3\lambda}{2}}} \p{\int_0^\infty w^{\frac{(2\beta-\frac{1}{2}) \wedge \frac{1}{4}}{1-\lambda}} q_0^2(z,w,t-\tau)\mathrm dw}^{1-\lambda}\mathrm d\tau.
    \end{align*}
    Set $\xi := \frac{((2\beta-\frac{1}{2}) \wedge \frac{1}{4})}{(1-\lambda)}$.
    Then, by Proposition \ref{prop:2F2AsymptoticExpansion}, the integral inside the power above can be bounded as
    \begin{align*}
        \int_{0}^{\infty}q_{0}^{2}\left(z,w,t-\tau\right)w^{\xi}dw & =\frac{\Gamma\left(1+\xi\right)}{2^{3+\xi}}\left(t-\tau\right)^{\xi-1}\left(\frac{2z}{t-\tau}\right)^{2}e^{-\frac{2z}{t-\tau}}
        {}_2F_2
        \left[
        \hspace{-1mm}
        \begin{array}{cc}
            \displaystyle \frac{3}{2}, \xi+1  \\
            2, 3
        \end{array} \hspace{-2mm}
        \mathrel{\Bigg|} \frac{2z}{t-\tau}
        \right]\\
        &\leq M_{\beta, \lambda}\left(t-\tau\right)^{\xi-1},
    \end{align*}
    provided that $\frac{3}{2}+\xi+1-2-3\leq-2$, i.e., when 
    \begin{equation}
        \label{eqn:range of gamma}
        \lambda \leq\begin{cases}
        2-4\beta & \text{ if }\frac{1}{4}<\beta\leq\frac{3}{8},\\
        \frac{1}{2} & \text{ if }\beta>\frac{3}{8}.
        \end{cases}
    \end{equation}
    It follows from above that
    \[  
        Q_t(z_1, z_2) \leq M_{\beta, \lambda} |z_1-z_2|^{\frac{\lambda}{2}} \int_0^t (t-\tau)^{(\xi - 1)(1-\lambda)-\frac{3}{2}\lambda} \mathrm d\tau = M_{\beta, \lambda} t^{\eta} |z_1-z_2|^{\frac{\lambda}{2}}.
    \]
    Note that by taking $0 \leq \lambda < (4\beta-1) \wedge 1/2$, we always have $\eta > 0$ and (\ref{eqn:range of gamma}).
\end{proof}

We are finally equipped to present the main result of this section.

\begin{theorem}
    Fix $\beta > 1/4$ and $t > 0$.
    Then, for any $Z > 0$ and any $\theta \in (0, (\beta - 1/4) \wedge 1/8)$, there exists an a.s. $\theta-$H\"older continuous modification of the process $z \in [0,Z] \mapsto u(z, t) \in \R$.
    In particular, $z \mapsto u(z,t)$ is continuous all the way to the boundary 0.
\end{theorem}
\begin{proof}
    For any $p \geq 2$ and $n \geq 1$, by (\ref{eqn:lpInequality}) and Lemma \ref{lem:fsDiffsqInt}, for some constant $M_{p, \beta, \gamma} > 0$,
    \begin{align}
        \p{\E\b{|u_n(z_1,t) - u_n(z_2,t)|^p}}^{1/p} &\leq (p-1)^{\frac{n}{2}} \p{\E\b{|u_n(z_1,t)-u_n(z_2,t)|^2}}^{\frac{1}{2}}, \notag \\
        &\leq M_{p, \beta, \lambda} \frac{\p{(p-1)C \sqrt{\pi t}}^{\frac{n-1}{2}}}{\sqrt{\Gamma\p{\displaystyle \frac{n+1}{2}}}} |z_1-z_2|^{\frac{\lambda}{4}}t^{\frac{\eta}{2}}. \label{eqn:unHolderContBdSTWN}
    \end{align}
    Substituting (\ref{eqn:unHolderContBdSTWN}) into (\ref{eqn:contWienerChaosExpansion}), 
    \begin{align*}
        &\p{\E\b{|u(z_1,t) - u(z_2,t)|^p}}^{1/p} \\
        &\hspace{25mm} \leq \abs{u_0(z_1,t) - u_0(z_2, t)} + \sum_{n=1}^\infty M_{p, \beta, \lambda} \frac{(C \sqrt t)^{n-1}}{\sqrt{\Gamma\p{\displaystyle \frac{n+1}{2}}}} |z_1-z_2|^{\lambda/4}t^{\eta/2}\\
        &\hspace{25mm} = \abs{e^{-\frac{z_1}{t}} - e^{-\frac{z_2}{t}}} + M_{p,\beta,\lambda, t} |z_1-z_2|^{\frac{\lambda}{4}}.
    \end{align*}
    When $p$ is sufficiently large, Kolmogorov's continuity theorem implies that for every $t > 0$, $z \mapsto u(z,t)$ is almost surely $\theta$-H\"older continuous for any $\theta \in \p{0, \frac{\lambda}{4} - \frac{1}{p}}$.
    Since $p$ can be arbitrarily large and $\lambda$ can be arbitrarily close to $(4\beta-1)\wedge\frac{1}{2}$, we conclude that for every $t > 0$, $z \mapsto u(z,t)$ is almost surely $\theta$-H\"older continuous for any $\theta \in (0, (\beta- \frac{1}{4}) \wedge \frac{1}{8})$ all the way to the boundary 0.
\end{proof}

\subsubsection{Colored Noise}
\label{ssec:contColored}

As before, we seek to apply the usual approach.
Recall from $\mathsection \ref{ssec:momentColored}$ that for $x, y > 0$, $r, s \in (0, t)$, and any $n \geq 1$, we have
\begin{align}
    \E[u_n(x,r)u_n(y,s)] \leq \hat x^{-\frac{1}{4}}&\hat y^{-\frac{1}{4}}\frac{(\Gamma_t F_\varepsilon C \sqrt{\pi t})^n}{\Gamma \p{\oldfrac{n}{2}+1}} \notag \\
    &\cdot \b{1 + \sum_{m=1}^n \sum_{k=0}^{m-1} \binom{m-1}{k} \p{\oldfrac{f(\varepsilon)\sqrt t}{F_\varepsilon C \sqrt \pi }}^{k+1} \frac{\Gamma\p{\oldfrac{n}{2}+1}}{\Gamma\p{\oldfrac{n+k+1}{2}+1}}} \label{eqn:unCovBdColored2}
\end{align}
which forms a convergent series for $\varepsilon > 0$ sufficiently small.
Let $M_{n, \varepsilon, t}$ denote the right hand side of (\ref{eqn:unCovBdColored2}) excluding the factor $\hat x^{-\frac{1}{4}}\hat y^{-\frac{1}{4}}$.
Similarly to the previous section, from (\ref{eqn:recursiveDegenContNormGen}) and Proposition \ref{prop:modifiedYoung}, we have
\begin{align}
    &\E \b{\abs{u_n(z_1, t) - u_n(z_2, t)}^2} \leq M_{n-1, \varepsilon, t} \int_{[0,t]^2} \hspace{-6mm} \gamma(r-s)\int_{\R_+^2} \hspace{-3mm} f(x-y)\hat x^{\beta-\frac{1}{4}} \hat y^{\beta-\frac{1}{4}} \notag \\
    &\hspace{50mm} \cdot d_0(z_1,z_2,x,t-r)d_0(z_1,z_2,y,t-s)\mathrm dx \mathrm dy\mathrm dr\mathrm ds, \notag \\
    &\leq \Gamma_t M_{n-1,\varepsilon,t} \left(F_\varepsilon \int_0^t \int_0^\infty \hat w^{2\beta-\frac{1}{2}} d_0^2(z_1,z_2,w,t-\tau)\mathrm dw\mathrm d\tau \right. \notag \\
    \raisetag{-8.5mm} &\left. \hspace{30mm} + f(\varepsilon) \int_0^t \abs{\int_0^\infty \hat w^{\beta-\frac{1}{4}} d_0(z_1,z_2,w,t-\tau)\mathrm dw}^2 \mathrm d\tau\right). \label{eqn:introContBdColored}
\end{align}
Again, no recursive argument is needed.
Notice that the first integral in the right hand side of (\ref{eqn:introContBdColored}) is bounded by $Q_t(z_1, z_2)$ from the previous section.
However, we need also extract a power of $|z_1-z_2|$ from the second integral.
The following result shows that we can obtain an exponent which is no worse for H\"older continuity than the one seen in the previous section.

\begin{lemma}
    \label{lem:holderContGenCov}
    Fix $\beta > 1/4$ and $t > 0$.
    For $z_1, z_2 > 0$, set
    \begin{equation}
        \label{eqn:holderContGenCovDefn}
        \tilde Q_t(z_1, z_2) := \int_0^t \abs{\int_0^\infty \hat w^{\beta-\frac{1}{4}} d_0(z_1,z_2,w,t-\tau) \mathrm dw}^2 \mathrm d\tau.
    \end{equation}
    Then, there exists a constant $M > 0$ such that for any $z_1, z_2 > 0$,
    \begin{equation}
        \label{eqn:holderContGenCovBd}
        \tilde Q_t(z_1, z_2) \leq M t^{\frac{3}{4}}|z_1 - z_2|^{\frac{1}{4}}.
    \end{equation}
\end{lemma}
\begin{proof}
    It's clear that
    \begin{equation}
        \label{eqn:unifDiffBd}
        \abs{\int_0^\infty \hat w^{\beta - \frac{1}{4}}d_0(z_1,z_2,w,t-\tau)\mathrm dw } \leq 2.
    \end{equation}
    Choose $p \in (0, 2)$.
    By (\ref{eqn:unifDiffBd}), we have
    \begin{align*}
        \tilde Q_t(z_1, z_2) &\leq 2^{2-p}\int_0^t \abs{\int_0^\infty \hat w^{\beta - \frac{1}{4}} d_0(z_1,z_2,w,t-\tau) \mathrm dw}^p \mathrm d\tau, \\
        &\leq 2^{2-p}\int_0^t \frac{1}{(t-\tau)^p} \p{ \int_{z_1}^{z_2} q_1(z,w,t-\tau) + q_0(z,w,t-\tau) \mathrm dw \mathrm dz}^p \mathrm d\tau.
    \end{align*}
    By direct computation, we observe
    \begin{align*}
        \int_0^\infty q_1(z,w,t-\tau) \mathrm dw &= \frac{1}{t-\tau} e^{-\frac{z}{t-\tau}} \sum_{n=0}^\infty \frac{z^n}{(t-\tau)^{2n}(n!)^2} \int_0^\infty e^{-\frac{w}{t-\tau}}w^n \mathrm dw\\
        &= e^{-\frac{z}{t-\tau}}\sum_{n=0}^\infty \frac{1}{n!}\p{\frac{z}{t-\tau}}^n = 1.
    \end{align*}
    Therefore together with Lemma \ref{lem:u0Int},
    \begin{align*}
        \tilde Q_t(z_1, z_2) &\leq 2^{2-p}\int_0^t \frac{1}{(t-\tau)^p} \p{\int_{z_1}^{z_2} \int_0^\infty q_1(z,w,t-\tau) \mathrm dw + \int_0^\infty q_0(z,w,t-\tau) \mathrm dw \mathrm dz}^p \mathrm d\tau\\
        &\leq 4 |z_1 - z_2|^p \int_0^t \frac{1}{(t-\tau)^p}\mathrm d\tau\\
        &= \frac{4}{1-p}|z_1-z_2|^p t^{1-p}.
    \end{align*}
    Taking $p = \frac{1}{4}$, we obtain the desired bound.
\end{proof}

\begin{theorem}
    \label{thm:contColored}
    Fix $\beta >  \frac{1}{4}$ and $t > 0$.
    Then, for any $Z > 0$ and $\theta \in (\beta - 1/4) \wedge 1/8$, there exists an a.s. $\theta$-H\"older continuous modification of the process $z \in [0, Z] \mapsto u(z, t) \in \R$.
    In particular, $z \in u(z,t)$ is continuous all the way to the boundary.
\end{theorem}
\begin{proof}
    For any $n \geq 1$ and $p = 2$, by substituting (\ref{eqn:fsDiffSqIntBd}) and (\ref{eqn:holderContGenCovBd}) into (\ref{eqn:introContBdColored}),
    \begin{align*}
        &\E \b{\abs{u_n(z_1, t) - u_n(z_2, t)}^2} \leq \Gamma_t M_{n-1,\varepsilon,t}\p{M_{\beta, \lambda} F_\varepsilon t^\eta |z_1-z_2|^{\frac{\lambda}{2}} +  Mf(\varepsilon)t^{\frac{3}{4}}|z_1 - z_2|^{\frac{1}{4}}},
    \end{align*}
    for $\eta$ and $\lambda$ as in Lemma \ref{lem:fsDiffsqInt}.
    Hence for $p \geq 2$, by (\ref{eqn:lpInequality}),
    \begin{align}
         \E &\b{\abs{u_n(z_1, t) - u_n(z_2, t)}^p}^{\frac{1}{p}} \notag \\
         &\hspace{10mm} \leq \sqrt{\Gamma_t M_{\beta, \lambda} (F_\varepsilon + f(\varepsilon))} (p-1)^{\frac{n}{2}} \sqrt{M_{n-1,\varepsilon,t}}\p{t^\eta |z_1-z_2|^{\frac{\lambda}{2}} + t^{\frac{3}{4}}|z_1-z_2|^{\frac{1}{4}}}^{\frac{1}{2}}
         \notag \\
         \raisetag{-6.5mm} &\hspace{10mm} \leq \sqrt{\Gamma_t M_{\beta, \lambda} (F_\varepsilon + f(\varepsilon))} (p-1)^{\frac{n}{2}} \sqrt{M_{n-1,\varepsilon,t}} \p{ t^{\frac{\eta}{2}}|z_1-z_2|^{\frac{\lambda}{4}} + t^{\frac{3}{8}}|z_1-z_2|^{\frac{1}{8}} }, \label{eqn:unHolderContBdColored}
    \end{align}
    where $(p-1)^{\frac{n}{2}}\sqrt{M_{n-1,\varepsilon, t}}$ is summable over $n$, independent of the values of the other variables.
    Substituting (\ref{eqn:unHolderContBdColored}) into (\ref{eqn:contWienerChaosExpansion}), we get
    \begin{align*}
        &\E \b{\abs{u(z_1, t) - u(z_2, t)}^p}^{\frac{1}{p}} \leq |u_0(z_1, t) - u_0(z_2, t)| \\
        &+ \sqrt{\Gamma_t M_{\beta, \lambda} (F_\varepsilon + f(\varepsilon))} \sum_{n=1}^\infty  (p-1)^{\frac{n}{2}}\sqrt{M_{n-1,\varepsilon,t}} \p{ t^{\frac{\eta}{2}}|z_1-z_2|^{\frac{\lambda}{4}} + t^{\frac{3}{8}}|z_1-z_2|^{\frac{1}{8}} }, \\
        &= |e^{-\frac{z_1}{t}} - e^{-\frac{z_2}{t}}| + M_{t, \varepsilon,\beta,\lambda} \p{ t^{\frac{\eta}{2}}|z_1-z_2|^{\frac{\lambda}{4}} + t^{\frac{3}{8}}|z_1-z_2|^{\frac{1}{8}} }.
    \end{align*}
    When $p$ is sufficiently large, Kolmogorov's continuity theorem implies that for every $t > 0$, $z \mapsto u(z,t)$ is almost surely $\theta$-H\"older continuous for any $\theta \in \p{0, \frac{\lambda}{4} - \frac{1}{p}}$.
    Since $p$ can be arbitrarily large and $\lambda$ can be arbitrarily close to $(4\beta-1)\wedge\frac{1}{2}$, we conclude that for every $t > 0$, $z \mapsto u(z,t)$ is almost surely $\theta$-H\"older continuous for any $\theta \in (0, (\beta- \frac{1}{4}) \wedge \frac{1}{8})$.
\end{proof}

\section{Conclusion and Further Questions} \label{sec:conclusion}
Beyond the results demonstrated above for existence, uniqueness, estimates, and continuity of the solution to (\ref{eqn:SKE1}), we are interested in exploring other commonly studied properties of SPDEs, such as long-term asymptotics and intermittency (see \cite{pamOG, Berger2023, intermittence, nualart, PAMCarmonaMolchanov}).

Moreover, while we have demonstrated results in the special case of a driftless ($\nu = 0$) stochastic Kimura equation, the deterministic problem (\ref{eqn:kimuraDrift}) studied in \cite{chenDEG} possesses a more general constant drift, $\nu < 1$.
At present, we have no cause to believe that the introduction of such a drift term would have a qualitative impact on existence, uniqueness, estimates or continuity of the solution.
This is because the only change that would occur is to replace $q_0(z,w,t)$ by $q_\nu(z,w,t)$.
We can follow exactly the same steps explained in $\mathsection \ref{sec:results}$ to construct and to study the solution to the stochastic Kimura equation with a constant drift $\nu < 1$.
We expect the same techniques employed in the above to work to remain applicable, albeit more involved to apply.
These technical challenges stem from the fact that computations such as (\ref{eqn:u0}) and (\ref{eqn:fsSqIntDefn}) are special cases of formulas involving hypergeometric functions.

On the other hand, 
it is not clear whether the estimates presented in this work, particularly the integral bounds derived in $\mathsection \ref{sec:results}$, are all sharp.
This is in part due to the spatial integrals over the domain $\R_+$, for which we use estimates for integrands involving local degeneracies $\hat z$ with Bessel functions or hypergeometric functions.
Hence, employing tighter bounds in these steps would yield tighter bounds overall.
In the same vein, it remains to be examined whether the requirement on the order of the degeneracy $\hat z$ of the noise near the boundary can be relaxed for obtaining the estimates and continuity in $\mathsection \ref{sec:moment}$ and $\mathsection \ref{sec:cont}$.
In light of improving sharpness overall in the results, it may be possible recover the same results under less degenerate or even non-degenerate Gaussian noise using more refined estimates.

Finally, in the deterministic setting, the Wright-Fisher equation, which is the two-sided boundary analog of the Kimura equation, can be transformed into a one-sided boundary model through a localization method \cite{chenWF}.
It would be interesting to explore and to adapt the localization method in the stochastic setting, which may allow us to investigate the solution to the stochastic Wright-Fisher equation and compare 
it to the solution of the stochastic Kimura equation studied in this work.

\newpage
\section{Appendix} \label{sec:appendix}
In this section, we collect some inequalities and identities that are used in $\mathsection \ref{sec:results}$.
We omit all the proofs.

\begin{proposition}{(Moment Estimate of Gaussian Distribution)}
    \label{prop:gaussianMomentBd}
    For every $p > -1$ there exists a constant $C_p > 0$ such that
    \begin{equation}
        \label{eqn:gaussianMomentBd}
        \frac{1}{\sqrt{2\pi\sigma^2}} \int_\R |u|^p e^{-\frac{(u-m)^2}{2\sigma^2}} \mathrm du \leq C_p(m^p + \sigma^p).
    \end{equation}
\end{proposition}

\begin{proposition}{(Young's Convolution Inequality, Theorem 3.9.4 in \cite{Bogachev2007})}
    \label{prop:modifiedYoung}
    Consider some sets $A, B, C \subseteq \R$ and locally integrable functions $f,g,h : \R \to \R$.
    Denote $D = \set{(x,y) \in B \times C | x-y \in A}$, then for all $p, q, r \geq 1$ such that
    \[
        \frac{1}{p} + \frac{1}{q} + \frac{1}{r} = 2,
    \]
    \begin{equation}
        \label{eqn:modifiedYoung}
        \abs{\int_D f(x-y)g(x)h(y)\mathrm dx \mathrm dy} \leq \p{\int_A |f|^r}^{\frac{1}{r}}\p{\int_B |g|^p}^{\frac{1}{p}}\p{\int_C |h|^q}^{\frac{1}{q}}.
    \end{equation}
\end{proposition}

Next, we present various special functions.
For more details, see \cite{watson}.
\newline

\noindent
\emph{Bessel Functions}. 
Let $\alpha \in \mathbb C$, the \emph{Bessel function of the first kind} $J_\alpha$ are defined as the solution to 
\[
    x^2 \dv[2]{y}{x} + x \dv{y}{x} + (x^2 - \alpha^2)y = 0,
\]
which is non-singular at the origin.
They have various representations, for example, 
\[
    J_\alpha(x) = \sum_{m=0}^\infty \frac{(-1)^m}{m! \Gamma(m+\alpha+1)} \p{\frac{x}{2}}^{2m+\alpha}.
\]
The \emph{modified Bessel function of the first kind} $I_\alpha$ is given by 
\begin{equation}
    \label{eqn:BesselIDefn}
    I_\alpha(x) = i^{-\alpha} J_\alpha(ix) = \sum_{m=0}^\infty \frac{1}{m!\Gamma(m + \alpha + 1)}\p{\frac{x}{2}}^{2m+\alpha}.
\end{equation}
They also have various representations and recurrence relations, for example for $n \in \N$,
\[
    \dv{}{x}\b{x^{-n} I_n(x)} = x^{-n} I_{n+1}(x).
\]
In particular for $n = 0$, $I_0'(x) = I_1(x)$.
\newline

\noindent
\emph{Hypergeometric Functions.}
For each $n \geq 0$ and $a \in \R$, we define the \emph{Pochhammer symbol} as
\[
    (a)_n :=
    \begin{cases}
        1 &; \text{ if } n = 0\\
        a(a+1)(a+2)\cdots(a+n-1) &; \text{ if } n \geq 1 
    \end{cases}.
\]
Given $q \geq p \geq 0$, and $a_1, \ldots, a_p, b_1, \ldots, b_q \in \R$, if none of $b_1, \ldots, b_q$ are non-positive integers, then for $x \in \C$, the \emph{generalized Hypergeometric function} ${}_pF_q$ is defined as
\[
    {}_pF_q
    \left[
    \hspace{-1mm}
    \begin{array}{cc}
        a_1, \ldots, a_p  \\
        b_1, \ldots, b_q 
    \end{array} \hspace{-2mm}
    \mathrel{\Big|} x
    \right] := \sum_{n=0}^\infty \frac{(a_1)_n\cdots(a_p)_n}{(b_1)_n \cdots(b_q)_n} \frac{x^n}{n!}.
\]

Here are two technical results involving these functions:
\begin{proposition}{(\href{https://dlmf.nist.gov/13.6}{(13.6.5)} in \cite{NIST:DLMF})}
    \label{prop:IncompleteGammaPFQ}
    Let $a \in \N$, then
    \begin{equation}
        \label{eqn:IncompleteGammaPFQ}
        {}_1F_1
        \left[
        \hspace{-1mm}
        \begin{array}{cc}
            1  \\
            a + 1 
        \end{array} \hspace{-2mm}
        \mathrel{\Big|} z
        \right]
        = a e^z z^{-a} \p{\Gamma(a) - \Gamma(a, z)}
    \end{equation}
\end{proposition}

\begin{proposition}{(Main Result in \cite{VolkmerWood})}
    \label{prop:2F2AsymptoticExpansion}
    Suppose that neither $c$ nor $d$ are non-positive integers, then $_2F_2$ has the following asymptotic expansion:
    \begin{equation}
        \label{eqn:2F2AsymptoticExpansion}
        \frac{\Gamma(a)\Gamma(b)}{\Gamma(c)\Gamma(d)} 
        {}_2F_2
        \left[
        \hspace{-1mm}
        \begin{array}{cc}
            a, b  \\
            c, d 
        \end{array} \hspace{-2mm}
        \mathrel{\Big|} z
        \right]
        \sim e^z z^\nu \sum_{k=0}^\infty c_k z^{-k},
    \end{equation}
    as $z \to \infty$, where $\nu = a + b - c - d$.
    If $\nu < 0$, then for every $0 \leq k \leq -\nu$,
    \begin{equation}
        \label{eqn:2F2UnifBd}
        \sup_{z > 0} e^{-z}z^k 
        {}_2F_2
        \left[
        \hspace{-1mm}
        \begin{array}{cc}
            a, b  \\
            c, d 
        \end{array} \hspace{-2mm}
        \mathrel{\Big|} z
        \right] < \infty
    \end{equation}
\end{proposition}

We are interested in the connections between Bessel functions and Hypergeometric functions in the form of certain integral quantities.
These are given in the following results from $\mathsection$13.3 of \cite{watson}.
\begin{proposition}
    \label{prop:watsonu0Int}
    Let $\Re(\nu + \mu) > 0$ then
    \begin{align}
        \raisetag{-7mm} \hspace{8mm} \label{eqn:watsonu0Int}
        \int_0^\infty \besselJ{\nu}{at}\exp(-p^2t^2)t^{\mu-1}\mathrm dt = \frac{\Gamma\p{\frac{\nu}{2}+\frac{\mu}{2}}\p{\frac{a}{2p}}^\nu}{2p^{\mu}\Gamma(\nu+1)}
        {}_1F_1
        \left[
        \hspace{-1mm}
        \begin{array}{cc}
            \textstyle \frac{(\nu+\mu)}{2}  \\
            \nu+1 
        \end{array} \hspace{-2mm}
        \mathrel{\Big|} -\frac{a^2}{4p^2}
        \right],
    \end{align}
\end{proposition} 
in which we are interested in the special case $\nu = \mu = 1$,
\[
    \int_0^\infty J_1(a t)\exp\p{-p^2t^2}\mathrm dt = \frac{a}{4p^2} {}_1F_1
        \left[
        \hspace{-1mm}
        \begin{array}{cc}
            1  \\
            2 
        \end{array} \hspace{-2mm}
        \mathrel{\Big|} -\frac{a^2}{4p^2}
        \right] = \frac{a}{4p^2} \frac{-1 + e^{-\frac{a^2}{4p^2}}}{\displaystyle \p{-\frac{a^2}{4p^2}}} = \frac{1 - e^{-\frac{a^2}{4p^2}}}{a}.
\]
In addition,
\begin{proposition}
    \label{prop:WeberIntegral}
    If $\Re(\nu) > -1/2$, $\Re(2\nu + \mu) > 0$, $|\arg p| < \pi/4$, setting $ \bar \omega = \sqrt{a^2 + b^2 - 2ab \cos\phi}$, then
    \begin{align*}
        \int_{\R_+} &\exp(-p^2u^2)\besselJ{\nu}{au}\besselJ{\nu}{bu}u^{\mu-1}\mathrm du\\   
        &= \frac{\displaystyle \Gamma\p{\nu + \frac{1}{2}\mu}}{\displaystyle2\pi p^{\mu}\Gamma\p{2\nu+1}} \p{\frac{ab}{p^2}}^\nu \int_0^\pi \exp\p{-\frac{\bar\omega^2}{4p^2}}{}_1F_1
        \left[
        \hspace{-1mm}
        \begin{array}{cc}
            \displaystyle 1 - \frac{\mu}{2}  \\
            \nu+1 
        \end{array} \hspace{-2mm}
        \mathrel{\Bigg|} \frac{\bar\omega^2}{4p^2}
        \right]\sin^{2\nu}\phi\mathrm d\phi,
    \end{align*}
\end{proposition} 
in which we are interested in the special case $\nu = 1$, $\mu = 0$, and $a = b > 0$.
Thus $\bar \omega = \sqrt{2}a\sqrt{1 - \cos{\phi}}$ and
\[
    \frac{\bar \omega^2}{4p^2} = \frac{a^2(1 - \cos\phi)}{2p^2},
\]
from which it follows
\[
    \int_0^\infty \exp(-p^2u^2)\besselJ{1}{au}^2u^{-1}\mathrm du = \frac{1}{2}\p{1 - e^{-\frac{a^2}{2p^2}}\besselI{0}{-\frac{a^2}{2p^2}} + e^{-\frac{a^2}{2p^2}}\besselI{1}{-\frac{a^2}{2p^2}}}.
\]

\newpage
\section{References} \label{sec:refs}
\printbibliography[heading=none]

\end{document}